\documentclass[12pt,a4paper]{amsbook}
\usepackage[left=3cm,right=3cm,top=3cm,bottom=3cm]{geometry}

\usepackage{amsmath, amssymb, amsthm}

\usepackage{hyperref}
\usepackage{tikz,tikz-cd}

\usepackage{newtxtext}
\usepackage{newtxmath}

\usepackage[bb=boondox,scr=boondox]{mathalfa}

\usepackage{bm}

\usepackage[backend=biber]{biblatex}
\newtheorem{thm}{Theorem}[chapter]
\newtheorem{prop}[thm]{Proposition}

\newtheorem{corol}[thm]{Corollary}
\theoremstyle{remark} \newtheorem{rmk}[thm]{Remark}
\theoremstyle{definition} \newtheorem{defn}[thm]{Definition}
\theoremstyle{definition} \newtheorem{ex}[thm]{Example}
\theoremstyle{definition} \newtheorem{ej}{Exercise}

\let\emph\relax
\DeclareTextFontCommand{\emph}{\it\bfseries}

\newcommand{\git}{\mathbin{/\mkern-6mu/}}
\newcommand{\hkgit}{\mathbin{/\mkern-6mu/\mkern-6mu/\mkern-6mu/}}

\newcommand{\C}{\mathbb C}

\newcommand{\RR}{\mathbb R}
\newcommand{\Z}{\mathbb Z}

\newcommand{\bbA}{\mathbb A}

\newcommand{\bbR}{\mathbb R}
\newcommand{\bbN}{\mathbb N}

\newcommand{\bbB}{\mathbb B}

\newcommand{\bbE}{\mathbb E}
\newcommand{\bbP}{\mathbb P}

\newcommand{\cA}{\mathcal A}
\newcommand{\cB}{\mathcal B}
\newcommand{\cC}{\mathcal C}
\newcommand{\cD}{\mathcal D}
\newcommand{\cE}{\mathcal E}
\newcommand{\cF}{\mathcal F}
\newcommand{\cH}{\mathcal H}

\newcommand{\cG}{\mathcal G}

\newcommand{\cL}{\mathcal L}
\newcommand{\cM}{\mathcal M}
\newcommand{\cN}{\mathcal N}

\newcommand{\cP}{\mathcal P}

\newcommand{\cS}{\mathcal S}

\newcommand{\cX}{\mathcal X}

\newcommand{\fU}{\mathfrak U}

\newcommand{\fg}{\mathfrak{g}}

\newcommand{\fk}{\mathfrak k}
\newcommand{\fu}{\mathfrak u}

\newcommand{\sO}{\mathscr O}

\newcommand{\sE}{\mathscr E}
\newcommand{\sF}{\mathscr F}
\newcommand{\sL}{\mathscr L}
\newcommand{\sM}{\mathscr M}
\newcommand{\sG}{\mathscr G}

\newcommand{\sV}{\mathscr V}

\newcommand{\Bt}{\mathrm{B}}
\newcommand{\dR}{\mathrm{dR}}
\newcommand{\Dol}{\mathrm{Dol}}
\newcommand{\Hod}{\mathrm{Hod}}

\newcommand{\R}{\mathrm{R}}

\newcommand{\Ad}{\mathrm{Ad}}
\newcommand{\ad}{\mathrm{ad}}

\DeclareMathOperator{\Aut}{Aut}

\DeclareMathOperator{\Hom}{Hom}

\DeclareMathOperator{\Pic}{Pic}
\DeclareMathOperator{\Jac}{Jac}

\DeclareMathOperator{\End}{End}

\DeclareMathOperator{\Spec}{Spec}
\DeclareMathOperator{\Proj}{Proj}
\DeclareMathOperator{\tr}{tr}

\DeclareMathOperator{\Bun}{Bun}

\DeclareMathOperator{\Nm}{Nm}

\DeclareMathOperator{\GL}{GL}
\DeclareMathOperator{\SL}{SL}
\DeclareMathOperator{\PGL}{PGL}

\DeclareMathOperator{\U}{U}
\DeclareMathOperator{\SU}{SU}
\DeclareMathOperator{\PU}{PU}
\DeclareMathOperator{\id}{id}

\DeclareMathOperator{\delbar}{\bar{\partial}}
\DeclareMathOperator{\diag}{diag}
\DeclareMathOperator{\Lie}{Lie}
\DeclareMathOperator{\rk}{rk}
\DeclareMathOperator{\pr}{pr}

\title{Non-Abelian Hodge Theory and Moduli Spaces of Higgs Bundles}
\author{Guillermo Gallego}

\address{Institut für Mathematik \\ Freie Universität Berlin \\ 14195 Berlin, Germany}
\email{\normalfont\href{mailto:guillermo.gallego.sanchez@fu-berlin.de}{guillermo.gallego.sanchez@fu-berlin.de}}
\urladdr{\normalfont\href{https://guillegallego.xyz}{https://guillegallego.xyz}}

\makeatletter
\hypersetup{
    pdfauthor={\authors},
    pdftitle={\@title},
    colorlinks,
    linkcolor=[rgb]{0.2,0.2,0.6},
    citecolor=[rgb]{0.2,0.6,0.2},
    urlcolor=[rgb]{0.6,0.2,0.2}}
\makeatother

\begin{document}
\begin{titlepage}
\setcounter{page}{1}
\thispagestyle{empty}
\begin{center}
\Huge \bfseries Non-Abelian Hodge Theory and Moduli Spaces of Higgs Bundles\par\vspace{2em}
\Large Guillermo Gallego \par\bigskip
\end{center}

\vspace{7em}

\begin{center}
\textbf{Abstract}
\end{center}
This paper provides an introduction to non-abelian Hodge theory and moduli spaces of Higgs bundles on compact Riemann surfaces. We develop the moduli theory of vector bundles and Higgs bundles, establish the main correspondences of non-abelian Hodge theory, and interpret them through the hyperkähler structure on the Hitchin moduli space. We study the Hitchin fibration and its geometric properties, including SYZ mirror symmetry and topological mirror symmetry for type $\mathsf{A}$ Hitchin systems. As an illustration, we compute the Poincaré polynomial of the rank 2 moduli space and verify topological mirror symmetry in this case. 

\vfill
\hrule \vspace{1em}
\begin{center}
\textsc{Institut für Mathematik, Freie Universität Berlin, 14195 Berlin, Germany} \vspace{1em}
\textit{Email:} \normalfont\href{mailto:guillermo.gallego.sanchez@fu-berlin.de}{guillermo.gallego.sanchez@fu-berlin.de} \hspace{1em}
\textit{URL:} \normalfont\href{https://guillegallego.xyz}{https://guillegallego.xyz}
\end{center}
\end{titlepage}

\tableofcontents
\chapter{Introduction}

\section{Abelian Hodge theory}
Hodge theory is about finding representatives of the cohomology classes of a manifold. In the case of a compact complex manifold, Hodge theory gives representatives of the Dolbeault cohomology classes. Moreover, if the compact complex manifold is also Kähler, there is the Hodge decomposition theorem, which splits the (smooth, $\C$-valued) de Rham cohomology of the manifold in terms of the Dolbeault cohomology. We are particularly interested in the case of a genus $g$ compact Riemann surface $X$. In that case, the Hodge decomposition theorem yields an isomorphism
\begin{equation*}
H^1_{\dR}(X,\C) \cong H^{1,0}(X) \oplus H^{0,1}(X)
\end{equation*}
and $H^{1,0}(X)$ and $H^{0,1}(X)$ are exchanged by conjugation. Therefore, one can identify
\begin{equation} \label{eq:abelian-hodge}
H^{1,0}(X) \cong H^1_{\dR}(X,\bbR) \text{ and } H^{0,1}(X)\cong H^1_{\dR}(X,i\bbR).
\end{equation}
The de Rham theorem matches the de Rham cohomology space $H^1_{\dR}(X,\bbR)$ with the singular cohomology $H^1(X,\bbR)$, which is in turn canonically isomorphic to $\Hom(\pi_1(X,x_0),\bbR)\cong \bbR^{2g}$, for any point $x_0\in X$. We can also take coefficients in the circle group
\begin{equation*}
\U(1)= \left\{ z\in \C : |z|=1 \right\}.
\end{equation*}
Note that there is the short exact sequence
\begin{center}
  \begin{tikzcd}
0 \rar & 2\pi i \Z \rar & i\bbR \rar{\exp} \rar & \U(1) \rar & 0,
  \end{tikzcd}
\end{center}
which allows us to give a diffeomorphism
\begin{equation} \label{eq:abelianNS}
\Hom(\pi_1(X,x_0),\U(1)) \cong H^1(X,i\bbR)/H^1(X,2\pi i \Z) \cong H^{0,1}(X)/H^1(X,2\pi i \Z).
\end{equation}
The quotient $\Jac(X):= H^{0,1}(X)/H^1(X,2\pi i \Z) = H^1(X,\sO_X)/H^1(X,2\pi i \Z)$ is an abelian variety of complex dimension $g$, called the Jacobian of $X$. On the other hand, note that $\Hom(\pi_1(X,x_0),\U(1))=\U(1)^{2g}$ is indeed a $2g$-dimensional compact real torus. The space $\Jac(X)$ is also the neutral connected component of the cohomology group $H^1(X,\sO_X^*)$, and thus parametrizes topologically trivial holomorphic line bundles. More generally, if we take coefficients on $\C^*$, we obtain a diffeomorphism
\begin{equation*}
\Hom(\pi_1(X,x_0),\C^*) \cong H^1(X,\C)/H^1(X,2\pi i \Z) \cong H^{1,0}(X) \times \Jac(X).
\end{equation*}
An element of $\Jac(X) \times H^{1,0}(X)$ is an isomorphism class of a pair $(\sL,\varphi)$ formed by a topologically trivial holomorphic vector bundle $\sL$ on $X$ and a holomorphic $1$-form $\varphi \in H^{1,0}(X)$. Such a pair $(\sL,\varphi)$ is an example of a (rank $1$, topologically trivial) Higgs bundle. We remark the fact that $\Hom(\pi_1(X,x_0),\C^*)=(\C^{2g})^*$ is a complex torus of complex dimension $2g$, while on the right hand side we get a product of an abelian variety of complex dimension $g$ with a complex vector space of complex dimension $g$. These are two different complex manifolds, which are diffeomorphic, but notably they are not complex analytically isomorphic.

The ``representatives'' of the cohomology classes that we mentioned at the beginning of the section are the harmonic forms. This means that the isomorphisms \eqref{eq:abelian-hodge} from where the identifications of the spaces of (topologically trivial) line bundles and of rank $1$ Higgs bundles with the corresponding representation spaces followed are mediated by solving the Laplace equation on the Riemann surface $X$. We are in the setting of abelian ($\U(1)$, or $\C^*$) gauge theory (aka electromagnetism).

\section{Non-abelian Hodge theory: A short history of the topic}
One could date back the genesis of non-abelian Hodge theory to 1965 with the publication of the celebrated result of M.S. Narasimhan and C.S. Seshadri \cite{NarasimhanSeshadri}
\begin{equation*}
\Hom(\pi_1(X,x_0),\U(r))/\U(r) \cong \cN_{r,0} ,
\end{equation*}
giving a homeomorphism between the moduli space of rank $r$ unitary representations of the fundamental group of $X$ with the moduli space $\cN_{r,0}$ classifying (semistable, topologically trivial) holomorphic rank $r$ vector bundles on $X$. For $r=1$, one recovers the statement \eqref{eq:abelianNS}. The structure of unstable bundles is more complicated, but notably Günter Harder and Narasimhan \cite{HarderNarasimhan} introduced an invariant that one can associate to any bundle, now called the Harder--Narasimhan filtration. Moreover, Harder and Narasimhan studied the global topology of the moduli spaces of holomorphic vector bundles in terms of the Weil conjectures.

Almost two decades later, in the early 1980s, Michael Atiyah and Raoul Bott \cite{AtiyahBott} reinterpreted the theorem of Narasimhan--Seshadri in terms of gauge theory. More precisely, they studied the Yang--Mills equations for vector bundles with unitary connections on $X$. It turns out that on a Riemann surface the Yang--Mills equations are equivalent to the equations for a flat connection (or, more generally, for a connection of constant central curvature). They then reformulated the Narasihman--Seshadri theorem as the existence of a (unique) Hermitian metric, such that its Chern connection is of constant central curvature in any polystable holomorphic vector bundle on $X$. Notably, Atiyah and Bott also reproved the formulas of Harder and Narasimhan for the Betti numbers of the moduli space of bundles in terms of gauge theory.

Around that time, Simon Donaldson (who was then a PhD student of Atiyah and Nigel Hitchin) gave a new proof of the theorem of Narasimhan--Seshadri \cite{DonaldsonNS} in these gauge-theoretical terms. His proof relied strongly on deep analytical results about connections which had been proven by Karen Uhlenbeck \cite{Uhlenbeck}. Donaldson's proof would later serve as a ``model'' for many other correspondences of this sort, between objects which satisfy some algebro-geometric stability conditions, and solutions to some gauge-theoretical equation. Correspondences of this kind are generally called ``Hitchin--Kobayashi'' correspondences. Notable examples are the Donaldson--Uhlenbeck--Yau theorem \cite{Donaldson, UhlenbeckYau} (the higher dimensional analogue of Narasimhan--Seshadri) and the Hitchin--Simpson \cite{Hitchin_SelfDuality,Simpson_Variations} correspondence for Higgs bundles.

Higgs bundles and the Hitchin fibration were introduced by Hitchin in his two seminal papers of 1987 \cite{Hitchin_SelfDuality,Hitchin_Systems}, although the term was coined one year later by Carlos Simpson\footnote{Notably, Hitchin talks in \cite{Hitchin_SelfDuality} about the ``Higgs field''. Similar terminology had already been used in the monopole literature by him and others (see, for example \cite{Hitchin_MonopolesGeodesics}), to refer to the additional section appearing in some gauge equations, specially when these are obtained by dimensional reduction of the Yang--Mills equations, as is the case for the Hitchin equations and for the Bogomolny equations governing monopoles. There is no direct relation with the widely popular and talked-about ``Higgs boson'', although see Witten's comment in \cite[Remark 2.1]{Witten_Hitchin}.} \cite{Simpson_Variations}. In \cite{Hitchin_SelfDuality}, Hitchin defined the stability condition for $\SL_2(\C)$-Higgs bundles, and proved that every polystable $\SL_2(\C)$-Higgs bundle admits a Hermitian metric such that the Chern connection solves the (now called) Hitchin equations. The further relation with representations of the fundamental group would then follow from a result of Donaldson in \cite{Donaldson_Harmonic}, which is the paper following \cite{Hitchin_SelfDuality} in the journal it was published. This then established non-abelian Hodge theory for $\SL_2(\C)$,
\begin{equation*}
\Hom(\pi_1(X,x_0),\SL_2(\C))\git \SL_2(\C) \simeq \check{\cM}_{2,0} ,
\end{equation*}
relating the $\SL_2(\C)$-character variety of $X$ with the moduli space of polystable $\SL_2(\C)$-Higgs bundles on $X$. Shortly after, Simpson \cite{Simpson_Variations} gave a wide generalization of Hitchin's results to arbitrary rank and to arbitrary dimension of the base (i.e. to arbitrary compact Kähler manifolds). The existence of harmonic metrics on flat bundles was proved in general by Kevin Corlette in 1988 \cite{Corlette}.

The global topology of the moduli space of Higgs bundles was already studied by Hitchin in \cite{Hitchin_SelfDuality}, who computed the Poincaré polynomial of the moduli space $\check{\cM}_{2,1}$ of degree $1$ (twisted) $\SL_2(\C)$-Higgs bundles by using a Morse-theoretical method. Hitchin's argument was later generalized to rank $3$ in the thesis of Peter Gothen \cite{Gothen}. The Betti numbers of the moduli space of Higgs bundles are not known for general rank, but there are conjectural formulas by Tamás Hausel and Fernando Rodríguez-Villegas \cite{HauselRodriguezVillegas}. The Morse-theoretical arguments of Hitchin were extended to arbitrary rank in a paper by Oscar García-Prada, Jochen Heinloth and Alexander Schmitt \cite{OscarHeinlothSchmitt}, using a motivic point of view. In particular, they gave explicit formulas in the case of rank $4$ and odd degree, and verified the conjecture of Hausel and Rodríguez-Villegas up to genus 21 using a computer algebra system.

The original formulation of the Hitchin fibration in \cite{Hitchin_Systems} is in terms of $G$-Higgs bundles, for general reductive $G$, and Hitchin already hinted in there towards the duality of Hitchin fibres for Langlands dual groups\footnote{See the remark at the top of page 109, where Hitchin mentions that the generic fibres for the type $\mathsf{B}$ and type $\mathsf{C}$ Hitchin fibrations are \textit{the same} Prym variety. Notably, in \cite{HitchinG2}, Hitchin mentions that Thaddeus later pointed out a mistake on his original paper, the resolution of which would yield that the fibres for types $\mathsf{B}$ and $\mathsf{C}$ are indeed \textit{dual}, and not the same as stated originally.}. Such duality was later, in 2002, explicitly stated in the case of $\SL_r(\C)$ and $\PGL_r(\C)$ by Hausel and Michael Thaddeus \cite{HauselThaddeus}, who interpreted it as SYZ mirror symmetry for the de Rham moduli space. This led them to conjecture their ``topological mirror test'' concerning a certain agreement of the (stringy) $E$-polynomials of the moduli spaces of $\SL_r(\C)$ and $\PGL_r(\C)$-Higgs bundles. They proved their conjectures for $r=2$ and $3$, using the Morse-theoretical study of the cohomology of the moduli space by Hitchin and Gothen. The general statement was proven in 2017 by Michael Groechenig, Dimitri Wyss and Paul Ziegler \cite{GWZ_Mirror}, using a technique of $p$-adic integration. A different proof appeared in 2020 by Davesh Maulik and Junliang Shen \cite{MaulikShen_Endoscopic}. Another major highlight of the study of the global topology of the Hitchin system is the proof by Ngô Bao Chau in 2008 \cite{Ngo} of the Fundamental Lemma of Langlands--Shelstad, which won him the Fields Medal in 2010. A different proof of Ngô's results using $p$-adic integration was given in 2018 by Groechenig, Wyss and Ziegler \cite{GWZ_Ngo}.

The description of the generic Hitchin fibres for general $G$ was developed in 2000 by Ron Donagi and Dennis Gaitsgory \cite{DonagiGaitsgory}, and the duality statement for general $G$ was formulated and proven by Donagi and Tony Pantev \cite{DonagiPantev} in 2006. They interpreted this statement as a ``classical limit'' of the ``geometric Langlands conjecture''. Notably, the geometric Langlands conjecture was first formulated by Alexander Beilinson and Vladimir Drinfeld as a ``quantization of Hitchin's system'' \cite{BeilinsonDrinfeld} in a preprint available since the mid 1990s. Also in 2006, Anton Kapustin and Edward Witten \cite{KapustinWitten} gave a physical interpretation of the geometric Langlands conjecture as ``hyperkähler-enhanced'' homological mirror symmetry on Hitchin's moduli space. A proof of (a refined and corrected version of) the geometric Langlands conjecture as a collaborative project of 9 authors (D. Arinkin, D. Beraldo, J. Campbell, L. Chen, J. Faergeman, D. Gaitsgory, K. Lin, S. Raskin and N. Rozenblyum) has been very recently made available \cite{ProofGLC}.

\section{Recent developments}
In recent years, research centered around the moduli space of Higgs bundles remains very active. We shall mention just a few lines of current research and subsequent ramifications, although there are many more.

\subsection*{The $P=W$ theorem and the cohomology of the moduli space} As we have already mentioned, the study of global geometric properties of the Hitchin fibration, and in particular the study of its cohomology ring, remains widely open. A major breakthrough has been made recently with the proof of the ``$P=W$ property'', conjectured by Mark de Cataldo, Hausel and Luca Migliorini \cite{PWConjecture}. The statement is about the agreement of two different filtrations on the cohomology of the moduli space of Higgs bundles: $W$ stands for the weight filtration associated with a complex affine variety, which comes from the fact that the moduli space is homeomorphic to the character variety; the other filtration is $P$, the perverse filtration associated with a proper map, which comes from the fact that the Hitchin fibration is proper. Two different proofs appeared almost at the same time, in September 2022, one by Maulik and Shen \cite{MaulikShen_PW}, using a support theorem and the ``global Springer theory'' of Zhiwei Yun \cite{Yun}, and another one by Hausel, Anton Mellit, Alexandre Minets and Olivier Schiffmann \cite{H2_PW}, which uses an action of the Heisenberg algebra.

\subsection*{Higgs bundles for real groups and higher rank Teichmüller theory} As many other topics related with Higgs bundles, the notion of a Higgs bundle with real structure group dates back to Hitchin's original paper \cite{Hitchin_SelfDuality}. There, he showed that a component of the moduli space of $\SL_2(\bbR)$-Higgs bundles can be identified with the Teichmüller space parametrizing hyperbolic metrics on $X$. In a later paper \cite{Hitchin_Classes}, Hitchin generalized this observation to higher rank, and more generally to split real groups $G_\bbR$. More precisely, he shows the existence of a component (now called the Hitchin component) inside the moduli space of $G_\bbR$-Higgs bundles which is a ball, and can be understood as a ``higher rank analogue'' of Teichmüller space. Nonabelian Hodge theory has been in fact generalized to consider $G_\bbR$-Higgs bundles, where $G_\bbR$ is any real reductive group, in the work of García-Prada with Gothen and Ignasi Mundet i Riera \cite{OscarGothenMundet}.

\subsection*{Understanding geometric Langlands duality} Even though a proof of the geometric Langlands program is now available, identifying some objects on both sides of the correspondence and constructing one from the other can be a complicated problem. One of the features of the ``hyperkähler enhanced'' mirror symmetry formulated by Kapustin and Witten, is that it predicts the duality of certain objects supported on certain special submanifolds of the moduli space of Higgs bundles, called branes. For example, branes supported on hyperkähler submanifolds are called ``BBB branes'', while branes supported on holomorphic Lagrangian submanifolds are called ``BAA branes''. These two types of branes are conjecturally dual. It is then an interesting problem to find holomorphic Lagrangian submanifolds of the moduli space and hyperkähler submanifolds of the moduli space with dual group which are dual to these. Moduli spaces of Higgs bundles with structure group a real form of a complex group are natural sources of Lagrangian submanifolds. These are in fact particular cases of ``Gaiotto Lagrangians'' \cite{Gaiotto,GinzburgRozenblyum} which are holomorphic Lagrangian submanifolds of the moduli space of Higgs bundles induced by Hamiltonian $G$-spaces. These can also be interpreted as ``boundary conditions'' in the ``relative Langlands program'' of David Ben-Zvi, Yiannis Sakellaridis and Akshay Venkatesh \cite{BZSV}. Some other selected references in this topic are \cite{BaragliaSchaposnik,HellerSchaposnik,BiswasOscarJacques,HauselEnhanced,FGOP,FrancoHanson,HameisterLuoMorrissey,ChenHsiaoYang}.

\subsection*{Generalizations of the Hitchin fibration} The Hitchin fibration is modelled over the Chevalley restriction map $\mathfrak{g}\rightarrow \mathfrak{g}\git G$. There are some other fibrations also arising from maps of the form $M\rightarrow M\git G$, which appear in the study of Higgs bundles and related topics. For example, the Hitchin fibration for higher dimensional manifolds, originally defined by Simpson \cite{SimpsonII}, arises from the $G\times \mathrm{GL}_{n}$-action on the commuting variety $\mathfrak{C}_n(\mathfrak{g})$ of $n$ elements on $\mathfrak{g}$. The case of $G=\GL_r$ is already quite complicated, and has been studied by Ngô and Tsao-Hsien Chen \cite{NgoChen} who obtain some results in the case of $n=2$. A similar kind of fibration, but over Riemann surfaces, associated with Higgs pairs twisted by a vector bundle, was studied in \cite{twistedpaper}. Another example is the Hitchin fibration for Higgs bundles with real structure group, considered in the work of García-Prada and Ana Peón-Nieto \cite{AnaOscar} and studied in much detail in the work of Thomas Hameister and Benedict Morrissey \cite{HameisterMorrissey}. In particular, they develop a framework to deal with several types of generalized Hitchin fibrations, called the ``regular quotient'', inspired in some constructions of García-Prada and Peón-Nieto \cite{AnaOscar}, which were originally based in the work of Abramovich--Olsson--Vistoli \cite[Appendix A]{AOV}. It is also worth mentioning ``multiplicative'' Hitchin fibrations, which arise from the Steinberg map $G\rightarrow G\git G$. These have been used by Griffin Wang to provide a more direct proof of the Fundamental Lemma \cite{Griffin}. A ``Donagi--Pantev-style'' duality for multiplicative Hitchin fibrations has been recently studied in \cite{dualitypaper}.

\section{About this paper}
Our aim in this paper is to provide a self-contained introduction to the rich theory of Higgs bundles and their moduli spaces, suitable for readers with a background in algebraic geometry or differential geometry. We begin with fundamental notions—vector bundles, connections, and holomorphic structures—building up to the sophisticated machinery of non-abelian Hodge theory. Along the way, we develop the necessary tools from geometric invariant theory to construct moduli spaces and establish stability conditions for both vector bundles and Higgs bundles.

A central theme is the interplay between algebraic, differential-geometric, and topological perspectives. We show how the Hitchin equations give rise to a hyperkähler structure on the moduli space, unifying three seemingly different moduli spaces: the Betti, de Rham, and Dolbeault moduli spaces. This hyperkähler geometry leads naturally to the Hitchin fibration, whose properties reveal deep connections to mirror symmetry and the Langlands program.

In the final chapter, we work out some explicit computations in rank 2. More precisely, following Hitchin's original approach, we calculate the Poincaré polynomial using Morse-theoretic methods and verify the topological mirror symmetry conjecture of Hausel and Thaddeus. These calculations illustrate the powerful techniques available for studying the global topology of Higgs bundle moduli spaces.

\subsection*{A note on proofs and technical details} We emphasize that this paper is primarily expository in nature. Although the definitions and statements are meant to be precise, we do not provide complete proofs of the main theorems of non-abelian Hodge theory—namely, the Narasimhan-Seshadri, Corlette-Donaldson, and Hitchin-Simpson theorems. These results require substantial analytical machinery, including elliptic regularity theory, estimates for non-linear PDEs, and Uhlenbeck's compactness theorems, which would take us too far afield. Instead, we explain the geometric content of these theorems and their role in the theory. Similarly, our treatment of geometric invariant theory sketches the main ideas while omitting some technical details. In Chapter \ref{sec:global}, however, we provide complete calculations for the Poincaré polynomials and verification of topological mirror symmetry in rank 2, as these serve to illustrate the general theory. Throughout, we provide references to the literature where interested readers can find rigorous proofs and further details.

\subsection*{Exercises and examples} Throughout the paper, we have included a substantial collection of exercises that serve multiple purposes: some verify claims made in the text, others develop important examples and special cases, and several extend the theory in directions we do not pursue in the main exposition. We encourage readers to work through these exercises, as they provide valuable intuition and complement the theoretical development. Examples include explicit computations with vector bundles on $\bbP^1$, constructions of special Higgs bundles, and verifications of key formulas.

\section{Organization of the paper}
In Chapter \ref{sec:primer}, we establish the foundational theory of vector bundles and connections on compact Riemann surfaces. We develop the differential-geometric perspective on holomorphic structures via Dolbeault operators, introduce flat connections and their relationship to representations of the fundamental group via the Riemann-Hilbert correspondence, and discuss Hermitian metrics and the Chern connection.

Chapter \ref{sec:classification} addresses the classification problem for vector bundles from multiple perspectives. We begin with the topological classification in terms of rank and degree, then introduce the Jacobian variety and its role in parametrizing holomorphic line bundles. The main focus is the construction of the moduli space of semistable vector bundles: we provide a review of geometric invariant theory and sketch both the algebraic construction as a GIT quotient and the differential-geometric construction as a symplectic (Kähler) quotient. We conclude by presenting the Narasimhan--Seshadri theorem, which identifies this moduli space with the character variety of unitary representations.

In Chapter \ref{sec:naht}, we develop the full non-abelian Hodge correspondence. We introduce the Betti moduli space (character varieties), the de Rham moduli space (vector bundles with connections), and the Dolbeault moduli space (Higgs bundles), establishing their real-analytic equivalence. The key results are the Donaldson--Corlette and the Hitchin--Simpson theorems on the existence of canonical metrics, which we interpret through the hyperkähler structure on the Hitchin moduli space. We explain how the different complex structures arise from a single hyperkähler metric and introduce the twistor family and $\lambda$-connections.

Chapter \ref{sec:hitchinsystem} is devoted to the Hitchin fibration and its remarkable properties. We define the Hitchin map and explain the spectral correspondence, which describes generic fibers as Prym varieties of spectral curves. We develop the theory for both $\SL_r$ and $\PGL_r$ Higgs bundles, establishing the duality between their Hitchin fibrations. This leads naturally to a discussion of SYZ mirror symmetry and topological mirror symmetry, which predicts a relationship between (stringy) $E$-polynomials of dual moduli spaces. We conclude with remarks on Langlands duality and connections to the geometric Langlands program.

Finally, Chapter \ref{sec:global} presents explicit calculations in rank $2$. We introduce equivariant cohomology and the theory of perfect stratifications, then review the Atiyah--Bott computation of the Poincaré polynomial for the moduli space of vector bundles using the Harder--Narasimhan stratification. We then reproduce Hitchin's calculation for Higgs bundles using the Bialynicki-Birula stratification induced by a $\C^*$-action, and conclude by verifying the topological mirror symmetry conjecture of Hausel--Thaddeus in this case.

\section{Acknowledgements}
This paper originated from a three-hour mini-course I taught in the ``Workshop on character varieties and Higgs bundles'' celebrated in Liberia, Guanacaste on the 4-8 of August 2025. Thus I am grateful to the organizers of that conference, Alexander Schmitt and Ronald Zúñiga-Rojas for giving me that oportunity. I would like to thank Sam Engleman, Cesare Goretti and Alfonso Zamora, for discussions and suggestions about preliminary versions of the paper. I also thank Miguel González for answering some questions about upward flows.

My research is funded by a project of the Deutsche Forschungsgemeinschaft (DFG) with number 524596398, under a postdoctoral contract at the Freie Universität Berlin.

\chapter{A primer on vector bundles and connections} \label{sec:primer}
\section{Vector bundles in different categories} \label{ss:vectorbundles}
Let $X$ be a compact Riemann surface. We can trivialize $X$ by giving a complex atlas: namely, we cover $X$ by a family $\mathfrak{U}$ of open subsets $U\subset X$ with a homeomorphism  $\psi_U:U\rightarrow D_U\subset \mathbb{C}$ with some disk $D_U$ in  $\mathbb{C}$, for each $U\in \mathfrak{U}$, in such a way that the coordinate change functions $\psi_{UV}=\psi_V \circ \psi_U^{-1}: D_U \cap D_V \rightarrow D_U \cap D_V$ are holomorphic.

\begin{defn}
A \emph{vector bundle} $E$  of rank $r$ over $X$ is given by gluing spaces of the form $E_U=U\times \mathbb{C}^r$ using a set of continuous transition functions \[\left\{g_{UV}: U,V \in \fU, U\cap V \neq \varnothing\right\},\] with \[g_{UV}:U\cap V \rightarrow \GL_r(\C)\] that satisfy the (1-)\emph{cocycle condition}
\begin{equation*}
g_{UV} g_{VW} = g_{UW}.
\end{equation*} 

A vector bundle is \emph{smooth}, \emph{holomorphic} or a \emph{local system} if the transition functions $g_{UV}$ are respectively smooth, holomorphic or locally constant.
\end{defn}

\begin{rmk}
Two vector bundles $E$ and $E'$ are isomorphic if and only if their corresponding $1$-cocycles $(g_{UV})$ and $(g'_{UV})$ are cohomologous, meaning that there exists some family $\left\{f_U:U\rightarrow \GL_r(\C): U\in \fU \right\}$ (i.e. a $0$-coboundary) such that
\begin{equation*}
g_{UV}' = f_U g_{UV} f_V^{-1}.
\end{equation*} 
The action of such $0$-coboundaries on the $1$-cocycles determines a groupoid, which naturally classifies vector bundles. The set of isomorphism classes of this groupoid can be understood as a ``non-abelian sheaf cohomology set'' \[H^1(X,\GL_r(\mathscr{A})).\]  Here, $\mathscr{A}$ denotes the sheaf $C^\infty_X$ of smooth functions on $X$, the sheaf $\sO_X$ of holomorphic functions on $X$ or the sheaf $\underline{\C}_X$ of locally constant functions on $X$ depending on whether we are considering smooth bundles, holomorphic bundles or local systems, respectively.
\end{rmk}

\begin{rmk}
Any compact Riemann surface is biholomorphic to the analytification of a smooth complex projective curve. Zariski open subsets determine open subsets in the analytic topology. Thus, if we let $\fU$ correspond to a cover by Zariski open subsets of $X$, we can also consider \emph{algebraic} vector bundles, determined by the condition that the transition functions $g_{UV}$ are regular maps into $\GL_r$ (regarded as a smooth algebraic variety over $\C$). Serre's GAGA theorem \cite{GAGA} implies that the category of algebraic vector bundles over a smooth complex projective curve is equivalent to the category of holomorphic vector bundles over its analytification.
\end{rmk}

\begin{rmk}
Consider a local system $E$ over $X$ determined by a set of locally constant transition functions $\left\{g_{UV}\right\}$. Given any point $x_0\in X$ we can define the \emph{monodromy representation} $\rho_E:\pi_1(X,x_0)\rightarrow \GL_r(\C)$ constructed as follows. If $\sigma$ is a loop on $X$ based at $x_0$, then we can partition the unit interval $t_0=0<t_1<t_2<\dots<t_N=1$ in such a way that, for every $i=1,\dots,N$ there exists some  $U_i \in \fU$ such that $\sigma([t_i,t_{i+1}])\subset U_i$. If we call $x_i=\sigma(t_i)$ and $g_{ij}=g_{U_iU_j}$, we can define
\begin{equation*}
\rho_E(\sigma) = g_{12}(x_1) g_{23}(x_2) \dots g_{N1}(x_N).
\end{equation*} 
\end{rmk}

\begin{ej} Verify the following statements about the remark above.
\textit{	
\begin{enumerate}
	\item There exists such a partition of the unit interval. {\rm \textbf{Hint}: Use ``Lebesgue's Number Lemma", \cite[p. 179]{Munkres}}.
	\item The map $\rho_E(\sigma)$ does not depend on the choice of ``Lebesgue partition".
	\item The map $\rho_E(\sigma)$ does not depend on the chosen representative of the homotopy class $[\sigma]$.
\end{enumerate}	}
\end{ej}

\subsection*{Some notations and basic rudiments of complex geometry}
We denote by $TX\rightarrow X$ (resp. $T^*X\rightarrow X$) the (smooth) tangent (resp. cotangent) bundles of $X$. Its transition functions are the differentials $d\psi_{UV}\in \GL_2(\bbR)$ (resp. the duals of the differentials $(d\psi_{UV})^*$) of the coordinate change functions $\psi_{UV}$ determined by the smooth structure of $X$. Moreover, since by assumption the $\psi_{UV}$ are holomorphic, they satisfy the Cauchy--Riemann equations, which precisely means that the matrices $d\psi_{UV}$ actually lie in the image of $\C^*$ under the natural embedding $\C^* \rightarrow \GL_2(\bbR)$. This means that $TX$ and $T^*X$ are actually holomorphic line bundles. When regarded as such, we shall write $\bm{T}X$ and $\bm{T}^* X$. Equivalently, we can see $X$ as equipped with an (integrable, although this is trivial in this case for dimensional reasons) almost-complex structure $I:TX \otimes \C \rightarrow TX \otimes \C$, which splits the complexified tangent space in eigenspaces $T^{1,0}X$ and $T^{0,1}X$, and then identify $\bm{T}X = T^{1,0}X$. In turn, we identify $\bm{T}^* X = (T^{1,0}X)^*$.

If $E\rightarrow X$ is a vector bundle, then we generally denote its space of sections as $\Gamma(X,E)$ or, more generally, we write $\Gamma(U,E)$ to denote the space of sections of $E$ defined on an open subset $U\subset E$. These sets determine a sheaf $\Gamma_E$, the sheaf of sections of $E$. It is a locally free sheaf of $\mathscr{A}$-modules, where $\mathscr{A}$ is the sheaf $C^\infty_X$, $\sO_X$ or $\underline{C}_X$, depending on if $E$ is smooth, holomorphic or locally constant. When $E$ is smooth, we shall also use the notation $\Omega^0(U,E):= \Gamma(U,E)$. The sheaf of sections of a holomorphic line bundle will generally be denoted with calligraphic font $\mathscr{E}$, and sometimes we will abuse notation and identify a holomorphic line bundle with its sheaf of sections, even if they are two different objects.

Given an integer $k$, we consider the (complexified) wedge product bundle $\Lambda^kX \otimes \C := \wedge^k (TX\otimes \C)$, whose global sections $\Omega^k(X)$ ($=\Omega^k(X,\C)$) are the (complex-valued) differential $k$-forms on $X$.
Given integers $p$ and $q$, we can consider the wedge product bundle $\Lambda^{p,q} X=\wedge^p (T^{1,0}X)^* \wedge \wedge^q(T^{0,1}X)^*$. The sheaf of sections of this bundle is denoted by $\Omega^{p,q}_X$. The space of global sections $\Omega^{p,q}(X)=\Gamma(X,\Lambda^{p,q}X)$ is the space of $(p,q)$-forms. More generally, if $E$ is a smooth complex vector bundle, we can also consider the tensor vector bundles $\Lambda^k X\otimes E$ and $\Lambda^{p,q} X \otimes E$, and define $\Omega^k(X,E)=\Gamma(X,\Lambda^k X\otimes E)$ $\Omega^{p,q}(X,E):= \Gamma(X,\Lambda^{p,q} X \otimes E)$. In particular, we obtain the (holomorphic) sheaves of holomorphic $p$-forms $\bm{\Omega}^p_X := \Omega^{p,0}_X$.

\section{Connections and curvature}
Recall that we have exterior differentiation $d:\Omega^k(X)\rightarrow \Omega^{k+1}(X)$. Connections provide a way to generalize exterior differentiation to differential forms with coefficients on a vector bundle.

 \begin{defn}
Let $E$ be a smooth vector bundle on $X$. A \emph{connection} $D$ on $E$ is a  $\C$-linear operator 
\begin{equation*}
D: \Omega^0(X,E) \rightarrow \Omega^1(X,E)
\end{equation*} 
satisfying the \emph{Leibniz rule}
\begin{equation*}
D(fs)= fD s + df \otimes s,
\end{equation*} 
for $f$ a smooth function on $X$ and $s$ a section of $E$ on $X$.
\end{defn}

\begin{rmk}
Note that the existence of smooth partitions of unity implies that $D$ can actually be regarded as a map of sheaves of sections
\begin{equation*}
D: \Gamma_E \rightarrow \Gamma_E \otimes \Omega^1_X.
\end{equation*} 
\end{rmk}

\begin{rmk}
If $U\in \fU$ is an open subset of $X$ in the complex atlas, then the space of sections of  $E$ on $U$ is a free $C^\infty(U)$-module of rank $r$. A basis  $\left\{e_1,\dots,e_r\right\}$  of this module is called a \emph{frame} of $E$ on $U$. If  $D$ is a connection on  $E$, for each $e_i$ of the frame, the connection acts as
\begin{equation*}
D e_i = \sum_j e_j A_i^j,
\end{equation*} 
for some $1$-form $A_i^j \in \Omega^1(U)$. In matrix notation, writing $e=(e_1 \cdots e_r)$ and $A=(A_i^j)$ as a square matrix, we obtain
\begin{equation*}
D e = eA.
\end{equation*} 
Any section of $E$ on $U$ can be written as $s=\sum_i s^i e_i$, for  $s^i \in C^\infty(U)$. Therefore, we can write
\begin{equation*}
D s = \sum_i (ds^i e_i + s^i D e_i) = \sum_i (ds^i e_i + s^i e_j A^i_j) = (d+A) s
\end{equation*} 
The matrix $A$ is called the \emph{connection $1$-form} of $D$ on $U$.
\end{rmk}

The exterior differential $d$ is well known to satisfy the condition $d^2$. This is however not true in general for connections, which gives rise to the notion of curvature. More precisely, there is a unique way of extending the map $D:\Omega^0(X,E)\rightarrow \Omega^1(X,E)$ to a map $D: \Omega^k(X,E) \rightarrow \Omega^{k+1}(X,E)$ in such a way that
\begin{equation*}
D(\omega \wedge \alpha) = d\omega \wedge \alpha + (-1)^k \omega \otimes D \alpha, 
\end{equation*} 
and
\begin{equation*}
D(\alpha \wedge \omega) = D\alpha \wedge \omega + (-1)^k \alpha \otimes D \omega,
\end{equation*} 
for $\omega \in \Omega^p(X)$ and $\alpha \in \Omega^{k-p}(X,E)$.

\begin{defn}
Let $D$ be a connection on a smooth vector bundle $E$. We define the \emph{curvature}	of $D$ as the operator
 \begin{equation*}
D^2: \Omega^0(X,E) \rightarrow  \Omega^2(X,E).
\end{equation*} 
\end{defn}

\begin{rmk}
The curvature $D^2$ is a  $C^\infty$-linear map since
 \begin{equation*}
D^2(fs)=D(sdf + fD s) = Ds \wedge df + df \wedge D s + f\wedge D ^2s = fD^2 s,
\end{equation*} 
for $s\in \Omega^0(X,E)$ and $f\in C^\infty(X)$.
\end{rmk}

\begin{rmk}
If $e$ is a local frame of $E$ on $U$, we have
\begin{equation*}
D^2(e)=D(eA)=D e \wedge A + edA = e(A\wedge A + dA) = eF_A,
\end{equation*} 
for $F_A=dA + A \wedge A$ a matrix of  $2$-forms called the \emph{curvature $2$-form} of $D$ on $U$.
\end{rmk}

\begin{ej}
Let $E_1$ and $E_2$ be two vector bundles, and consider an isomorphism $g:E_1\rightarrow E_2$. Fix two local frames  $e_1$ and $e_2$ of  $E_1$ and $E_2$, respectively, on $U$, and consider the associated matrix  $g_U: U \rightarrow \GL_r(\C)$. Let $D_1$ be a connection on $E_1$ and consider the ``gauge transformed''  connection $D_2=g \circ D_1 \circ g^{-1}$ on $E_2$.  Let $A_1$ and $A_2$ denote the corresponding connection $1$-forms on $U$, of $D_1$ and $D_2$ with respect to $e_1$ and $e_2$.  
\textit{Show that
\begin{equation*}
A_2 = g_U A_1 g_U^{-1} + g_U dg_U^{-1},
\end{equation*} 
and
\begin{equation*}
F_{A_2} = g_U F_{A_1} g_U^{-1}.
\end{equation*} }
In particular, this implies that the locally defined $F_A$ determine a globally defined $(\End E)$-valued $2$-form $F_D\in \Omega^2(X,\End E )$.
\end{ej}

\begin{ej}[Distributions and connections]
A \emph{distribution} over a smooth manifold $M$ of dimension $n$ is a subbundle $\Xi\subset TM$ of the tangent bundle of  $M$. Consider now the natural projection $p:E\rightarrow X$ of a smooth vector bundle over $X$. Its differential determines a natural morphism $TE \rightarrow p^*TX$ of vector bundles over $E$. The kernel of this map is the bundle $V_E:=p^*E\subset TE$, which we call the \emph {vertical distribution}. 

1. \textit{Prove that a connection $D$ on $E$ determines a \emph{horizontal distribution}, namely, that it determines a distribution $H_D\subset TE$ such that
 \begin{equation*}
TE = V_E \oplus H_D.
\end{equation*} }

A distribution is \emph{involutive} if, for any two sections of it (that is, for any two vector fields $\xi$, $\eta$ which lie on it), their Lie bracket $[\xi,\eta]$ also lies on $\Xi$.

2. \textit{Prove that $H_D$ is involutive if and only if $D^2=0$}.
\end{ej}

\section{Flat bundles and local systems}
\begin{defn}
A connection $D$ on a smooth vector bundle $E$ is \emph{flat} if its curvature is $0$. A pair  $(E,D)$ formed by a smooth vector bundle and a flat connection is called a \emph{flat bundle}. 

If $(E_1,D_1)$ and $(E_2,D_2)$ are two flat bundles, a \emph{morphism of flat bundles} $g:(E_1,D_1)\rightarrow (E_2,D_2)$ is determined by a morphism of bundles $g:E_1\rightarrow E_2$ such that $D_2=g\circ D_1 \circ g^{-1}$.
\end{defn}

\begin{thm}[Frobenius] \label{thm:Frobenius}
	Let $(E,D)$ be a flat bundle. Suppose that $E$ is determined by a cocycle $(g_{UV})$. Then there exists a $0$-coboundary $(f_U)$ such that the functions $g'_{UV}=(f_V)^{-1} g_{UV} f_U$ are locally constant. The corresponding local system $E'$ determined by $(g'_{UV})$ is called the \emph{holonomy local system} associated with $(E,D)$.
\end{thm}

\begin{proof}
Suppose that we can find, for each $U\in \fU$, a frame  $\epsilon_U$ of $E$ on $U$	such that $D \epsilon_U = 0$. If we start from the family of frames $\left\{e_U:U\in \fU\right\}$ determining $E$ in terms of the cocycle $(g_{UV})$, each $\epsilon_U$ is of the form  $\epsilon_U = e_U f_U$, for $f_U:U\rightarrow \GL_n(\C)$. Now, on a non-empty overlap $U\cap V$, putting $g_{UV}'=f_V^{-1} g_{UV}f_U$, we have
\begin{equation*}
0=D \epsilon_U = D(e_U f_U) = D(e_V g_{UV} f_U) = D(\epsilon_Vg'_{UV})= D \epsilon_V g'_{UV} + \epsilon_V dg'_{UV}.
\end{equation*} 
We conclude that $dg'_{UV}=0$ and thus the $g'_{UV}$ are locally constant.

It remains to see that we can find such a frame $\epsilon_U$. Equivalently, we want to find matrix-valued functions  $f_U:U\rightarrow \GL_n(\C)$ satisfying
\begin{equation*}
0= D(e_U f_U) = D(e_U)f_U + e_U df_U = e_U(A f_U + df_U),
\end{equation*} 
where $A$ is the connection  $1$-form in the frame $e_U$. Therefore, our problem is reduced to that of finding solutions  $f$ to the differential equation
\begin{equation*}
df + Af=0.
\end{equation*} 
As we explain in Exercise \ref{ej:Frobenius}, this is just an application of Frobenius theorem, where the integrability condition corresponds precisely to $F_A=dA+ A\wedge A = 0.$
\end{proof}

\begin{ej}[The Frobenius theorem. Analysts version] \label{ej:Frobenius}
Consider an open subset $U\times V\subset \bbR^m \times \bbR^n$, where  $U$ is a neighborhood of $0\in \bbR^m$. Consider a family of $C^\infty$ functions $F_1,\dots,F_m:U\times V\rightarrow \RR^n$. The theorem of Frobenius tells us	that, for every $x \in V$, there exists one and only one smooth function  $\alpha:W\rightarrow V$, defined in a neighborhood  $W$ of $0$ in $\bbR^n$, with  $\alpha(0)=x$ and solving the PDE
\begin{equation*}
\frac{\partial \alpha}{\partial t^i}(t)=F_i(t,\alpha(t)), \text{ for all } t\in W,
\end{equation*} 
if and only if there is a neighborhood of $(0,x)\in U\times V$ on which
 \begin{equation*}
\frac{\partial F_j}{\partial t^i} - \frac{\partial F_i}{\partial t^j} + \sum_{k=1}^n \frac{\partial F_j}{\partial x^k} F_i^k - \sum_{k=1}^n \frac{\partial F_i}{\partial x^k} F_j^k  =0,
\end{equation*} 
for $i,j=1,\dots,m$.
\textit{Prove that the equation $df+ Af=0$ can be written as the PDE above, and that the integrability condition corresponds to the condition $dA + A\wedge A = 0$.}
\end{ej}

\begin{ej}[The Frobenius theorem. Geometers version]
A distribution $D$ on a smooth manifold $M$ is \emph{integrable} if there exists some submanifold $N\subset M$ such that, for any point  $p\in N$, we have that  $T_pN=D_p$. In that case, we say that  $N$ is an integral manifold of $D$. The ``geometers version'' of the Frobenius theorem says that a distribution $D$ is integrable if and only if it is involutive.

1. \textit{Prove the geometers version of Frobenius theorem from the ``analysts version'' from the previous exercise}.

We saw in a previous exercise that a connection $D$ on  $E$ determines a horizontal distribution $H_D \subset TE$ and that $D$ is flat if and only if  $H_D$ is involutive. Consider the corresponding integral manifold $Y\subset E$.

2.\textit{	 Show that the natural projection $p:E\rightarrow X$ restricts to a local homeomorphism  $\pi:Y\rightarrow X$. }

3.\textit{	 Show that this $\pi:Y\rightarrow X$ is in fact a covering space.}

4.\textit{	 Show that the monodromy representation associated with $\pi:Y\rightarrow X$ coincides with the monodromy representation associated with the holonomy local system determined by $(E,D)$.}
\end{ej}

\begin{ej}
We have shown that with any flat bundle $(E,D)$ we can associate a local system $E'$. We want to show that this can be upgraded to an equivalence of categories. In order to do so, we must show the following.

1. \textit{Show that there is a bijection between the set of morphisms $(E_1,D_1)\rightarrow (E_2,D_2)$ of flat bundles and between the set of $0$-coboundaries $(f_U)$ such that $g'_{1,UV}=f_U g'_{2,UV} f_V^{-1}$}.

2. \textit{Given a local system $E'$ over $X$, construct a flat bundle $(E,D)$ such that its holonomy local system is isomorphic to $E'$}.
\end{ej}

\begin{ej}[de Rham theorem for degree 1 cohomology]
Consider the trivial line bundle $\C_X:=X\times \C\rightarrow X$. \textit{Show that the set of equivalence classes of flat connections on $\C_X$ is in natural bijection with the de Rham cohomology group  $H^1_{\mathrm{dR}}(X,\C)$. Use the correspondence of this section to prove that this group is isomorphic to the singular cohomology group $H^1(X,\C)$.}
\end{ej}

\begin{ej}[Connections of constant central curvature]
Let $\omega_X\in \Omega^2(X)$ be a volume form on $X$ with $\int_X \omega_X = 1$.	A connection $D$ on a smooth vector bundle has \emph{constant central curvature} if
\begin{equation*}
F_D = c \id_E \omega_X,
\end{equation*} 
for some constant $c\in \C^*$. Formally, we can assume that $\omega_X$ can get more and more concentrated at a single point $x_1\neq x_0 \in X$ so that in the limit we would obtain a ``Dirac delta distribution'' $\omega_X=\delta(x_1)$. In this limit, a connection $D$ of constant central curvature is flat away from $x_1$, so $(E,D)$ restricts to a flat bundle over $X\setminus \left\{x_1\right\}$ and thus determines a representation of the fundamental group
\begin{equation*}
\rho: \pi_1(X\setminus \left\{ x_1 \right\},x_0) \rightarrow \GL_r(\C).
\end{equation*} 
\textit{Let $\sigma$ be a contractible loop in $X$ around $x_1$ and based in  $x_0$.
Show that the representation $\rho$ must map the class of $\sigma$ to $\exp(c)I_r$.}
\end{ej}

As a wrap up for this section, we summarize the main results in the following.

\begin{thm}
  The holonomy representation determines an equivalence of categories between the category of flat bundles on $X$ and the category of linear representations of the fundamental group $\pi_1(X,x_0)$.

  More generally, holonomy determines an equivalence between the category of pairs $(E,D)$ formed by a smooth vector bundle on $X$ and a connection $D$ of constant central curvature $F_D=c\id_E \omega_X$ and the category of representations $\pi_1(X\setminus \left\{ x_1 \right\},x_0)\rightarrow \GL_r(\C)$ mapping the class of a loop $\sigma$ around $x_1$, based in $x_0$ and contractible in $X$, to $\exp(c)I_r$.
\end{thm}

\section{Holomorphic structures as Dolbeault operators}
Recall that we have the Dolbeault differentials $\delbar: \Omega^{p,q}(X)\rightarrow \Omega^{p,q+1}(X)$. Holomorphic structures arise naturally by generalizing Dolbeault operators to vector bundles.

\begin{defn}
Let $E$ be a smooth vector bundle on $X$. A \emph{holomorphic structure} $\delbar_{\sE}$ on $E$ is a $\C$-linear operator
\begin{equation*}
\delbar_{\sE}: \Omega^0(X,E) \rightarrow \Omega^{0,1}(X,E)
\end{equation*} 
such that
\begin{equation*}
\delbar_{\sE}(fs)=s \delbar f + f \delbar s
\end{equation*} 
for every smooth function $f$ on $U$ and every section $s$ of $E$ on $U$, for any open subset $U\subset X$. 
\end{defn}

\begin{rmk}
In higher dimensions, to obtain a holomorphic structure one should add the ``integrability condition'' that $\delbar_{\sE}^2 = 0$. However on Riemann surfaces this condition is empty, since  $\Omega_X^{0,2}=0$.
\end{rmk}

There is what we could call an ``analogue'' of the Frobenius theorem for holomorphic structures, with a sustantially more difficult proof.

\begin{thm}
Consider a pair $(E,\delbar_{\sE})$ formed by a smooth vector bundle on $X$ with a holomorphic structure $\delbar_{\sE}$. Suppose that  $E$ is determined by a cocycle $(g_{UV})$. Then there exists a $0$-coboundary $(f_U)$ such that the functions $g_{UV}'=(f_V)^{-1}g_{UV}f_U$ are holomorphic.
\end{thm}

\begin{rmk}
Following the same arguments as in the proof of \ref{thm:Frobenius}, it suffices to show that there exist local frames $\epsilon_U$ with  $\delbar_E \epsilon_U$. This problem itself reduces to finding solutions to the equation
 \begin{equation*}
\delbar f + Af = 0.
\end{equation*} 
We refer the reader to \cite[Section 5]{AtiyahBott} for details on the integrability of this equation. Alternatively, we can also deduce the result as an application of the Newlander--Niremberg theorem about the integrability of almost-complex structures.
\end{rmk}

\begin{rmk}
The above theorem tells us that, instead of thinking about a holomorphic vector bundle $\sE$, we can think about the pair $(E,\delbar_{\sE})$ formed by the smooth vector bundle $E$ underlying $\sE$ and the holomorphic structure $\delbar_\sE$. This is the typical approach in gauge theory.
\end{rmk}

\begin{rmk}
We also remark the fact that, if $D$ is a connection on a smooth vector bundle $E$, then we can take its $(0,1)$ part $\delbar_D=D^{0,1}$, which determines a holomorphic structure on $E$. 
\end{rmk}

\section{Holomorphic connections}
In the previous sections, we have studied ``smooth'' connections. Considering a notion of connection which is intrinsically holomorphic is quite interesting, as it also allows us to regard the same objects in algebraic geometry.

\begin{defn}
  Let $\sV$ be a holomorphic vector bundle on $X$. A \emph{holomorphic connection} $\bm{D}$ on $\sV$ is a $\C$-linear morphism of sheaves
  \begin{equation*}
\bm{D}: \sV \longrightarrow \sV \otimes \bm{\Omega}^1_X
  \end{equation*}
  satisfying the \emph{holomorphic Leibniz rule}
  \begin{equation*}
\bm{D}(fs) = f \bm{D} s + \partial_X f \otimes s,
  \end{equation*}
  for any local sections $f \in \sO_X(U)$ and $s\in \sV(U)$.

  A \emph{holomorphic vector bundle with connection} is a pair $(\sV,\bm{D})$ formed by a holomorphic vector bundle $\sV$ on $X$ and a holomorphic connection $\bm{D}$ on $\sV$. A morphism $f:(\sV_1,\bm{D}_1) \rightarrow (\sV_2,\bm{D}_2)$ of holomorphic vector bundles with connection is a morphism of holomorphic vector bundles $g:\sV_1\rightarrow \sV_2$ such that $\bm{D}_2 \circ g = (g\otimes \id_{\bm{\Omega}^1_X}) \circ \bm{D}_1$.
\end{defn}

\begin{rmk}
Note that, since $X$ has complex dimension $1$, holomorphic connections on $X$ are automatically integrable.
\end{rmk}

Let $(\sV,\bm{D})$ be a holomorphic vector bundle with holomorphic connection and let $E$ be the underlying smooth complex vector bundle. The holomorphic structure of $\sV$ determines the operator
  \begin{equation*}
\delbar_{\sV}: \Omega^0(X,E) \longrightarrow \Omega^{0,1}(X,E).
  \end{equation*}
  On the other hand, the holomorphic connection $\bm{D}$ determines an operator
  \begin{equation*}
\bm{D}: \Omega^0(X,E) \longrightarrow \Omega^{1,0}(X,E).
  \end{equation*}
  The fact that $\bm{D}$ is a holomorphic connection on $\sV$ is equivalent to the commutativity of these two operators. That is, to the vanishing of the $E$-valued $(1,1)$-form
  \begin{equation*}
[\bm{D},\delbar_{\sV}] \in \Omega^{1,1}(X,E).
  \end{equation*}

\begin{ej}
 \textit{Verify that the sum}
  \begin{equation*}
D=\bm{D} + \delbar_{\sV} : \Omega^0(X,E) \rightarrow \Omega^1(X,E)
  \end{equation*}
\textit{determines a connection on $E$. Show that this connection is flat.} \textbf{Hint}: Show that $F_D=[\bm{D},\delbar_{\sV}]=0$.
\end{ej}

Conversely, if $(E,D)$ is a flat bundle on $X$, the $(0,1)$-part of the connection $D^{0,1}$ determines a holomorphic structure on $E$, while the $(1,0)$-part $D^{1,0}$ determines a holomorphic connection on the holomorphic vector bundle $(E,D^{0,1})$. Indeed, we have $[D^{1,0},D^{0,1}]=F_D=0$. We conclude the following.

\begin{prop} \label{prop:holconnflatconn}
The map $(\sV,\bm{D})\mapsto (E,\bm{D}+\delbar_{\sV})$ determines an equivalence of categories between the category of holomorphic bundles with connection on $X$ and the category of flat bundles on $X$.
\end{prop}

\begin{ej}
The correspondence above can be extended to consider connections with constant central curvature. In order to do this, we need to fix a point $x_1\in X$ and consider the sheaf $\bm{\Omega}^1_X(* x_1)$ of meromorphic $1$-forms on $X$ with a simple pole on $x_1$. A \emph{meromorphic connection} $\bm{D}$ on a holomorphic vector bundle $\sV$ is a $\C$-linear morphism of sheaves
  \begin{equation*}
\bm{D}: \sV \longrightarrow \sV \otimes \bm{\Omega}^1_X(* x_1)
  \end{equation*}
  satisfying the holomorphic Leibniz rule
  \begin{equation*}
\bm{D}(fs) = f \bm{D} s + \partial_X f \otimes s,
  \end{equation*}
  for any local sections $f \in \sO_X(U)$ and $s\in \sV(U)$.
  Near the point $x_1$, the ``connection $1$-form'' of any such connection is a meromorphic $\End E$-valued $1$-form $A(z)$ with a Laurent expansion near the point $x_1$ of the form
  \begin{equation*}
A(z) = A_{-1} z^{-1} + A_0 + A_1 z + A_2 z^2 + \dots.
  \end{equation*}
  The matrix $\mathrm{res}_{x_1}\bm{D}=A_{-1}$ is called the \emph{residue} of $\bm{D}$ at $x_1$. \textit{Show that, if $(\sV,\bm{D})$ is a bundle with meromorphic connection such that $\mathrm{res}_{x_1}\bm{D}=-i\tfrac{c}{2\pi} I_r$, then $D=\bm{D}+\delbar_{\sV}$ is a connection with constant central curvature $F_D=c\id_E \omega_X$, and viceversa.}
\end{ej}

We conclude the \emph{Riemann-Hilbert correspondence}.
\begin{thm}
The map $(\sV,\bm{D})\mapsto (E,\bm{D}+\delbar_{\sV})$, composed with the holonomy representation, determines an equivalence of categories between the category of holomorphic bundles with connection on $X$ and the category of linear representations of the fundamental group $\pi_1(X,x_0)$.

More generally, it determines an equivalence between the category of holomorphic bundles with meromorphic connection $(\sV,\bm{D})$ with a simple pole on $x_1$ and with $\mathrm{Res}_{x_1}\bm{D}=-i\tfrac{c}{2\pi} I_r$ and the category of representations $\pi_1(X\setminus \left\{ x_1 \right\},x_0)\rightarrow \GL_r(\C)$ mapping the class of a loop $\sigma$ around $x_1$, based in $x_0$ and contractible in $X$, to $\exp(c)I_r$.
\end{thm}

\section{Hermitian metrics and the Chern correspondence}
\begin{defn}
Let $E$ be a smooth vector bundle on $X$. A \emph{Hermitian metric} $H$ on $E$ is determined by a Hermitian product  $\langle-,-\rangle_{H,x}$	on each fibre $E_x$, in such a way that for every open subset $U\subset X$ and for every two sections  $s$ and $t$ of  $E$ on $U$, the map
\begin{align*}
\langle s,t \rangle_H: U & \longrightarrow \C \\
x & \longmapsto \langle s(x), t(x) \rangle_{H,x}
\end{align*}
is smooth. A pair $(E,H)$ formed by a smooth vector bundle and a Hermitian metric is called a \emph{Hermitian vector bundle}.
\end{defn}

\begin{defn}
Let $(E,H)$ be a Hermitian vector bundle. A connection $\nabla$ on $E$ is ($H$-)\emph{unitary} if, for every two local sections $s$ and  $t$ of $E$ and for every vector field of $\xi$ on an open $U\subset X$, we have
\begin{equation*}
d \langle s,t \rangle_H(\xi) = \langle \nabla s(\xi), t \rangle_H + \langle s, \nabla t(\xi) \rangle_H.
\end{equation*}
\end{defn}

\begin{ej}
\textit{Show that if $(E,H)$ is a Hermitian vector bundle and $\nabla$ is a flat  unitary connection on $E$, then the monodromy representation $\rho:\pi_1(X,x_0)\rightarrow \GL_n(\C)$ associated with the corresponding local system factors through the unitary group $\U(n)\subset \GL_n(\C)$ .}
\end{ej}

\begin{thm}
Let $\sE=(E,\delbar_{\sE})$ be a holomorphic vector bundle. For every Hermitian metric $H$ on $E$, there exists a unique unitary connnection $\nabla_H$ on  $E$ such that $\delbar_{\sE}=\nabla_H^{0,1}$. This connection is called the \emph{Chern connection}.
\end{thm}

\begin{proof}
Consider a frame $\left\{e_1,\dots,e_r\right\}$	of $E$ over some open subset $U\in \fU$, and assume that this frame is holomorphic; that is, that $\delbar_{\sE} e_i = 0$, for  $i=1,\dots,r$. Let us consider the functions $h_{ij}=\langle e_i, e_j \rangle_H$. If such a $\nabla_H$ exists, then its connection $1$-form $A$ with respect to this framing must be of type $(0,1)$, since we must have $\delbar_{\sE}=\nabla_H^{0,1}$. But then
\begin{equation*}
dh_{ij} = d \langle e_i, e_j \rangle = \sum_k A_i^k h_{kj} + h_{ik} \bar{A}^k_j,
\end{equation*}
so $\partial h_{ij}= \sum_k A^k_i h_{kj}$ and $\delbar h_{ij} =\sum_k h_{ik}\bar{A}^k_j$. Therefore, if we consider the matrix $h=(h_{ij})$ we can just set
$A = h^{-1} \partial h$.
\end{proof}

\section{Principal $G$-bundles}
Let $G$ be a Lie group. Recall the definition of vector bundle from Section \ref{ss:vectorbundles}. If we replace the local pieces $E_U$ by $P_U=U\times G$ and the functions $g_{UV}$ determining the gluing cocycle by continuous functions $g_{UV}:U\cap V \rightarrow G$, we obtain the notion of a \emph{principal $G$-bundle} $P\rightarrow X$. Here, we are regarding the group $G$ acting on itself by right multiplication. Moreover, if the functions $g_{UV}$ are smooth or locally constant, we obtain a \emph{smooth principal $G$-bundle} or a \emph{$G$-local system}, respectively. If the group $G$ is a complex Lie group or a complex algebraic group, then we can also consider functions $g_{UV}$ which are holomorphic or regular, respectively, in which case we get a \emph{holomorphic principal $G$-bundle} or an \emph{algebraic principal $G$-bundle}, respectively. It also follows from Serre's GAGA that the categories of holomorphic and principal $G$-bundles are equivalent \textit{over a Riemann surface}\footnote{This is indeed trickier in higher dimensions, since the Zariski topology is usually not fine enough to trivialize holomorphic $G$-bundles. Therefore, in algebraic geometry principal $G$-bundles are generally locally trivialized in the \emph{étale topology}. However, over complex dimension $1$ this distinction is not necessary, since every non-compact Riemann surface is a Stein manifold.}.

\begin{ej}
\textit{Show that a principal $G$-bundle admits a section if and only if it is isomorphic to a trivial bundle}.
\end{ej}

A complex vector bundle $E\rightarrow X$ determines automatically a principal $\GL_r(\C)$-bundle, its \emph{frame bundle} $\mathrm{Fr}(E)\rightarrow X$, whose fibre over $x\in X$ is the set of basis of the fibre $E_x$. The group $\GL_r(\C)$ has a natural free and transitive action on the fibre $\mathrm{Fr}(E)_x$. Conversely, a principal $\GL_r(\C)$-bundle $P\rightarrow X$ determines a vector bundle $E\rightarrow X$, defined as
\begin{equation*}
E = (P \times \C^r)/\left\{ (p,v) \sim (p\cdot g, g^{-1} \cdot v),\ g\in \GL_r(\C) \right\}.
\end{equation*}
More generally, for any continuous linear representation $\rho:G \rightarrow \GL(V)$, of the group $G$ on a vector space $V$, any principal $G$-bundle $P\rightarrow X$ determines a vector bundle
\begin{equation*}
E:= P\times^{\GL_r(\C)} V := (P \times V)/\left\{ (p,v) \sim (p\cdot g, \rho(g)^{-1} \cdot v),\ g\in \GL_r(\C) \right\}.
\end{equation*}

Given a morphism of Lie groups $f:H\rightarrow G$, any principal $H$-bundle $P_H$ determines a principal $G$-bundle $P_G$, by putting
\begin{equation*}
P_G:= P_H\times^{H} G := (P_H \times G)/\left\{ (p,g) \sim (p\cdot h, f(h)^{-1} g),\ h\in H \right\}.
\end{equation*}
When $f:H\hookrightarrow G$ is the inclusion of a closed subgroup, this is called \emph{extension of the structure group}. Conversely, if $P_G$ is a $G$-bundle, then a \emph{reduction of the structure group} from $G$ to $H$ is determined by a $G$-equivariant map $\sigma: P_G \rightarrow G/H$, where $G/H$ is the corresponding homogeneous $G$-space.

\begin{ej}
\textit{Show that, given a reduction of structure group $\sigma:P_G \rightarrow G/H$, the fibre $P_\sigma=\sigma^{-1}(H)\subset P_G$ naturally admits the structure of a principal $H$-bundle.  Show that, if the map $\sigma$ is smooth, holomorphic or locally constant, then $P_\sigma$ is smooth, holomorphic or a local system, respectively.}
\end{ej}

\begin{ej}
\textit{Show that a Hermitian metric $H$ on a complex vector bundle $E$ determines a reduction of the frame bundle $\mathrm{Fr}(E)$ from $\GL_r(\C)$ to $\U(r)$.}
\end{ej}

\chapter{Classifying vector bundles} \label{sec:classification}
\section{Topological classification}
With any vector bundle $E$ on $X$ we can associate its \emph{determinant line bundle}, defined as follows. If $E$ is determined by gluing spaces of the form $E_U=U\times \C^r$ via transition functions $g_{UV}:U\cap V \rightarrow \GL_r(\C)$, then $\det E$ is obtained by gluing the spaces $(\det E)_U = U \times \wedge^r \C^r$ through the transition functions $\det g_{UV}: U\cap V \rightarrow \C^*$.

\begin{ej}
\textit{Show that any smooth vector bundle $E$ of rank $r>1$ has a nowhere vanishing global section.	Is this true for holomorphic vector bundles?} \textbf{Hint}: Use a transversality argument.
\end{ej}

The section $s$ from the exercise determines an injection $s:\C_X \hookrightarrow E$. Now, any smooth vector bundle admits a Hermitian metric, so we can orthogonally decompose  $E=s(\C_X) \oplus s(\C_X)^\perp$. Iterating this process, we obtain that $E$ can be written as
\begin{equation*}
E = \C_X^{r-1} \oplus L,
\end{equation*} 
for some line bundle $L$. Note however that  $L$ must be isomorphic to the determinant line bundle $L\cong \det E$. We conclude the following.

\begin{thm}
A smooth vector bundle on $X$ is determined by its rank and its determinant.	
\end{thm}

It thus remains to solve the question of classfying smooth line bundles. Now, recall that these line bundles are classified by the cohomology group $H^1(X,(C^\infty_X)^*)$, where $(C^\infty_X)^*$ denotes the sheaf of smooth functions $U\rightarrow \C^*$. The exponential exact sequence
\begin{center}
\begin{tikzcd}
	0 \rar & 2\pi i \Z \rar & C^\infty_X\rar & (C^\infty_X)^* \rar & 0
\end{tikzcd}
\end{center}
induces an exact sequence
\begin{center}
\begin{tikzcd}
	H^1(X,C^\infty_X) \rar & H^1(X,(C^\infty_X)^*) \rar{\delta} & H^2(X,2\pi i \Z) \rar & H^2(X,C^\infty_X).
\end{tikzcd}
\end{center}
The existence of smooth partitions of unity implies that $H^i(X,C^\infty_X)=0$ for $i>0$, so we obtain an isomorphism  $\delta:H^1(X,(C^\infty_X)^*)\cong H^2(X,2\pi i \Z)$. If $L$ is a smooth line bundle on $X$ represented by a cohomology class $[L]\in H^1(X,(C^\infty_X)^*)$, we define the \emph{first Chern class} of $L$ as
 \begin{equation*}
c_1(L) = \frac{i}{2\pi} \delta([L]) \in H^2(X,\Z).
\end{equation*} 
Recall that integration determines an isomorphism
\begin{align*}
\int_X: H^2(X,\Z)  \rightarrow \Z, \  \alpha  \mapsto \int_X \alpha.
\end{align*} 
The \emph{degree} of $L$ is the number
\begin{equation*}
\deg L= \int_X c_1(L) \in \Z.
\end{equation*} 

More generally, if $E$ is a vector bundle, then we define its first Chern class as $c_1(E)=c_1(\det E)$, and its degree as $\deg E = \deg(\det E)$.

\begin{ej}[Chern-Weil theory: Computing Chern classes using curvature]
Let $E$ be a smooth vector bundle on $X$. If $D$ is a connection on $E$ and $F_D \in \Omega^2(X,\End E)$ its curvature. Its trace determines a  $2$-form $\mathrm{tr}(F_D)\in \Omega^2(X)$, and we can consider its cohomology class $[\tr(F_D)]$.
\textit{Prove that}
\begin{equation*}
c_1(E) = \frac{i}{2\pi} [\tr(F_D)].
\end{equation*} 
In particular, this implies that, if $E$ admits a flat connection, then $\deg E = 0$.
\end{ej}

To sum up, we conclude the following.
\begin{thm}
Smooth vector bundles on a Riemann surface are classified by their rank and their degree.
\end{thm}

\section{Holomorphic line bundles: the Jacobian}
Consider now the sheaf $\sO_X$ of holomorphic functions on $X$ and the sheaf $\sO^*_X$ of non-vanishing holomorphic functions.  Isomorphism classes of holomorphic line bundles form the \emph{Picard group} $\Pic(X)=H^1(X,\sO_X^*)$. In the holomorphic case we also have an exponential exact sequence
\begin{center}
\begin{tikzcd}
	0 \rar & 2\pi i \Z \rar & \sO_X \rar & \sO_X^* \rar & 0,
\end{tikzcd}
\end{center}
which induces an exact sequence
\begin{center}
\begin{tikzcd}
	H^1(X,2\pi i \Z) \rar &	H^1(X,\sO_X) \rar & \Pic(X) \rar{\delta} & H^2(X,2\pi i \Z).
\end{tikzcd}
\end{center}
Consider the subgroup $\Pic^0(X)=\left\{[L]\in \Pic(X): \delta([L])=0\right\}$. It follows from the exact sequence above that $\Pic^0(X)$ is isomorphic to the \emph{Jacobian} of $X$, which is defined as the quotient
\begin{equation*}
\Jac(X) = H^1(X,\sO_X)/H^1(X,2\pi i \Z).
\end{equation*}
More generally, the map $\delta$ splits $\Pic(X)$ in connected components
\begin{equation*}
\Pic(X) = \bigsqcup_{d\in \Z} \Pic^d(X),
\end{equation*}
labelled by the degrees of the line bundles in them.

\begin{ej}
	\textit{ Verify that $\Jac(X)$ is an abelian variety of dimension $g$, where  $g$ is the genus of $X$}. \textbf{Hint}: First, you need to convince yourself that $H^1(X,\sO_X)$ is a complex vector space of dimension $g$. This follows either directly from Hodge theory or from GAGA and the fact that $X$ is the analytification of a smooth projective curve (which essentially follows from Hodge theory). Second, you need to verify that there is a Riemann form with respect to the lattice $H^1(X,2\pi i \Z)$. You can construct this form using Poincaré duality.
\end{ej}

\begin{rmk}
Note that this already hints on the complexity of classifying holomorphic vector bundles. While smooth line bundles were simply determined by a number, there is an infinite amount of isomorphism classes of holomorphic line bundles with the same degree. However, these isomorphism classes can be nicely organized in a complex manifold (a \emph{moduli space}) and we can study its geometry.
\end{rmk}

\begin{ej}
A holomorphic \emph{$\SL_{n}(\C)$-vector bundle} is a pair $(\sE,\eta)$ formed by a holomorphic vector bundle $\sE$ and by a holomorphic trivialization of its determinant line bundle $\eta: \sO_X\overset{\sim}{\rightarrow} \det \sE$. \textit{Show that the frame bundle $\mathrm{Fr}(\sE)$ of a holomorphic $\SL_n(\C)$-vector bundle admits a holomorphic reduction of structure group from $\GL_n(\C)$ to $\SL_n(\C)$.}
\end{ej}

We can also consider ``twisted'' versions of $\SL_n(\C)$-vector bundles. Namely, if we fix any holomorphic line bundle $\xi$, we can consider pairs $(\sE,\eta)$ formed by a holomorphic vector bundle $\sE$ and an isomorphism $\eta:\xi \overset{\sim}{\rightarrow} \det \sE$. We call these \emph{holomorphic $\xi$-twisted $\SL_n(\C)$-vector bundles}. Generally we will drop the trivialization $\eta$ from the notation and just talk about the vector bundle $\sE$.

\begin{ej}
Let $\xi$ and $\xi'$ be holomorphic line bundles with the same degree $d$. \textit{Show that there is a natural equivalence of categories between the category of holomorphic $\xi$-twisted $\SL_n(\C)$-vector bundles and the category of holomorphic $\xi'$-twisted $\SL_n(\C)$-vector bundles}. Therefore, we can just fix our favourite $\xi$ (for example, $\xi=\sO_X(d x_1)$, for some point $x_1\in X$) and talk about \emph{holomorphic $d$-twisted $\SL_n(\C)$-vector bundles} without loss of generality. \textbf{Hint}: The equivalence is given by tensorization with the degree $0$ line bundle $\xi^{-1} \xi'$.
\end{ej}

\begin{ej}
Let $\sE$ and $\sE'$ be two holomorphic vector bundles and consider their frame bundles $\mathrm{Fr}(\sE)$ and $\mathrm{Fr}(\sE')$. Let $\mathrm{PFr}(\sE)$ and $\mathrm{PFr}(\sE')$ be the corresponding principal $\PGL_r(\C)$-bundles induced by the natural projection $\GL_r(\C)\rightarrow \PGL_r(\C)$. \textit{Show that $\mathrm{PFr}(\sE)$ and $\mathrm{PFr}(\sE')$ are isomorphic if and only if there is some holomorphic line bundle $\sL$ such that $\sE'=\sE \otimes \sL$.}
Suppose moreover that $\sE$ and $\sE'$ are holomorphic $d$-twisted $\SL_n(\C)$-vector bundles. \textit{Show that $\sE'=\sE \otimes \sL$ if and only if $\sL^{\otimes r}$ is trivial, and thus determines an $r$-torsion point in the Jacobian $\Jac(X)$.}
\end{ej}

\section{Towards the moduli space}

\paragraph{Early generalizations.} Generalizing the notion of Jacobian to higher rank was the starting point of the theory of moduli spaces of bundles, in the work of Weil (see Grothendieck's note \cite{Grothendieck_Bourbaki}).
When the genus of $X$ is low, the problem of classifying vector bundles can be solved relatively easily. For example, Grothendieck \cite{Grothendieck} showed that every holomorphic vector bundle over the Riemann sphere $\C\mathbb{P}^1$ can be decomposed as a direct sum of line bundles. For genus $1$, Atiyah \cite{atiyah} obtained a explicit description of vector bundles in terms of extensions. The problem gets its full complexity for genus $\geq 2$.

\paragraph{Non-abelian sheaf cohomology.} For studying the problem in general genus, the first naive approach would be to simply consider the whole set of isomorphism classes of vector bundles of rank $r$. As we explain in Section \ref{ss:vectorbundles}, this set of isomorphism classes can be understood as a non-abelian sheaf cohomology set
\begin{equation*}
\Bun_{r}:=H^1(X,\GL_r(\C)_X),
\end{equation*}
where $\GL_r(\C)_X$ is the sheaf of germs of holomorphic functions from $X$ to $\GL_r(\C)$. This set has a natural geometric structure, as it is the quotient of the space of \v{C}ech $1$-cocycles $Z^1(\mathfrak{U},\GL_r(\C)_X)$ by the action of the topological group of $0$-cochains $C^0(\fU,\GL_r(\C)_X)$. The space $\Bun_r$ has infinite connected components, labelled by the degrees of the vector bundles. For each $d\in \Z$, we denote by $\Bun_{r,d}$ the connected component of isomorphism classes of vector bundles of rank $r$ and degree $d$.

\paragraph{Holomorphic structures.} Let $E\rightarrow X$ be a smooth vector bundle of rank $r$ and degree $d$. We can consider the space $\cC_E$ of holomorphic structures $\delbar_{\sE}$ on $E$. The difference of any two holomorphic structures is a $(1,0)$-form valued in in $\End E$. Therefore, the space $\cC_E$ is an affine space modelled by the infinite dimensional vector space $\Omega^{0,1}(X,\End E)$. The \emph{complex gauge group} $\cG_E^\C= \Omega^0(X,\Aut E)$ is an infinite dimensional Lie group which acts on $\cC_E$ by conjugation
\begin{equation*}
g \cdot \delbar_{\sE} = g \delbar_{\sE} g^{-1}.
\end{equation*}
For such a $g\in \cG_E^\C$, the holomorphic vector bundles $(E,\delbar_{\sE})$ and $(E,g\cdot \delbar_{\sE})$ are isomorphic. Conversely, two holomorphic vector bundles $\sE$ and  $\sE'$ are isomorphic if and only if their associated operators $\delbar_{\sE}$ and $\delbar_{\sE'}$ are related by some $g\in \cG_E^\C$.
The quotient set
\begin{equation*}
\Bun_{r,d}=\cC_E/\cG_E^\C
\end{equation*}
is again the set of isomorphism classes of holomorphic vector bundles with underlying smooth bundle $E$. This provides an equivalent way to endow $\Bun_{r,d}$ with a geometric structure.

\paragraph{Algebraic moduli problem and moduli spaces.} Another way to endow the set $\Bun_{r,d}$ with natural geometric structure comes from algebraic geometry.  In order to do so, we regard $X$ as the analytification of some smooth complex projective curve, that we also denote by $X$. In general, if $Q$ is a set of geometric objects over $X$ and $S$ is a $\C$-scheme, then by a \emph{family of objects of $Q$ parametrized by $S$} we mean a locally free sheaf $\sF$ on $S\times X$ such that, for every closed point $s\in S(\C)$, the isomorphism class of $\sF_s$ is an element of $Q$. We say that two families $\sF_1$ and $\sF_2$ parameterized by $S$ are \emph{equivalent} if there is some line bundle $\sL \rightarrow S$ such that
\begin{equation*}
\sF_2 \cong \sF_1 \otimes \mathrm{pr}_{S}^* \sL.
\end{equation*}
The \emph{moduli problem for $Q$} is the functor
\begin{align*}
\mathsf{Q}: (\C\text{-schemes})^{\mathrm{op}}  \longrightarrow  \text{Set}
\end{align*}
which maps a $\C$-scheme $S$ to the set of equivalence classes of families of objects of $Q$ parametrized by $S$, and a morphism  $S\rightarrow T$ to the map sending a family to its pull-back.

A \emph{fine moduli space for $Q$} is a $\C$-scheme $M$ \emph{representing} $\mathsf{Q}$, that is, for every $\C$-scheme $S$, we have
\begin{equation*}
\mathsf{Q}(S) = \Hom(S,M).
\end{equation*}
In particular, note that $Q=M(\C)$, so this is a natural way to give a geometric description of $Q$. The existence of a fine moduli space amounts to the existence of a ``universal family'' $\mathscr{U}\rightarrow M\times X$ from which every other family arises as pullback. That is, every family $\sF$ of objects of $Q$ parametrized by $S$ is equivalent to the family $\sF_u$ determined by a map $u:S\rightarrow M$ by taking the pullback
\begin{center}
  \begin{tikzcd}
 \sF_u   \ar{r} \ar{d} & \mathscr{U} \ar{d} \\
 S \times X   \ar{r}{(s,\id_X)} & M \times X.
  \end{tikzcd}
\end{center}

Sometimes, asking for a fine moduli space is too much. There is the weaker notion of coarse moduli space that will be useful for us. A \emph{coarse moduli space for $Q$} is a $\C$-scheme $M$ with a morphism of functors $\Psi:\mathsf{Q}\rightarrow \Hom(-,M)$ \emph{universally corepresenting} $\mathsf{Q}$, meaning that
\begin{enumerate}
	\item $\Psi(\C):Q=\mathsf{Q}(\C) \rightarrow M(\C)$ is a bijection,
	\item for every $\C$-scheme  $M'$ and any morphism $\Psi': \mathsf{Q} \rightarrow \Hom(-,M')$, there exists a unique morphism $M\rightarrow M'$ such that the following diagram commutes
		 \begin{center}
		\begin{tikzcd}
		&& \Hom(-,M) \ar{dd} \\
			\mathsf{Q} \ar{rru}{\Psi} \ar{rrd}{\Psi'} && \\
		&& \Hom(-,M')  .
		\end{tikzcd}
		\end{center}
\end{enumerate}

We could try to take $Q=\mathrm{Bun}_{r,d}$ and thus look for a fine or a coarse moduli space for it, but it turns out that such a space cannot exist. Algebraic geometers find two ways to solve this issue. The first one is to restrict the problem to consider a very wide subclass of bundles for which a coarse moduli space can be constructed; these are the \emph{stable bundles}. The other solution is to consider the theory of stacks. Roughly, a stack is similar to a moduli functor, but it is ``groupoid-valued'' instead of ``set-valued'', with the appropriate higher-categorical notion necessary to make sense of this. This is then a relatively nice notion of space, that allows one to make sense and study certain geometric structures. More details about stacks are given in Section \ref{ss:stacks}.

\paragraph{The jumping phenomenon.} The main inconvenient for considering all isomorphism classes of vector bundles is the issue that the space $\Bun_{r,d}$ cannot be separated. Indeed, if it were separated the following ``jumping phenomenon'' would not happen. If the genus of $X$ is greater than $0$, then the cohomology space $H^1(X,\sO_X)$ is not trivial. This cohomology space parametrizes extensions
\begin{center}
  \begin{tikzcd}
0 \rar & \sO_X \rar & E \rar & \sO_X \rar & 0.
  \end{tikzcd}
\end{center}
Given a class $\alpha \in H^1(X,\sO_X)$, we can consider the map $\C \rightarrow H^1(X,\sO_X)$, $t\mapsto t\alpha$, which determines a continuous family $E_t$ of extensions of $\sO_X$ by $\sO_X$ parametrized by $\C$. Now, if $t,t' \neq 0$, the bundles $E_t$ and $E_{t'}$ are isomorphic and non-trivial, but for $t=0$, we have $E_0 \cong \sO_X^2$.

\begin{ej}
  Consider a point $x\in X$  and a small disk $D$ around it, with holomorphic coordinate $z$. Let $f$ be a holomorphic function on $D\setminus \left\{ x \right\}$.

  1. \textit{Show that $f$ determines a cohomology class $\alpha \in H^1(X,\sO_X)$}.

  2. The bundle $E_t$ can be trivialized over $X\setminus \left\{ x \right\}$ and over $D$. \textit{Show that then $E_t$ is determined by the transition function}
  \begin{equation*}
\begin{pmatrix}
  1 & tf \\
  0 & 1
\end{pmatrix}.
  \end{equation*}

  3. \textit{Show that $E_t$ and $E_{t'}$ are isomorphic, for $t,t'\neq 0$, and that $E_0$ is trivial}.

\end{ej}

\begin{ej} \label{ej:jumpingP1}
  On $X=\bbP^1$ we have $H^1(X,\sO_X)=0$ but the same ``jumping phenomenon'' occurs. \textit{Can you come up with an example of a $\C$-family $E_t$ of vector bundles on $\bbP^1$ such that $E_0$ is trivial and all the $E_t$ are isomorphic for $t\neq 0$?} \textbf{Hint}: Take $E_t$ to be isomorphic to $\sO_{\bbP^1}^2$ for $t\neq 0$ and to $\sO_{\bbP^1}(1)\oplus\sO_{\bbP^1}(-1)$ for $t=0$ and use ``Birkhoff factorization''
  \begin{equation*}
\begin{pmatrix}
  z^{-1} & t \\
  0 & z
\end{pmatrix}=
\begin{pmatrix}
  0 & 1 \\
  1 & t^{-1}z
\end{pmatrix}
\begin{pmatrix}
  -t^{-1} & 0 \\
  z^{-1} & t
\end{pmatrix}.
  \end{equation*}
\end{ej}

\paragraph{Size issues} Another important issue, also related with the jumping phenomenon, is that the space parametrizing all isomorphism classes of vector bundles of fixed rank and degree would by all means be ``too big''. In algebraic terms, this means that the stack of vector bundles is not of finite type. In fact, as we explain in Section \ref{ss:HarderNarasimhan}, it is an infinite union of disjoint accumulating strata.

\begin{ej} \label{ej:stratificationP1}
  This ``weird topology'' of the space of vector bundles is easy to illustrate on $X=\bbP^1$. A vector bundle on $\bbP^1$ of degree $0$ and rank $2$ is isomorphic to a bundle of the form $E_n=\sO_{\bbP^1}(n)\oplus \sO_{\bbP^1}(-n)$. Therefore, as a set $\Bun_{2,0}=\left\{ E_n:n\in \mathbb{N} \right\}$. However, Exercise \ref{ej:jumpingP1} shows that $E_1$ lies in the closure of $E_0$. \textit{Using a similar argument, show that $E_n$ lies in the closure of $E_{m}$, for every $m\leq n$.} If we denote $\Bun_{2,0}^n = \left\{ E_n \right\}$, we conclude that we can stratify
  \begin{equation*}
\Bun_{2,0} = \bigcup_{n\in \mathbb{N}} \Bun_{2,0}^n,
  \end{equation*}
  and that the closure of a stratum $\Bun_{2,0}^n$ is the union
  \begin{equation*}
\overline{\Bun_{2,0}^n} = \Bun_{2,0}^{\geq n} :=\bigcup_{m\geq n} \Bun_{2,0}^m.
  \end{equation*}
  In particular, the point $E_0$ is dense in $\Bun_{2,0}$.
\end{ej}

\section{Taking quotients in geometry} \label{ss:quotients}
The jumping phenomenon implies that a space classifying all holomorphic vector bundles cannot be separated. The main reason why this happens is that we are trying to endow the set of isomorphism classes of holomorphic vector bundles with geometric structure by regarding it as a quotient of a space under the action of some ``geometric'' group. The non-separatedness of the space then arises from an easy fact in topology: if a topological group acts on a topological space, then non-closed orbits give rise to non-separated phenomena in the quotient space. We can illustrate this with an example.

\begin{ex}
  Consider the action of $\C^*$ on $\C^2$ defined as follows
  \begin{equation*}
t \cdot (z,w) = (t z, t^{-1} w).
  \end{equation*}
  The orbits of this action are the family of conics
  \begin{equation*}
O_a = \left\{(z,w)\in \C : zw=a \right\}, \text{ for } a \in \C^*,
  \end{equation*}
  plus the two punctured lines $O_{0_1}=\left\{ (z,0): z\in \C^* \right\}$ and $O_{0_2}=\left\{ (0,w): w\in \C^* \right\}$, and the fixed point $O_{0}=\left\{ (0,0) \right\}$. Note that the conics $C_a$ are closed, as well as the point $(0,0)$. On the other hand, the two lines $O_{0_1}$ and $O_{0_2}$ are not closed, and both have the point $(0,0)$ in their closure. The map $a\mapsto O_{a}$, for $a\in \C \cup \left\{0_1,0_2 \right\}$, determines a homeomorphism of the orbit set $\C^2/\C^*$, endowed with the quotient topology, with the complex line with three origins, which is a non-separated topological space. However, if we restrict to the closed orbits, the map $a\mapsto O_a$ determines a bijection of the set of closed orbits and the complex line $\C$.
  \end{ex}

  There are two ways of constructing ``nice'' quotients that are relevant from us. One is coming from algebraic geometry and the other from symplectic geometry. We closely follow Hoskins' notes \cite{Hoskins}.

  \subsection*{Quotients in algebraic geometry. Geometric invariant theory}
  Recall that an affine $\C$-scheme $Y$ of finite type is by definition the spectrum of a finitely generated $\C$-algebra $A$, that is $Y=\Spec A$. If $G$ is a complex algebraic group acting on $Y$, then there is a naturally induced $G$-action on $A$ and, if the group $G$ is \emph{reductive} (i.e. if it is the complexification of a compact Lie group), then by Nagata's theorem the invariant subalgebra $A^G$ is also finitely generated. The inclusion $A^G\hookrightarrow A$ determines a morphism of affine varieties
  \begin{equation*}
 Y \rightarrow Y \git G:= \Spec A^G,
  \end{equation*}
  called the \emph{affine GIT quotient}. The affine GIT quotient is a categorical quotient in the sense that it satisfies a natural universal property in the category of $\C$-schemes.

  We are also interested in considering actions of reductive groups on projective schemes. Recall that if $Y\subset \bbP^n$ is a projective variety, then we can consider its homogeneous coordinate ring, which is a finitely generated graded $\C$-algebra $R=\oplus_{r\geq 0} R_r$ with $R_0=\C$ and such that the generators lie in $R_1$. This algebra $R$ does not depend only on $X$, but also on the way it is embedded in $\bbP^n$. Conversely, the projective spectrum of such an algebra $R$ not only provides the projective scheme $Y=\Proj R$, but also a very ample line bundle $L=\sO(1)\rightarrow Y$. A pair $(Y,L)$ of a projective scheme with a very ample line bundle on it is called a \emph{polarized} projective scheme. Polarized projective schemes are thus in bijective correspondence with graded $\C$-algebras finitely generated in degree $1$.

  Let $Y$ be a projective scheme equipped with an action of a reductive group $G$. A \emph{linearization} of this action is a line bundle $L\rightarrow Y$ such that the action of $G$ lifts to $L$ in such a way that the projection is equivariant and the morphisms on the fibres $L_y \rightarrow L_{gy}$ are linear. Suppose that $L\rightarrow Y$ is an ample line bundle providing such a linearization, and consider the graded ring
  \begin{equation*}
R(Y,L) = \bigoplus_{r\geq 0} H^0(Y,L^{r}).
  \end{equation*}
  The group $G$ acts naturally on $R(Y,L)$ preserving the graded pieces, and we denote
  \begin{equation*}
Y\git_L G := \Proj R(Y,L)^G.
  \end{equation*}
Note that we still do not have a categorical quotient. Rather, what we have is a rational map $Y\dashrightarrow Y\git_L G$ induced by the inclusion $R(Y,L)^G\hookrightarrow R(Y,L)$, but undefined on the closed subscheme $V(R(Y,L)^G_+)\subset Y$. The points of this subscheme are said to be ($L$-)\emph{unstable}. This motivates the following definition.

\begin{defn}
  A point $y\in Y$ is ($L$-)\emph{semistable} if there exists $r>0$ and an invariant section $\sigma \in H^0(Y,L^{r})^G$ such that $\sigma(y)\neq 0$. The ($L$-)semistable points in $Y$ form the open subscheme
  \begin{equation*}
Y^{ss}_L = Y \setminus V(R(Y,L)^G_+).
  \end{equation*}
  The (projective) \emph{GIT quotient} of $Y$ by the action of $G$ with respect to $L$ is the natural morphism
  \begin{equation*}
Y^{ss}_L \longrightarrow Y\git_L G,
\end{equation*}
induced by the inclusion $R(Y,L)^G\hookrightarrow R(Y,L)$.
\end{defn}

It is a theorem of Mumford \cite{GIT} that the GIT quotient $Y^{ss}_L\rightarrow Y\git_L G$ is in fact a categorical quotient. In particular this implies that the $\C$-points of $Y\git_L G$ are in bijection with the $G$-orbits which are closed in $Y^{ss}_L$; if a point is in one of these orbits then we say that it is ($L$-)\emph{polystable}. These polystable points form a subset $Y^{ps}_L\subset Y^{ss}_L$, and we are saying that there is a bijection $Y^{ps}_L(\C)/G \cong Y\git_L G(\C)$. An open subset $Y^{s}_L\subset Y^{ss}_L$ is formed by ($L$-)\emph{stable} points, which are polystable points such that their orbit has dimension equal to the dimension of $G$. When restricted to $Y^s_L$, the GIT quotient is in fact a geometric quotient, in the topological sense.

\begin{ex} \label{ex:projective}
  Consider the group $G=\C^*$ acting on $Y=\bbP^n$ by
  \begin{equation*}
t \cdot (z_0:\dots:z_n) = (t^{-1}z_0:tz_1:\dots:tz_n).
  \end{equation*}
  The ample line bundle $L=\sO_{\bbP^n}(1)$ provides a linearization of this action, and we have
  \begin{equation*}
R(\bbP^n,\sO_{\bbP^n}(1))^{\C^*} = \bigoplus_{r\geq 0} \C[z_0,\dots,z_r]_r^{\C^*} = \C[z_0z_1,\dots,z_0z_n].
  \end{equation*}
  The semistable locus is $(\bbP^n)^{ss}=\bbP^n\setminus V(z_0z_1,\dots,z_0z_n)\cong \bbA^n \setminus \left\{ 0 \right\}$ and all semistable points are actually stable. The GIT quotient is then the natural quotient
  \begin{equation*}
(\bbP^n)^{ss}\cong \bbA^n \setminus \left\{ 0 \right\} \longrightarrow \bbP^n\git \C^* = \Proj \C[z_0z_1,\dots,z_0z_n] \cong (\bbA^n\setminus \left\{ 0 \right\})/\C^* = \bbP^{n-1}.
  \end{equation*}
\end{ex}

\begin{ex}
More generally, if $G$ is a reductive group acting linearly on a projective scheme $Y\subset \bbP^n$, we obtain a lift of this action to the affine cone $\tilde{Y}\subset \C^{n-1}$. Now, for any $y\in Y$, we can consider a non-zero lift $\tilde{y}\subset \tilde{Y}$, and we obtain the following topological criterion
\begin{enumerate}
\item $y$ is semistable if and only if $0\not \in \overline{G\cdot \tilde{y}}$; equivalently, $y$ is unstable if and only if $0 \in \overline{G\cdot \tilde{y}}$.
\item $y$ is polystable if and only if $G\cdot \tilde{y}$ is closed.
\item $y$ is stable if and only if $G\cdot \tilde{y}$ is closed and has dimension equal to the dimension of $G$.
\end{enumerate}

Suppose in particular that $G=\C^*$. The linear action of $\C^*$ on $V:=\C^{n-1}$ splits the vector space $V$ into a direct sum $V=\oplus_{i\in \Z} V_i$, where on each component $V_i$ the action is given as $t \cdot v=t^iv$. Now, for each $y\in Y$ we consider any non-zero lift $\tilde{y}\in V$ and the corresponding set of \emph{weights}
\begin{equation*}
P(y) = \left\{ i \in \Z : \text{ $\tilde{y}$ has a non-zero component in } V_i  \right\}.
\end{equation*}
Let $\mu(y)=\min P(y)$ denote the minimum of these weights.
\begin{enumerate}
\item If $\mu(y)>0$, then $\lim_{t\rightarrow 0} t\cdot \tilde{y}=0$, so $y$ is unstable.
  \item If $\mu(y)=0$ then we have two possible cases:
        \begin{enumerate}
          \item  $P(y)=\left\{ 0 \right\}$, in which case $\C^*\cdot \tilde{y}= \left\{  \tilde{y} \right\}$, and thus $y$ is (strictly) polystable,
        \item there are some $i\in P(y)$ with $i>0$, in which case the limit $\lim_{t\rightarrow 0}t \tilde{y}$ is the $V_0$-component of $\tilde{y}$, which is not in $G\cdot \tilde{y}$, and thus $y$ is strictly semistable.
        \end{enumerate}
\item if $\mu(y)<0$ then the orbit $\C^* \cdot \tilde{y}$ is closed and $1$-dimensional, so $y$ is stable.
\end{enumerate}
\end{ex}

The idea of the above example can be generalized. Consider an action of a reductive group $G$ on a projective scheme $Y$, with a linearization provided by an ample line bundle $L\rightarrow Y$. Given any $1$-parameter subgroup $\lambda:\C^*\rightarrow G$, for any point $y\in Y$ we can consider the limit $y_0 := \lim_{t\rightarrow 0} \lambda(t)\cdot y$. Since this $y_0$ is a fixed point of the induced $\C^*$-action, we obtain a $\C^*$-action on the fibre $L_{y_0}$, which is necessarily of the form $t \cdot s = t^n s$, for a unique $n\in \Z$. We denote this number by $\mu_L(y,\lambda)=n$, and call it the \emph{Hilbert-Mumford weight}.

\begin{thm}[Hilbert-Mumford criterion]
  A point $y\in Y$ is semistable (resp. stable) with respect to $L$ if and only if for all nontrivial $\lambda:\C^* \rightarrow G$, we have $\mu_L(y,\lambda) \leq 0$ (resp. $<0$).
\end{thm}

The Hilbert--Mumford criterion also gives us a way to characterize polystable points which are not stable. A polystable point is semistable, so we have $\mu_L(y,\lambda)\leq 0$ for all $\lambda:\C^* \rightarrow G$. Now, suppose that $\mu_L(y,\lambda)=0$ for some $\lambda$. Since the orbit $G\cdot y$ is closed, if we denote $y_0:=\lim_{t\rightarrow 0} \lambda(t)\cdot y$, there must be some $g\in G$ such that $g\cdot y = y_0$. Since $y_0$ is fixed under $\lambda(\C^*)$, the point $y$ is fixed under the 1-parameter subgroup $\lambda^g(\C^*):= g^{-1} \lambda(\C^*) g$.

\subsection*{Quotients in symplectic geometry}
Recall that a symplectic manifold is an even-dimensional (real) smooth manifold $Y$ equipped with a symplectic form $\omega \in \Omega^2(Y)$. This is a closed ($d\omega = 0$) and non-degenerate differential $2$-form. A symplectomorphism is a diffeomorphism preserving the symplectic form. Let $K$ be a Lie group acting on $Y$ by symplectomorphisms. We denote by $\fk$ the Lie algebra of $K$ and by $\fk^*$ its dual, which is naturally equipped with the coadjoint action.

A \emph{moment map} for the action of $K$ is a $K$-equivariant map $\mu:Y\rightarrow \fk^*$ such that, for every $a\in \fk^*$, we have
\begin{equation*}
d \mu_a = \omega(\vec{a},-).
\end{equation*}
Here, $\mu_a:Y\rightarrow \bbR$ denotes the map $\mu_a(x)=\langle \mu(x), a \rangle$, and $\vec{a}$ is the vector field on $Y$ defined as
\begin{equation*}
\vec{a}_x = \left. \frac{d}{dt}\right|_{t=0} \exp(t a) \cdot x.
\end{equation*}

\begin{ej} \label{ej:fubinistudy}
  Consider the unitary group $K=\U(n+1)$ acting on projective space $Y=\bbP^n$ through its standard action on $\C^{n+1}$. Recall that $\bbP^n$ is naturally symplectic, with the Fubini--Study form, which in complex coordinates $\bm{z}=(z_0,\dots,z_n)$ is written as
  \begin{equation*}
\omega_{[\bm{z}]} = \frac{i}{2} \left( \sum_j \frac{dz_j \wedge d\bar{z}_j}{\lVert \bm{z} \rVert^2} - \sum_{j,k} \frac{\bar{z}_j z_k dz_j \wedge d \bar{z}_k}{\lVert \bm{z} \rVert^4} \right).
  \end{equation*}
  This symplectic form is constructed from the standard Hermitian form $(-,-)$ on $\C^{n+1}$, so it is preserved under the $\U(n+1)$-action. \textit{Show that there is a moment map $\mu:\bbP^n \rightarrow \mathfrak{u}(n+1)^*$ for this action, satisfying}
    \begin{equation*}
\langle \mu([\bm{z}]) , a \rangle = \frac{\tr (\bm{z}^\dagger a \bm{z})}{ 2 i \lVert \bm{z} \rVert^2},
    \end{equation*}
    \textit{where $\bm{z}^\dagger$ denotes the conjugate transpose of $\bm{z}$.}
\end{ej}

When the action of $K$ admits a moment map, we can perform \emph{symplectic reduction}. Let $\eta \in \fk^*$ be a fixed point under the coadjoint action. We can then consider the ``level set'' $i:\mu^{-1}(\eta)\hookrightarrow Y$ and restrict the $K$-action to it. The \emph{symplectic quotient} is the natural quotient
\begin{equation*}
\pi: \mu^{-1}(\eta) \longrightarrow Y\git_\eta K:= \mu^{-1}(\eta)/K.
\end{equation*}

\begin{thm}[Marsden--Weinstein--Meyer]
  If the action of $K$ on $\mu^{-1}(\eta)$ is free and proper, then:
  \begin{enumerate}
\item the symplectic quotient $Y\git_\eta K$ is a smooth manifold of dimension $\dim Y - 2\dim K$,
\item there is a unique symplectic form $\bar{\omega}$ on $Y\git_\eta K$ such that $\pi^*\bar{\omega}=i^* \omega$.
  \end{enumerate}
\end{thm}

\begin{ej} \label{ej:fubinistudy2}
  Consider the action of $K=\U(1)$ on $Y=\bbP^n$ by
  \begin{equation*}
t\cdot (z_0:\dots:z_n) = (t^{-1}z_0:tz_1:\dots:tz_n).
  \end{equation*}
  We can regard this action as the composition of the action of Exercise \ref{ej:fubinistudy} with the homomorphism $\U(1)\rightarrow \U(n+1)$ given by $t \mapsto \diag(t^{-1},t,\dots,t)$. \textit{Show that this action admits the moment map}
  \begin{equation*}
\mu(z_0:\dots:z_n) = \frac{-\lvert z_0\rvert^2 + \lvert z_1\rvert^2 + \dots + \lvert z_n\rvert^2}{\lvert z_0\rvert^2 + \lvert z_1\rvert^2 + \dots + \lvert z_n\rvert^2}.
  \end{equation*}
 \textit{Conclude that $\mu^{-1}(0)$ is a $(2n-1)$-dimensional sphere and that the symplectic quotient $\mu^{-1}(0)/ \U(1)$ is isomorphic to $\bbP^{n-1}$.}
\end{ej}

The reader might have noticed that all the symplectic manifolds $(Y,\omega)$ that we deal with in this section are actually Kähler. This means that there exists a complex structure $J$ on $Y$ such that the bilinear form $g(-,-)=\omega(-,J-)$ is actually a Riemannian metric on $Y$. Equivalently, a Kähler manifold is a complex manifold $(Y,J)$ with a Hermitian metric $h$ on $Y$ (i.e. a Hermitian metric on $TY$) such that the associated $2$-form $\omega=\mathrm{Im}(h)$ is closed. The Fubini--Study form on $\bbP^n$ is obtained from the Fubini--Study Hermitian metric, which determines a Kähler structure on $\bbP^n$. As a consequence, all smooth projective varieties (or rather, their analytifications) are naturally Kähler manifolds. It turns out that symplectic quotients of Kähler manifolds actually inherit the Kähler structure.

\begin{prop}
Suppose $(Y,J,h)$ is a Kähler manifold with a Lie group $K$ acting on $Y$ preserving both the complex structure $J$ and the Hermitian metric $h$ and with moment map $\mu:Y\rightarrow \fk^*$. Let $\eta \in \fk^*$ be a fixed point under the coadjoint action and suppose that the action of $K$ on $\mu^{-1}(\eta)$ is free and proper. Then we obtain an induced complex structure $\bar{J}$ and an induced Hermitian metric $\bar{h}$ on the symplectic quotient $Y\git_\eta K = \mu^{-1}(\eta)/K$, with respect to which it is a Kähler manifold.
\end{prop}

  \subsection*{The Kempf-Ness theorem}
  Recall the action of $\C^*$ on $\bbP^n$ from Example \ref{ex:projective}. The restriction of this action to the unit circle $\U(1)\subset \C^*$ determines the symplectic action from Exercise \ref{ej:fubinistudy2}. Observe that the GIT quotient and the symplectic quotients coincide
  \begin{equation*}
\bbP^n\git \C^* \cong \mu^{-1}(0)/ \U(1) \cong \bbP^{n-1}.
  \end{equation*}
  It turns out that this is the general situation.

  Let $K$ be a compact Lie group and let $G=K^{\C}$ be its complexification, which is a complex reductive group. Suppose that $G$ acts linearly on a smooth projective variety $Y\subset \bbP^n$ through a representation $\rho:G\rightarrow \GL_{n+1}(\C)$. The analytification of $Y$ is a smooth Kähler manifold, with respect to the Fubini--Study metric; in an abuse of notation, we identify $Y$ with its analytification. The $G$-action induces an action of $K$ on $\bbP^n$ and, since $K$ is compact, we can choose coordinates on $\bbP^n$ so that $\rho$ restricts to a unitary representation $\rho:K \rightarrow \U(n+1)$, and thus the Fubini--Study metric on $Y$ is preserved by $K$. We can now define a moment map $\mu:Y\rightarrow \fk^*$ by composing the embedding $Y\hookrightarrow \bbP^n$ with the moment map $\bbP^n\rightarrow \mathfrak{u}(n+1)^*$ from Exercise \ref{ej:fubinistudy} and with the map $\mathfrak{u}(n+1)^* \rightarrow \fk^*$ induced by the representation $\rho:K\rightarrow \U(n+1)$.

  \begin{thm}[Kempf--Ness]
    For any point $y\in Y$, we have that
    \begin{enumerate}
\item $y$ is semistable if and only if $\overline{G\cdot y} \cap \mu^{-1}(0)\neq \varnothing$;
      \item $y$ is polystable if and only if $G\cdot y \cap \mu^{-1}(0)\neq \varnothing$ and, in that case, $G\cdot y \cap \mu^{-1}(0)=K\cdot y$. Therefore,
            \begin{equation*}
\mu^{-1}(0) \subset G \cdot \mu^{-1}(0) = Y^{ps} \subset Y^{ss}.
            \end{equation*}
    \end{enumerate}
            Moreover, the inclusion $\mu^{-1}(0)\subset Y^{ss}$ induces a biholomorphism
           \begin{equation*}
\mu^{-1}(0)/K \longrightarrow Y \git G,
           \end{equation*}
           identifying the symplectic quotient with (the analytification of) the GIT quotient.
  \end{thm}

\section{Algebraic construction of the moduli space} \label{ss:algebraicconstructionN}
Our first approach to the construction of the moduli space of holomorphic vector bundles of rank $r$ and degree $d$ on the compact Riemann surface $X$ is as a GIT quotient. In order to do so, we have to define (semi)stability conditions for vector bundles on $X$. We first state these conditions, and later will deduce them from the Hilbert--Mumford criterion. The moduli space will then be a projective variety $\cN_{r,d}$ obtained as a GIT quotient of a big space parametrizing all semistable bundles. The closed points of the moduli space will be in bijection with isomorphism classes of vector bundles which are polystable, and there will be an open subset $\cN_{r,d}^s\subset \cN_{r,d}$ parametrizing stable bundles. The fact that, when restricted to stable points, GIT quotients are geometric will imply that the space $\cN_{r,d}^{s}$ will satisfy a certain universal property, that of being a \emph{coarse moduli space}. We explain all these notions in the following.

\begin{defn}
	The \emph{slope} of a vector bundle $E$ on $X$ is the number
	\begin{equation*}
	\mu(E)= \deg E / \mathrm{rk}\ E.
	\end{equation*} 
A holomorphic vector bundle $\sE$ on $X$ is \emph{semistable} (resp. \emph{stable}) if and only if for every holomorphic subbundle $\sE'\subset \sE$, we have
 \begin{equation*}
\mu(\sE') \leq \mu(\sE) \text{ (resp. }<).
\end{equation*} 
We say that $\sE$ is \emph{polystable} if it is either stable or a direct sum of stable vector bundles of slope equal to $\mu(\sE)$.
\end{defn}

\begin{ej}
\textit{Show that semistability (resp. stability) for a vector bundle $\sE$ is equivalent to any of the following conditions:}
\begin{itemize}
  \item For every proper quotient bundle $\sE\rightarrow \sE'$, we have
        \begin{equation*}
\mu(\sE') \geq \mu(\sE) \text{ (resp. } \mu(\sE')> \mu(\sE)).
        \end{equation*}
  \item For every proper subsheaf $\sF\hookrightarrow \sE$, we have
        \begin{equation*}
\mu(\sF) \leq \mu(\sE) \text{ (resp. } \mu(\sF)< \mu(\sE)).
        \end{equation*}
  \item For every proper quotient sheaf $\sE\rightarrow \sF$, we have
        \begin{equation*}
\mu(\sF) \geq \mu(\sE) \text{ (resp. } \mu(\sF)> \mu(\sE)).
        \end{equation*}
\end{itemize}
\end{ej}

\begin{ej}
\textit{Prove the following:}
  \begin{itemize}
    \item Every holomorphic line bundle on $X$ is stable.
    \item If $\rk \sE$ and $\deg \sE$ are coprime, then $\sE$ is semistable if and only if it is stable.
    \item $\sE$ is stable if and only if its dual is.
          \item For any holomorphic line bundle $\sL$ on $X$, $\sE$ is (semi)stable if and only if $\sE \otimes \sL$ is.
  \end{itemize}
\end{ej}

We let $\Bun^{s}_{r,d}$ denote the set of isomorphism classes of stable holomorphic vector bundles of rank $r$ and degree $d$ on $X$, and let $\mathsf{Bun}^s_{r,d}$ denote the moduli problem for this set.

\begin{thm}[Seshadri]
There exists a projective variety $\mathcal{N}_{r,d}$, the \emph{moduli space of semistable vector bundles of rank $r$ and degree $d$ on $X$}, such that:
\begin{enumerate}
\item The set of closed points $\mathcal{N}_{r,d}(\C)$ is in natural bijection with the set of isomorphism classes of polystable holomorphic vector bundles of rank $r$ and degree $d$ on $X$.
\item There is a Zariski open subvariety $\mathcal{N}^s_{r,d}\subset \mathcal{N}_{r,d}$ which is a coarse moduli space for the moduli problem $\mathsf{Bun}^s_{r,d}$.
\item If $r$ and $d$ are coprime, then $\cN_{r,d}=\cN_{r,d}^s$ is a fine moduli space for $\mathsf{Bun}^s_{r,d}$. In particular, there is a \emph{universal vector bundle} $\mathscr{U}_{r,d}\rightarrow X \times \cN_{r,d}$ from which any flat family of stable vector bundles of rank $r$ and degree $d$ on $X$ arises as pullback.
\end{enumerate}
\end{thm}

We can give a rough outline of the proof, and in turn explain how the stability conditions arise from GIT. As we have already mentioned, the main idea for the construction of $\mathcal{N}_{r,d}$ is to obtain it as a GIT quotient.

\subsection*{Step 1: Bounded families}
Recall that we can regard $X$ as the analytification of a smooth complex projective curve. This amounts to find a relatively ample line bundle on $X$, that we denote by $\sO_X(1)$. As usual for any integer $n$, we denote $\sO_X(n)=\sO_X(1)^{\otimes n}$ and, for every holomorphic vector bundle $\sE$ on $X$, we put $\sE(n)=\sE \otimes \sO_X(n)$. The \emph{Hilbert polynomial of $\sE$} is given by
\begin{equation*}
n \mapsto P_{\sE}(n):= \chi(\sE(n))= d + r(n+1-g).
\end{equation*}
For any such $\sE$, there exists some integer $n_{\sE}>0$, called the \emph{bound of $\sE$}, such that, for every $n\geq n_{\sE}$, the bundle $\sE(n)$ is generated by global sections and $H^1(X,\sE(n))=0$. In particular, this implies that, for $n\geq n_{\sE}$, we have $P_{\sE}(n)=h^0(\sE(n))$.
A family $F$ of holomorphic vector bundles on $X$ is \emph{bounded} if the set $\left\{ n_{\sE}: \sE \in F \right\}$ is bounded above. An upper bound for this set is called a \emph{bound of $F$}.

If we want to construct a moduli space for a family $F$ as the GIT quotient of some projective variety, the family $F$ ought to be bounded. Luckily for us, the family $F_{r,d}^{ss}$ of semistable holomorphic vector bundles on $X$ with rank $r$ and degree $d$ is in fact bounded. Let us take any bound $n$ of this family and consider
\begin{equation*}
N:= P_{\sE}(n)= h^0(\sE(n)) = d + r(n+1-g),
\end{equation*}
where $\sE$ is any element of $F_{r,d}^{ss}$.

\subsection*{Step 2: The Quot scheme}
Since $F_{r,d}^{ss}$ is bounded by $n$, for every $\sE \in F_{r,d}^{ss}$ the bundle $\sE(n)$ is generated by global sections so we can put it as a quotient $\sO_X^N \rightarrow \sE(n) \rightarrow 0$. Equivalently, $\sE$ arises as a quotient $\sV:=\sO_X(-n)^N \overset{q}{\rightarrow} \sE \rightarrow 0$. Consider the polynomial $P(m)=d + r(m+1-g)$ and let $Q^P_{\sV}$ denote the set of quotients $\sV \overset{q}{\rightarrow} \sF \rightarrow 0$ such that the resulting coherent sheaf $\sF$ has Hilbert polynomial $P_{\sF}=P$. The set $Q^P_{\sV}$ determines a moduli problem, that we denote by $\mathsf{Quot}^P_{\sV}$.
We refer the reader to \cite{Nitsure_Quot} for a proof of the following fact.

\begin{thm}[Grothendieck]
The moduli problem $\mathsf{Quot}^P_{\sV}$ admits a fine moduli space. This space is a projective variety $\mathrm{Quot}^P_{\sV}$ called the \emph{Quot scheme}.
\end{thm}

\subsection*{Step 3: The GIT quotient}
We indentify now the open subscheme $Y\subset \mathrm{Quot}^P_{\sV}$ determined by quotients $\sV \overset{q}{\rightarrow} \sE \rightarrow 0$ such that $\sE$ is locally free. In particular, (the set of closed points of) this scheme $Y$ contains all the elements of the family $F_{r,d}^{ss}$. Base change on $\sV$ induces an action of $G=\SL_N(\C)$ on $Y$, and the isomorphism class of a bundle $\sE$ in $Y$ depends only on its $G$-orbit on $Y$. An ample linearization $\sL_m \rightarrow Y$ for this $G$-action is determined by any $m\geq n$ and by an standard ``Plücker embedding'' of $Y$ in a certain Grassmannian. Equivalently, $\sL_m \rightarrow Y$ is the ``determinant bundle'' whose fibre over a quotient $\sV \overset{q}{\rightarrow} \sE \rightarrow 0$ is the vector space \[\sL_{m,q}=\det H^0(X,\sE(m))^\vee \otimes \det H^1(X,\sE(m)).\]
If we take $m$ large enough, we can assume that $\sL_{m,q}=\det H^0(X,\sE(m))^\vee$.
The moduli space $\cN_{r,d}$ is finally obtained as the GIT quotient
\begin{equation*}
\cN_{r,d} := Y \git_{\sL_m} G.
\end{equation*}

\subsection*{Step 4: Stability conditions from the Hilbert--Mumford criterion}
Consider now a $1$-parameter subgroup $\lambda:\C^* \rightarrow G=\SL_N(\C)$. This $\lambda$ has some associated weights $a_1>\dots >a_s$, and determines a weight decomposition of $V:=\C^N$ as
\begin{equation*}
V = \oplus_{i=1}^s V_{i}.
\end{equation*}
We denote $N_i=\dim_\C V_i$, so $\sum_{i=1}^s N_i = N$, and remark the fact that, for $\lambda$ to determine a $1$-parameter subgroup of $\SL_N(\C)$ we must have $\sum_{i=1}^s N_i a_i=0$.
We can write the weight decomposition as a filtration $0\subset F_1 \subset \dots \subset F_{s-1} \subset F_s=V$, by putting $F_k  = \oplus_{i=1}^k V_i$.
Each of these $F_i$ determine a subbundle $\sE_i = \mathrm{im}(F_i \otimes \sO_X(-n)\rightarrow \sE) \subset \sE$. Hence, we obtain a filtration \[0\subset \sE_1 \subset \dots \subset \sE_{s-1}\subset \sE_s=\sE.\]
The corresponding graded pieces are denoted by $\sG_i = \sE_i/\sE_{i-1}$.

The $1$-parameter subgroup acts with weight $a_i$ on each of the cohomology spaces $H^k(X,\sG_{i}(m))$, so it is not hard to show that the Hilbert--Mumford weight is the number
\begin{equation} \label{eq:HilbertMumford}
\mu_{\sL_m}(\sE,\lambda) = -\sum_{i=1}^s a_i \chi(\sG_{i}(m)) = \sum_{i=1}^{s-1}(a_{i+1}-a_i) \left(\chi(\sE_i(m)) - \frac{\dim_\C F_i}{N} \chi(\sE(m))\right).
\end{equation}

\begin{ej}
\textit{Show the second equality in the formula above}. \textbf{Hint}: Use the relations $\dim_\C F_i = N_1 + \dots + N_i$, $N=\sum_{i=1}^sN_i$, $\chi(\sG_i)=\chi(\sE_i)-\chi(\sE_{i-1})$ and $\sum_{i=1}^s N_i a_i = 0$.
\end{ej}

\begin{prop}
  An element $\sV \overset{q}{\rightarrow} \sE \rightarrow 0$ of $Y$ is ($\sL_m$-)semistable if and only if for every subspace $0\neq V' \subsetneq V$, we have the inequality
  \begin{equation*}
\frac{\dim_\C V'}{\chi(\sE'(m))} \leq \frac{N}{\chi(\sE(m))},
  \end{equation*}
  where $\sE'=\mathrm{im}(V' \otimes \sO_X(-n)\rightarrow \sE)$.
\end{prop}

\begin{proof}
We can conveniently rewrite the inequality as
  \begin{equation*}
\chi(\sE'(m))-\frac{\dim_\C V'}{N}\chi(\sE(m)) \geq 0.
  \end{equation*}
  In particular, if this inequality is satisfied for every $V'$, then every term of the form $\chi(\sE_i(m)) - \frac{\dim_\C F_i}{N} \chi(\sE(m))$ in equation \eqref{eq:HilbertMumford} is $\geq 0$ and, since the $a_{i+1}-a_i$ are negative, we have $\mu_{\sL_m}(\sE,\lambda)\leq 0$. Conversely, if there is a $V'$ which does not satisfy the inequality, we can construct a $1$-parameter subgroup $\lambda$ determining the filtration $0\subset V' \subset V$, with weights $a_1>a_2$, for which we immediately see that
\begin{equation*}
\mu_{\sL_m}(\sE,\lambda) = (a_{2}-a_1) \left(\chi(\sE'(m)) - \frac{\dim_\C V'}{N} \chi(\sE(m))\right) > 0,
\end{equation*}
so $\sV \overset{q}{\rightarrow} \sE \rightarrow 0$ is unstable.
\end{proof}

The key now is to note that we can take $m$ to be arbitrarily large. Indeed, if we take $m$ large enough, for every proper subbundle $\sE'\subset \sE$ determined by a subspace $V'\subset V$ the numbers $\chi(\sE'(m))$ and $\chi(\sE(m))$ are positive, and thus we can multiply in the inequality from the previous proposition to obtain an inequality
\begin{equation*}
N' r  m + N' (d + r(1-g)) \leq N r' m + N(d' + r' (1-g)),
\end{equation*}
where we are denoting $r'=\mathrm{rank}(\sE')$, $d'=\deg(\sE')$ and $N'=\dim_\C V'$. We also denote by $\mu'=r'/d'$ the slope of $\sE'$. Since this equality is satisfied for $m$ arbitrarily large, we have an inequality on the leading terms
\begin{equation*}
N'/r' \leq N/r.
\end{equation*}

Finally, we remark that we can also take the upper bound $n$ to be arbitrarily large. If we do that, then we can assume that $N'=\chi(\sE'(n))= r'n + d' + r'(1-g)$, while by definition we have $N=rn+d+r(1-g)$. Therefore, the above inequality becomes
\begin{equation*}
n + \mu' + (1-g) \leq n + \mu + (1-g),
\end{equation*}
from where we deduce that $\mu' \leq \mu$. We conclude the following.

\begin{thm}[Seshadri]
An element $\sV \overset{q}{\rightarrow} \sE \rightarrow 0$ of $Y$ is ($\sL_m$-)semistable (resp. stable, polystable) if and only if the vector bundle $\sE$ is semistable (resp. stable, polystable).
\end{thm}

\section{The moduli space as a symplectic quotient}
The moduli space of holomorphic vector bundles can also be constructed as the symplectic reduction of some infinite-dimensional complex vector space under the action of an infinite-dimensional Lie group. Although dealing with the analytical technicalities is beyond the scope of this paper, we remark that the theory of symplectic and Kähler quotients developed in Section \ref{ss:quotients} can be generalized to an infinite-dimensional setting. In order to do so precisely, one needs to construct $L^2$-completions of the spaces we deal with in this section, and then find estimates that provide the regularity of the solutions obtained. As we say, in this paper we will ignore many of these technicalities, and refer to Kobayashi's book \cite{Kobayashi} for more details.

Let us start by consdering a smooth vector bundle $E$ of rank $r$ and degree $d$. We denote $\mu=\mu(E)=d/r$. We denote by $\cA_E$ the space of connections on  $E$. This is an affine space modelled over the (infinite-dimensional) complex vector space $\Omega^1(X,\End E)$. If $H$ is a Hermitian metric on $E$, we can consider the subspace $\cA_{E,H}\subset \cA_E$ of connections which are $H$-unitary. This is an affine subspace, modelled by the vector space $\Omega^1(X,\mathfrak{u}_H E)$, where $\mathfrak{u}_H E$ is the subspace of endomorphisms of $E$ which are skew-Hermitian (that is $f^\dagger = -f$) with respect to the metric $H$. The space $\Omega^1(X,\fu_H E)$ admits a non-degenerate skew-symmetric form
\begin{equation*}
\omega(A, B) = -\int_X \mathrm{tr}(A\wedge B),
\end{equation*}
which endows it with the structure of an infinite dimensional symplectic manifold.

The space  $\cA_{E,H}$ is admits a symplectic action by the \emph{unitary gauge group} $\cG_{E,H}=\Omega^0(X,U_H E)$, where $U_H E$ is the subgroup of  $H$-unitary automorphisms.
This is an infinite-dimensional Lie group, whose Lie algebra is the infinite dimensional vector space
\begin{equation*}
\Lie \cG_{E,H} = \Omega^0(X,\fu_H E).
\end{equation*}
The dual Lie algebra is isomorphic to $\Lie \cG_{E,H}^*=\Omega^2(X,\fu_H E)$ with the duality pairing induced by integration and the Killing form, that is
\begin{equation*}
(a,\alpha) = \int_X \mathrm{tr}(a \alpha),
\end{equation*}
for $a\in \Omega^0(X,\fu_H E)$ and $\alpha \in \Omega^2(X,\fu_H E)$.

\begin{ej}
Consider the map
\begin{align*}
\mu:\cA_{E,H}  \rightarrow  \Omega^2(X,\mathfrak{u}_H E),\ \nabla  \mapsto F_\nabla
\end{align*}
sending each unitary connection to its curvature. \textit{Show that this map is a moment map for the action of $\cG_{E,H}$ on $\cA_{E,H}$.} \textbf{Hint}: For $a\in \Omega^0(X,\fu_H E)$, compute
\begin{equation*}
\vec{a}_\nabla = \left. \frac{d}{dt} \right|_{t=0} \exp(ta) \nabla \exp(-ta) = -\nabla a
\end{equation*}
and, for $A\in T_\nabla \cA_{E,H}= \Omega^1(X,\fu_H E)$, compute
\begin{equation*}
d_\nabla \mu (A) = \left. \frac{d}{dt} \right|_{t=0} \mu(\nabla + tA) = -\nabla A.
\end{equation*}
Conclude that
\begin{equation*}
\langle d_\nabla \mu(A), a \rangle = \omega(A,\vec{a}).
\end{equation*}
\end{ej}

Let us fix now a volume form $\omega_X \in\Omega^2(X)$, with $\int_X \omega_X = 1$. The $\fu_H E$-valued $2$-form
\begin{equation*}
\alpha_E := -2\pi i \mu \id_E \omega_X
\end{equation*}
is invariant under the coadjoint action of $\cG_{E,H}$. Therefore, we can consider the preimage $\mu^{-1}(\alpha_E)$. A unitary connection $\nabla$ lies in $\mu^{-1}(\alpha_E)$ if and only if it is a connection of constant central curvature. Recall that this means that
\begin{equation*}
F_\nabla = c \id_{\cE} \omega_X,
\end{equation*}
for some constant $c\in \C^*$.

\begin{ej}
\textit{Using Chern-Weil theory show that, for any such $\nabla$, the constant $c$ must be equal to $-2\pi i \mu$.}
\end{ej}

Recall that, if we fix two points $x_0,x_1 \in X$, the holonomy of a connection of constant central curvature determines a representation
\begin{equation*}
\rho: \pi_1(X\setminus \left\{ x_1 \right\},x_0) \longrightarrow \U(r)
\end{equation*}
mapping the class of a contractible loop in $X$ around $x_1$ and based in $x_0$ to the element $\exp(-2\pi i \mu)$. We conclude the following.

\begin{thm}
The symplectic quotient $\mu^{-1}(\alpha_E)/\cG_{E,H}$ is in natural bijection with the \emph{$d$-twisted $\U(r)$-character variety} (see Section \ref{ss:charactervarieties})
\begin{align*}
  \frac{\left\{(A_1,\dots,A_g, B_1,\dots, B_g, Z) \in \U(r)^{2g+1}: \prod_{i=1}^g [A_i,B_i]=Z, Z=e^{-\tfrac{2\pi id}{r}}I_r \right\}}{\U(r)},
\end{align*}
where $\U(r)$ acts by conjugation. We denote this character variety by $\mathcal{X}^d_{\U(r)}$.
\end{thm}

\section{Deformation theory} \label{ss:deformation}
One can show that the natural bijection $\mu^{-1}(\alpha_E)/\cG_{E,H} \cong \mathcal{X}^d_{\U(r)}$ can in fact be upgraded to an homeomorphism, which restricts to a diffeomorphism on the smooth points. We explain the main idea behind this fact by constructing the tangent spaces of both spaces. We also compare these deformations with deformations of holomorphic vector bundles.

\subsection*{Deformations of representations}
To understand the tangent space to the character variety $\mathcal{X}^d_{\U(r)}$ we can consider the more general case where $\rho:\Pi \rightarrow G$ is a homomorphism from a discrete finitely presented group $\Pi$ to a Lie group $G$.
An element of the tangent space $T_\rho \Hom(\Pi,G)$ is obtained by considering a smooth family of representations $\rho_t$, $t\in \bbR$, with $\rho_0=\rho$ and differentiating at $t=0$. This yields a map $\dot{\rho}: \Pi \rightarrow \fg$, where $\fg$ is the Lie algebra of $G$. The homomorphism condition on $\rho$ induces the following property on $\dot{\rho}$:
\begin{equation*}
\dot{\rho} (\sigma_1 \sigma_2) = \dot{\rho}(\sigma_1) + \Ad_{\rho(\sigma_1)} (\dot{\rho}(\sigma_2)).
\end{equation*}
This means that $\dot{\rho}$ is a $1$-cocycle $Z^1_\rho(\Pi,\fg)$ for the \emph{group cohomology theory} associated with the representation of $\Pi$ on $\fg$ defined by $\sigma \mapsto \Ad_{\rho(\sigma)}$.
We can also compute, for $v \in \fg$
\begin{equation*}
\frac{d}{dt}|_{t=0} \exp(t v) \rho(\sigma) \exp(-tv) = \Ad_{\rho(\sigma)}(v)- v.
\end{equation*}
This determines a map $d: \fg \rightarrow Z^1_\rho (\Pi,\fg)$, by putting $d(v)(\sigma) = \Ad_{\rho(\sigma)}(v)- v$. The cokernel of this map is the \emph{group cohomology group} $H^1_\rho(\Pi,\fg)$. We can then consider the quotient space \[\cX_{\Pi,G} = \Hom(\Pi,G)/G\] by the conjugation action of $G$. Note that the centre $Z_G$ of $G$ acts trivially, so we actually have $\cX_{\Pi,G}=\Hom(\Pi,G)/G^{\ad}$, for $G^{\ad}=G/Z_G$.

The ``expected dimension'' of $\cX_{\Pi,G}$ is then
\begin{align*}
\dim'  \cX_{\Pi,G} = \dim \Hom(\Pi,G) - \dim G^{\ad}=\dim \Hom(\Pi,G) - \dim G + \dim Z_G.
\end{align*}
On the other hand, the ``expected tangent space'' of $\cX_{\Pi,G}$ at the equivalence class of a representation $\rho$ is the group cohomology space
\begin{equation*}
T_{\rho} \cX_{\Pi,G} = H^1_\rho(\Pi,\fg).
\end{equation*}
We compute now the dimension of $H^1_\rho(\Pi,\fg)$. The presentation of $\Pi$ determines a resolution of $\Z$ as a $\Z[\Pi]$ module that induces the formula
\begin{equation*}
\sum_{i=0}^\infty (-1)^i \dim H^i_\rho(\Pi,\fg) = (1-m(\Pi)) \dim G,
\end{equation*}
where $m(\Pi)$ is the difference between the number of generators and the number of relations of $\Pi$.

For our particular choice of group $\Pi=\pi_1(X \setminus \left\{ x_1 \right\},x_0)$, if $\rho$ maps the class of a contractible loop in $X$ around $x_1$ to a central element of $G$, then this class acts trivially on $\fg$ through the adjoint action, so we can just take $\Pi=\pi_1(X,x_0)$. Since $X$ is a $K(\Pi,1)$, the group cohomology $H_\rho^*(\Pi,\fg)$ actually coincides with the ordinary local-system valued cohomology $H^*(X,\fg_\rho)$, where $\fg_\rho$ is the local system on $X$ induced by the representation $\Ad \circ \rho:\pi_1(X)\rightarrow \GL(\fg)$. Therefore, we have $H^i_\rho(\Pi,\fg)=0$ for $i\geq 3$ and
\begin{equation*}
\dim H^1_\rho(\Pi,\fg) = \dim H^0_\rho(\Pi,\fg) + \dim H^2_\rho (\Pi,\fg) + 2(g-1) \dim G.
\end{equation*}
Since the adjoint representation is self-dual, Poincaré duality gives an isomorphism
\begin{equation*}
H^0(X,\fg_\rho) \cong H^2(X,\fg_\rho)
\end{equation*}
so in fact we have
\begin{equation*}
\dim H^1_\rho(\Pi,\fg) = 2[\dim H^0_\rho(\Pi,\fg) + (g-1) \dim G].
\end{equation*}
The space $H^0_\rho(\Pi,\fg)$ is the subset of elements $z \in \fg$ such that $z=\Ad_{\rho(\sigma)}(z)$ for every $\sigma \in \Pi$. This space is the infinitesimal stabilizer $\mathfrak{z}_\rho$, i.e. the Lie algebra of the stabilizer $G_\rho$ of the representation $\rho$. We conclude
\begin{equation*}
\dim T_\rho \cX_{\Pi,G} = 2[ \dim \mathfrak{z}_\rho + (g-1) \dim G].
\end{equation*}

The dimension of $\Hom(\Pi,G)$ is the rank as a $G$-module of the kernel of the map
\begin{align*}
  G^{2g}  \longrightarrow G,\ (A_1,\dots,A_g,B_1,\dots,B_g)   \longmapsto \prod_{i=1}^g [A_i,B_i].
\end{align*}
The rank of this map is actually $\dim G^{\ad}=\dim G - \dim Z_G$, so
\begin{equation*}
\dim \Hom(\Pi,G) = (2g-1) \dim G + \dim Z_G,
\end{equation*}
and thus the expected dimension of $\cX_{\Pi,G}$ is
\begin{equation*}
\dim' \cX_{\Pi,G}=2[\dim Z_G + (g-1)\dim G ].
\end{equation*}
For $\cX_{\Pi,G}$  to be smooth at a point $[\rho]$, the dimension of $H^1_\rho(\Pi,\fg)$ must coincide with the expected dimension $\dim' \cX_{\Pi,G}$. We say that a representation is \emph{infinitesimally simple} if the infinitesimal stabilizer $\mathfrak{z}_\rho$ is isomorphic to the centre $\mathfrak{z}\subset \fg$ of the Lie algebra. For such a representation, we have
\begin{equation*}
\dim T_\rho \cX_{\Pi,G} = 2 [\dim \mathfrak{z} + (g-1) \dim G] = \dim' \cX_{\Pi,G}.
\end{equation*}
In particular, we say that $\rho$ is \emph{simple} if $G_\rho=Z(G)$. Near a simple class $[\rho]$, we obtain a local model for $\cX_{\Pi,G}$ as the space $H^1_\rho(\Pi,\fg)$.

The real dimension of $\U(r)$ is equal to the complex dimension of $\GL_r(\C)$, which is equal to $r^2$. Moreover, $\U(r)$ has the $1$-dimensional center $\U(1)$. We obtain a first formula for the dimension of $\cX_{\U(r)}$
\begin{equation*}
\dim \cX_{\U(r)} = 2[(g-1)r^2 + 1].
\end{equation*}

\subsection*{Deformations of flat unitary connections}
We consider now a connection $\nabla \in \mu^{-1}(\alpha_E)$ and study the tangent space at $[\nabla]$, the gauge equivalence class of $A$, of the space $\mu^{-1}(\alpha_E)/\cG_{E,H}$. We do this by considering the \emph{gauge complex}
\begin{center}
  \begin{tikzcd}
0 \rar & \Omega^0(X, \fu_H E) \rar{\nabla } & \Omega^1(X,\fu_H E) \rar{\nabla } & \Omega^2(X,\fu_H E) \rar & 0,
  \end{tikzcd}
\end{center}
and let $H^i_\nabla$ be its cohomology groups. For an infinitesimal deformation in the direction $A\in \Omega^1(X,\fu_H E)$ to stay tangent to $\mu^{-1}(\alpha_E)$, we must have $\nabla A = 0$, so the tangent space of $\mu^{-1}(\alpha_E)$ is identified  with the space of $1$-cocycles of the above complex. Moreover the first map of the complex encodes the infinitesimal action of $\cG_{E,H}$, so we can precisely identify the tangent space $T_\nabla (\mu^{-1}(\alpha_E)/\cG_{E,H})$ with the cohomology group $H^1_\nabla$.
Using the basic theory of elliptic complexes, one can show that in fact $H^0_\nabla \cong H^2_\nabla$ are $1$-dimensional precisely when $\nabla$ is an irreducible connection (and thus corresponds to a simple representation). Moreover, one can use the Atiyah--Singer index formula to compute
\begin{equation*}
\dim H^1_\nabla = 2[(g-1)r^2 + 1].
\end{equation*}
See \cite{AtiyahBott} for more details.

We can conclude the following.

\begin{corol}
The open dense subspace $\mu^{-1}(\alpha_E)^s/\cG_{E,H}\subset \mu^{-1}(\alpha_E)/\cG_{E,H}$ parametrizing gauge equivalence classes of connections of constant central curvature which are irreducible is a Kähler manifold of complex dimension $(g-1)r^2 +1$, diffeomorphic to the open subspace $(\cX^d_{\U(r)})^s\subset (\cX^d_{\U(r)})^s$ of the twisted character variety parametrizing conjugacy classes of simple representations.
\end{corol}

\subsection*{Deformations of holomorphic structures}
Recall that the space $\cC_E$ of holomorphic structures on $E$ is an affine space modelled by the vector space $\Omega^{0,1}(X,\End E)$, so we can identify this vector space with the tangent space of a holomorphic structure $\delbar_{\sE}$. The action of the complex gauge group $\cG_E^\C=\Omega^0(X,\Aut E)$ induces the infinitesimal action
\begin{align*}
  \Omega^0(X,\End E) & \longrightarrow \Omega^{0,1}(X,\End E)  \\
  \alpha  & \longmapsto \delbar_{\sE} \alpha.
\end{align*}
This allows us to identify the (holomorphic) tangent space to the moduli space $\cN_{r,d}$ at the class $[\sE]$ of a polystable holomorphic bundle as
\begin{equation*}
\bm{T}_{[\sE]} \cN_{r,d} = H^1(X, \End \sE).
\end{equation*}
We can compute this cohomology group using Riemann--Roch
\begin{equation*}
\dim H^0(X,\End \sE)- \dim H^1(X,\End \sE) = (1-g)\mathrm{rk}(\End \sE) + \deg (\End \sE).
\end{equation*}
Now, note that the rank of $\End \sE$ is $r^2$ and that $\End \sE$ has trivial degree. Hence,
\begin{equation*}
\dim_{\C} \bm{T}_{[\sE]} \cN_{r,d} = (g-1)r^2 + \dim H^0(X,\End \sE).
\end{equation*}

\begin{ej} \label{ej:stablesimple}
A holomorphic vector bundle $\sE$ is \emph{simple} if and only if \[H^0(X,\End \sE)\cong \C.\] \textit{Show that a polystable holomorphic vector bundle $\sE$ is simple if and only if it is stable}.
\end{ej}

\begin{corol}
The Zariski open subvariety $\cN^s_{r,d}\subset \cN_{r,d}$ is a smooth complex quasi-projective variety of (complex) dimension $(g-1)r^2+1$.
\end{corol}

\section{The theorem of Narasimhan--Seshadri} \label{ss:Narasimhan-Seshadri}

\subsection*{Hermitian-Einstein metrics}
Recall that there is a unique unitary connection $\nabla_H$, the Chern connection, on a Hermitian holomorphic vector bundle $(\sE,H)$ such that $\delbar_{\sE}$ is recovered as $\nabla_H^{0,1}$.
\begin{defn}
A \emph{Hermitian--Einstein} metric (HE metric) on a holomorphic vector bundle $\sE$ is a Hermitian metric  $H$ on $\sE$ such that its Chern connection $\nabla_H$ has constant central curvature; that is, such that
\begin{equation*}
F_H = -2\pi i \mu(\sE) \id_{\sE} \omega_X,
\end{equation*} 
for $F_H=F_{\nabla_H}$.
\end{defn}

\begin{prop} \label{prop:HEimpliesPS}
If $\sE$ admits a HE metric then it is polystable.
\end{prop}

\begin{proof}
Suppose that $\sE'\subset \sE$ is a holomorphic subbundle of $\sE$ and consider the quotient $\sE''=\sE/\sE'$. We can write
 \begin{equation*}
\nabla_H = 
\begin{pmatrix}
	\nabla' & \beta \\
	-\beta^\dagger & \nabla''
\end{pmatrix}.
\end{equation*} 
Here, $\nabla'$ and $\nabla''$ are the restriction and the projection of $\nabla_H$ to $\sE'$ and  $\sE''$, respectively, while $\beta\in \Omega^{0,1}(X,\mathrm{Hom}(\sE'',\sE'))$ is a representative of the class of $\sE$ as extension of $\sE''$ by $\sE'$. In particular, if $\beta=0$, then  $\sE=\sE'\oplus \sE''$. The form $\beta^\dagger \in \Omega^{1,0}(X,\mathrm{Hom}(\sE',\sE''))$ is just the conjugate transpose of $\beta$.

Now, the top left element of  $F_H$ is  $F_{\nabla'}-\beta \wedge \beta^\dagger$. Taking traces, integrating and multiplying by $\tfrac{i}{2\pi}$, we obtain
\begin{equation*}
c\frac{i}{2\pi} \mathrm{rk}\ \sE'  = \frac{i}{2\pi} \int_X \mathrm{tr}\ F_{\nabla'} + \lVert \beta \rVert^2.
\end{equation*} 
From here, we get
\begin{equation*}
\mu(\sE) = \mu(\sE') + C\lVert \beta \rVert^2,
\end{equation*} 
for some constant $C>0$. Therefore,  $\mu(\sE)\geq \mu(\sE')$, with equality if and only if $\beta=0$.
\end{proof}

The converse of the above proposition is the celebrated theorem of Narasimhan--Seshadri \cite{NarasimhanSeshadri}, as interpreted by Atiyah-Bott \cite{AtiyahBott}. A direct proof in these terms was provided by Donaldson \cite{DonaldsonNS}.

\begin{thm}[Narasimhan--Seshadri]
Every polystable holomorphic vector bundle admits a HE metric.
\end{thm}

\subsection*{Narasimhan--Seshadri as infinite dimensional Kempf--Ness}
It is important to remark that the content of the original result of Narasimhan and Seshadri is not only the existence of Hermitian--Einstein metrics, but also the identification of two moduli spaces: the moduli space of semistable vector bundles and the twisted character variety of unitary representations. Let us explain this statement.

The theorem of Narasimhan--Seshadri can be interpreted in a more ``dynamical'' way. We start by consdering a smooth vector bundle $E$ of rank $r$ and degree $d$, and fix a Hermitian metric $H$ on $E$. Recall that we denoted by $\cC_E$ the space of holomorphic structures on $E$, and by $\cA_{E,H}$ the space of unitary connections on $(E,H)$. There is a natural map
\begin{align*}
\cC_E  \longrightarrow \cA_{E,H}, \delbar_{\sE} \longmapsto \nabla_{\sE}=\nabla_{(\sE,H)},
\end{align*}
sending a holomorphic structure $\delbar_{\sE}$ to the Chern connection of the Hermitian holomorphic bundle $(\sE,H)$. The complex gauge group $\cG^\C_E = \Omega^0(X,\Aut E)$ acts on $\cC_E$ and, through this map, on the space $\cA_{E,H}$. Given an element $g\in \cG^\C_E$, the connection $\nabla_{\sE}$ is mapped to $g\cdot \nabla_{\sE}:=\nabla_{(\sE,g^* H)}$, the Chern connection for $\sE$ equipped with the Hermitian metric $H$ transformed by $g$. Therefore, the existence of a Hermitian--Einstein metric $H_0$ on a holomorphic vector bundle $\sE$ can be interpreted as the existence of a unitary connection of constant central curvature $\nabla_0=\nabla_{(\sE,H_0)}$ in the $\cG^\C_E$-orbit of $\nabla_{\sE}$. The theorem of Narasimhan--Seshadri can then be reformulated as follows.

\begin{thm}[Narasimhan--Seshadri]
Let $E$ be a smooth complex vector bundle of rank $r$ and degree $d$ on $X$, and fix a Hermitian metric $H$ on $E$.
For any holomorphic vector bundle $\sE$ with underlying smooth vector bundle $E$, consider the unitary connection $\nabla_{\sE}$, defined as the Chern connection of $(\sE,H)$.
We have that
    \begin{enumerate}
\item $\sE$ is semistable if and only if $\overline{\cG^\C_E \cdot \nabla_{\sE}} \cap \mu^{-1}(\alpha_E)\neq \varnothing$;
\item $\sE$ is polystable if and only if $\cG^\C_E\cdot \nabla_{\sE} \cap \mu^{-1}(\alpha_E)\neq \varnothing$ and, in that case, $\cG^\C_E\cdot \nabla_{\sE} \cap \mu^{-1}(\alpha_E)=\cG_{E,H} \cdot \nabla_{\sE}$. Therefore,
            \begin{equation*}
\mu^{-1}(\alpha_E) \subset \cG^\C_E \cdot \mu^{-1}(\alpha_E) = \cC_E^{ps} \subset \cC_E^{ss}.
            \end{equation*}
    \end{enumerate}
Moreover, the inclusion $\mu^{-1}(\alpha_E)\subset \cC_E^{ps}$ induces a homeomorphism
           \begin{equation*}
\cX^d_{\U(r)}\cong \mu^{-1}(\alpha_E)/\cG_{E,H} \longrightarrow \cC_E^{ps}/\cG_E^\C \cong \cN_{r,d},
           \end{equation*}
identifying the $d$-twisted $\U(r)$-character variety with (the analytification of) the moduli space of semistable holomorphic vector bundles of rank $r$ and degree $d$.
  \end{thm}

\subsection*{Moduli of (twisted) $\SL_r$-bundles and $\PGL_r$-bundles}
Consider the special unitary group $\SU(r)\subset \U(r)$ of unitary matrices of determinant $1$. The character variety $\cX_{\SU(r)}=\Hom(\pi_1(X,x_0),\SU(r))/\SU(r)$ and its twisted versions $\cX_{\SU(r)}^d$ (defined as in Section \ref{ss:charactervarieties}) also admit an interpretation in terms of holomorphic vector bundles. The \emph{moduli space of semistable holomorphic $\SL_r(\C)$-vector bundles} is by definition the preimage
\begin{equation*}
\check{\cN}_r=\cN(\SL_r):=\det{}^{-1}(\sO_X)
\end{equation*}
of the map $\det:\cN_{r,0}\rightarrow \Jac(X)$. More generally, for any integer $d$ we can fix a degree $d$ holomorphic line bundle $\xi$ and define the \emph{moduli space of semistable holomorphic $d$-twisted $\SL_r(\C)$-vector bundles} $\check{\cN}_{r,d}=\cN_d(\SL_r)$ as the preimage of $\xi$ under the map $\det: \cN_{r,d}\rightarrow \Pic^d(X)$. The theorem of Narasimhan--Seshadri identifies
\begin{equation*}
\cX^d_{\SU(r)} \cong \check{\cN}_{r,d}.
\end{equation*}
In particular, the complex dimension of $\check{\cN}_{r,d}$ is
\begin{equation*}
\dim_\C \check{\cN}_{r,d} = (r^2-1)(g-1).
\end{equation*}

There is an easy way to recover the space $\cN_{r,d}$ from $\check{\cN}_{r,d}$. We consider the finite group
\begin{equation*}
\Gamma_r = \Jac(X)[r] \cong H^1(X,\Z_r) \cong (\Z_r)^{2g}
\end{equation*}
of order $r$ points of the Jacobian of $X$. The group $\Gamma_r$ acts on $\check{\cN}_{r,d}$ and on $\Pic^d(X)$ by tensorization. The space $\cN_{r,d}$ is then recovered as
\begin{equation*}
\cN_{r,d} = (\check{\cN}_{r,d} \times \Pic^d(X) )/\Gamma_r.
\end{equation*}

We can also consider the projective unitary group \[\PU(r)=\U(r)/\left\{ \zeta I_r: \zeta \in \U(1) \right\}= \SU(r)/\left\{ \zeta I_r: \zeta^r=1 \right\},\] and the character variety $\cX_{\PU(r)}=\Hom(\pi_1(X,x_0),\PU(r))/\PU(r)$. This character variety is determined by the relation
\begin{equation*}
[A_1,B_1]\dots [A_g,B_g] = I_r
\end{equation*}
in $\PU(r)$, which translates to the relation
\begin{equation*}
[A_1,B_1]\dots [A_g,B_g] = \zeta I_r, \text{ for some } \zeta\in \U(1) \text{ with } \zeta^r=1,
\end{equation*}
in $\U(r)$. The choice of $r$-th root of unity determines $r$ connected components on $\cX_{\PU(r)}$, that we label as $\cX_{\PU(r)}^0,\dots,\cX_{\PU(r)}^{r-1}$. Note that we can identify, for each $d=0,\dots,r-1$,
\begin{equation*}
\cX_{\PU(r)}^d = \cX_{\SU(r)}^d / (\Z_r)^{2g},
\end{equation*}
where the finite group $(\Z_r)^{2g}$ acts naturally by multiplication.
The theorem of Narasimhan--Seshadri identifies the character variety
\begin{equation*}
\cX^d_{\PU(r)} \cong \hat{\cN}_{r,d}
\end{equation*}
with the \emph{moduli space of semistable holomorphic $d$-twisted $\PGL_r(\C)$-vector bundles}, defined as the quotient
\begin{equation*}
\hat{\cN}_{r,d}=\cN_d(\PGL_r):= \check{\cN}_{r,d} / \Gamma_r.
\end{equation*}

\begin{rmk}
We remark that the subspace $\check{\cN}_{r,d}^s \subset \check{\cN}_{r,d}$ is in fact smooth, but the space $\hat{\cN}_{r,d}^s$ is not. However, it has a nice structure, being a quotient of the smooth manifold $\check{\cN}_{r,d}^s$ by the finite group $\Gamma_r$, the moduli space $\hat{\cN}_{r,d}^s$ is naturally an \emph{orbifold}.
\end{rmk}

\begin{rmk}
  More generally, we could consider a complex reductive group $G$ and study the classification of holomorphic principal $G$-bundles. This problem was studied by Ramanathan \cite{Ramanathan}, who gave explicit (semi)stability conditions for holomorphic principal bundles and proved a result analogous to the theorem of Narasimhan--Seshadri. Namely, Ramanathan's theorem identifies the moduli space $\cN(G)$ of semistable holomorphic principal $G$-bundles with the character variety parametrizing conjugacy classes of representations $\pi_1(X,x_0)\rightarrow K$, where $K\subset G$ is a maximal compact subgroup.
\end{rmk}

\section{More details about stacks} \label{ss:stacks}
A natural way to think about a stack is as a ``groupoid valued sheaf''. We can understand this as a ``functor'' from a ``geometric category'' (for us, generally this geometric category is either complex analytic spaces or $\C$-schemes, equipped with the analytic or étale topologies, respectively) to the (2-)category of groupoids. The word functor in this context can be made precise in terms of higher categories. Moreover, if we want to think about a ``sheaf'', then we need to impose some ``gluing'' (aka \emph{descent}) conditions. Thinking in these terms allows one to consider a notion of space that not only contains points, but also ``automorphism groups'' (aka \emph{inertia groups}) attached to each point. We shall illustrate these ideas with examples.

Any group $G$ can be regarded as a groupoid: namely, we can consider the category with one object and whose morphisms are given by the elements of $G$. If $G$ is a group in the geometric category (i.e., for us, a complex Lie group or a group $\C$-scheme), then we can consider the stack $\bbB G$ which maps any complex space $S$ to the groupoid of principal $G$-bundles on $S$. Note that if $S$ is just a point, or more general a space where all principal $G$-bundles are trivial, then $\bbB G(S)$ is just the groupoid $G$. Indeed, $\bbB G(S)$ only has one object, the trivial bundle, but this trivial bundle has a whole $G$ worth of automorphisms.

Another interesting stack to consider is the stack $\mathbf{Pic}_X$, mapping any complex space $S$ to the groupoid of line bundles on $X\times S$. Again, the $\C$-points $\mathbf{Pic}_X(\C)$ form a groupoid whose objects are simply the $\C$-points of the moduli space $\Pic(X)$ (i.e. the isomorphism classes of line bundles on $X$), but each point comes equipped with a whole $\C^*$ worth of automorphisms. In fact the decomposition
\begin{equation*}
\mathbf{Pic}_X = \Pic(X) \times \bbB \C^*
\end{equation*}
is global: since there is a universal line bundle $\mathscr{U}\rightarrow \Pic(X) \times X$, for each $S$ we can identify the objects of $\mathbf{Pic}_X(S)$ with the morphisms $S\rightarrow \Pic(X)$, and the automorphisms of an object of $\mathbf{Pic}_X(S)$ are given by maps $S\rightarrow \C^*$.

More generally, we can consider the stack $\mathbf{Bun}_{r,d}$, mapping any complex space $S$ to the groupoid of families of vector bundles of rank $r$ and degree $d$ on $X$ parametrized by $S$, with morphisms given by equivalence, or the more manageable stack $\mathbf{Bun}_{r,d}^{s}$ of stable vector bundles. Recall that stable bundles are simple (Exercise \ref{ej:stablesimple}), so the $\C$-points $\mathbf{Bun}_{r,d}^{s}(\C)$ form a groupoid whose objects are the $\C$-points of the moduli space $\cN_{r,d}^s$, but each point comes equipped with a whole $\C^*$ worth of automorphisms. This determines a map $\mathbf{Bun}_{r,d}^s \rightarrow \cN_{r,d}^s$. However, unlike in the case line bundles we cannot generally split $\mathbf{Bun}_{r,d}^s$ as $\cN_{r,d}^s \times \bbB \C^*$, even if $r$ and $d$ were coprime. This is because, even though in the case of $r$ and $d$ coprime there is a universal family $\mathscr{U}\rightarrow \cN_{r,d}^s=\cN_{r,d}$, it follows from our definition of equivalence of families that this universal family does not determine a vector bundle, but rather just a $\PGL_r(\C)$-bundle. We denote this bundle by $\bbP \bbE \rightarrow \cN_{r,d}$.

The structure of $\mathbf{Bun}_{r,d}^s \rightarrow \cN^s_{r,d}$ is that of a gerbe. If $H$ is a commutative complex group, we say that a map $\mathbf{M}\rightarrow M$ from a stack to a scheme is an \emph{$H$-gerbe} if
\begin{enumerate}
  \item it is a \emph{$BH$-torsor}, meaning that for every $S\rightarrow M$, the set of objects of the groupoid $\mathbf{M}(S)$ is $\Hom(S,M)$, and there is a transitive action of the group $H(S)$ on the set of automorphisms of each object of $\mathbf{M}(S)$ and
  \item it is locally trivial, meaning that we can cover $M$ by opens $U$ such that $\mathbf{M}(U)$ is non-empty and it is isomorphic to $M(U) \times \bbB H(U)$.
\end{enumerate}
A \emph{trivialization} or \emph{splitting} of a gerbe is an isomorphism $\mathbf{M} \cong M \times \bbB H$.

\begin{ej}
Show that $H$-gerbes are classified by the \v{C}ech cohomology group $H^2(X,H)$. \textbf{Hint}: Trivialize on local opens and show that on double intersections you have ``transition bundles''. Now, compare these transition bundles on triple intersections to obtain a \v{C}ech $2$-cocycle with values on $H$.
\end{ej}

We have shown that $\mathbf{Pic}_X\rightarrow \Pic(X)$ is a trivial $\C^*$-gerbe but that $\mathbf{Bun}_{r,d}^s \rightarrow \cN^s_{r,d}$ is a $\C^*$-gerbe which in general does not split. In fact, the element of $H^2(\cN^s_{r,d},\C^*)$ determined by this gerbe coincides with the image of the class $[\bbP \bbE] \in H^1(\cN^s_{r,d},\PGL_r(\C))$ through the map induced by the short exact sequence \[1\rightarrow \C^* \rightarrow \GL_r(\C) \rightarrow \PGL_r(\C) \rightarrow 1.\]
We can also consider the stack $\mathbf{Bun}_{\SL_r,d}^s$ of $d$-twisted $\SL_r$-vector bundles, which is a $\mu_r$-gerbe over the moduli space $\check{\cN}^s_{r,d}$, for $\mu_r\subset \C^*$ the group of $r$-th roots of unity.

\chapter{Non-abelian Hodge theory} \label{sec:naht}
\section{Character varieties and the Betti moduli space} \label{ss:charactervarieties}
In the previous chapter, we studied the moduli space of holomorphic vector bundles on a compact Riemann surface $X$ and showed how it is related to the (twisted) character variety parametrizing unitary representations of the fundamental group of $X$, via the theorem of Narasimhan--Seshadri. We are now interested in considering the space of \textit{all} linear representations of the fundamental group, not only those which are unitary. In other words, we want to classify homomorphisms $\rho:\pi_1(X,x_0)\rightarrow \GL_r(\C)$ up to conjugacy. This leads naturally to the algebraic theory of character varieties.

Let $\Pi=\langle s_1,\dots,s_p : r_1(s_1,\dots,s_p)=1, \dots,r_q(s_1,\dots,s_q)=1 \rangle$ be a finitely presented group. The \emph{$\GL_r$-representation variety} $\mathcal{R}_{\Pi,\GL_r}$ (over $\C$) associated with $\Pi$ is the affine variety representing the functor sending any  $\C$-algebra $A$ to the set
\begin{align*}
	\mathcal{R}_{\Pi,\GL_r} & (A)=\Hom(\Pi,\GL_r(A)) \\ &=\left\{S_1, \dots, S_p \in \GL_r(A): r_1(S_1,\dots,S_p)=I_r,\dots, r_q(S_1,\dots,S_p)=I_r\right\}.
\end{align*}
The group $\GL_r$ acts on  $\mathcal{R}_{\Pi,\GL_r}$ by conjugation and the affine GIT quotient
\begin{equation*}
\mathcal{X}_{\Pi,\GL_r} = \mathcal{R}_{\Pi,\GL_r} \git \GL_r = \Spec(\C[\mathcal{R}_{\Pi,\GL_r}]^{\GL_r})
\end{equation*}
is called the \emph{$\GL_r$-character variety} (over $\C$) associated with $\Pi$.
More generally, if we fix a generator $s_i \in \Pi$ and a conjugacy class $c \subset \GL_r$, we can also consider the subvariety $\mathcal{R}^{c,s_i}_{\Pi,\GL_r}\subset \mathcal{R}_{\Pi,\GL_r}$ representing the functor
 \begin{equation*}
A \mapsto \mathcal{R}^{c,s_i}_{\Pi,\GL_r}(A) = \left\{\rho:\Pi \rightarrow \GL_r(A): \rho(s_i) \in c\right\} ,
\end{equation*}
and the corresponding GIT quotient
\begin{equation*}
\mathcal{X}_{\Pi,\GL_r}^{c,s_i}= \mathcal{R}_{\Pi,\GL_r}^{c,s_i}\git \GL_r.
\end{equation*}

Recall that one of the important properties of the affine GIT quotient is that the closed points of $\mathcal{X}_{\Pi,\GL_r}^{c,s_i}$ correspond to the closed $\GL_r(\C)$ orbits in $\mathcal{R}_{\Pi,\GL_r}^{c,s_i}(\C)$. Now, these orbits are precisely the orbits of the semisimple representations. A representation $\rho:\Pi \rightarrow \GL_r(\C)$ is \emph{semisimple} if and only if it decomposes as a direct sum of simple representations. Therefore, if we consider the subset $\mathcal{R}^{c,s_i}_{\Pi,\GL_r}(\C)^+ \subset\mathcal{R}^{c,s_i}_{\Pi,\GL_r}(\C)$ consisting of semisimple representations, we have
\begin{equation*}
\mathcal{X}^{c,s_i}_{\Pi,\GL_r}(\C) = \mathcal{R}^{c,s_i}_{\Pi,\GL_r}(\C)^+/\GL_r(\C).
\end{equation*}

Let us consider now our compact Riemann surface $X$, with two marked points $x_0$ and $x_1$, and let us take
\begin{equation*}
\Pi=\pi_1(X\setminus \left\{x_1\right\}, x_0)=\left\langle a_1,\dots,a_g,b_1,\dots,b_g,z : \prod_{i=1}^g [a_i,b_i]=z \right\rangle.
\end{equation*}
For any integer $d$, we let  $c_d\subset \GL_r$ denote the conjugacy class of the matrix  $e^{-\frac{2\pi i d}{r}}I_r$. We define the \emph{Betti moduli space} $\mathcal{M}^B_{r,d}$ of $X$ as
\begin{equation*}
\mathcal{M}_{r,d}^B= \mathcal{X}_{\Pi,\GL_r}^{c_d,z}.
\end{equation*}
In particular, for $d=0$, we obtain the character variety
\begin{equation*}
\mathcal{M}_{r,d}^B= \mathcal{X}_{\pi_1(X,x_0),\GL_r}.
\end{equation*}

Recall from Section \ref{ss:deformation} that we can consider the Zariski open subset $\cM_{r,d}^{B,s} \subset \cM_{r,d}^{B}$ consisting of simple representations. It follows from our discussions there that $\cM_{r,d}^{B,s}$ is a smooth complex algebraic variety, of complex dimension
\begin{equation*}
\dim_\C \cM_{r,d}^{B,s} = 2[(g-1)r^2 +1].
\end{equation*}

\section{The de Rham moduli space} \label{ss:deRham}
\subsection*{Algebraic construction}
Recall that the Riemann--Hilbert correspondence relates representations of the fundamental group with holomorphic bundles with holomorphic connection. This motivates the study of the moduli space of such pairs, that is usually called the \emph{de Rham moduli space}.

\begin{defn}
A holomorphic vector bundle with meromorphic connection $(\sV,\bm{D})$ on $X$ is \emph{semistable} (resp. \emph{stable}) if and only if for every $\bm{D}$-invariant holomorphic subbundle $\sV'\subset \sV$ (that is, $\bm{D}(\sV')\subset \sV' \otimes \bm{\Omega}^1_X$), we have
 \begin{equation*}
\mu(\sV') \leq \mu(\sV) \text{ (resp. } <).
\end{equation*}
We say that $(\sV,\bm{D})$ is \emph{polystable} if it is either stable or a direct sum of stable pairs with of slope equal to $\mu(\sV)$.
\end{defn}

We let $\mathrm{Conn}^{s}_{r,d}$ denote the set of isomorphism classes of stable holomorphic vector bundles with meromorphic connection of rank $r$ and degree $d$ on $X$, and let $\mathsf{Conn}^s_{r,d}$ denote the moduli problem for this set.

\begin{thm}[Simpson]
There exists a quasiprojective variety $\mathcal{M}^{\dR}_{r,d}$, the \emph{de Rham moduli space of rank $r$ and degree $d$ on $X$}, such that:
\begin{enumerate}
\item The set of closed points $\mathcal{M}^{\dR}_{r,d}(\C)$ is in natural bijection with the set of isomorphism classes of polystable holomorphic vector bundles with meromorphic connection of rank $r$ and degree $d$ on $X$.
\item There is a Zariski open subvariety $\mathcal{M}^{\dR,s}_{r,d}\subset \mathcal{M}^{\dR}_{r,d}$ which is a coarse moduli space for the moduli problem $\mathsf{Conn}^s_{r,d}$.
\item If $r$ and $d$ are coprime, then $\mathcal{M}^{\dR}_{r,d}=\mathcal{M}^{\dR,s}_{r,d}$ is a fine moduli space for $\mathsf{Conn}^s_{r,d}$. In particular, there is a \emph{universal vector bundle with connection} $\mathscr{U}_{r,d}\rightarrow \mathcal{M}^{\dR}_{r,d}$ from which any flat family of vector bundles with connection of rank $r$ and degree $d$ on $X$ arises as pullback.
\end{enumerate}
\end{thm}

The proof of this result can be found in \cite{SimpsonI,SimpsonII}, and follows a very similar argument to Seshadri's construction of the moduli space of vector bundles, that we explained in  Section \ref{ss:algebraicconstructionN}. Simpson also proved that the Riemann--Hilbert correspondence can be upgraded to a complex-analytic isomorphism of moduli spaces. Since taking holonomy is essentially transcendental, this isomorphism is not algebraic.

\begin{thm}[Simpson]
  Let $(\sV,\bm{D})$ be a holomorphic vector bundle with meromorphic connection of rank $r$ and degree $d$ on $X$, and let \[\rho:\Pi=\pi_1(X\setminus \left\{ x_1 \right\},x_0)\rightarrow \GL_r(\C)\] be the corresponding representation. Then, the pair $(\sV,\bm{D})$ is polystable if and only if the representation $\rho$ is semisimple. This determines a natural map $\mathcal{R}^{c_d,z}_{\Pi,\GL_r}(\C)\rightarrow \mathrm{Conn}^{ps}_{r,d}$ which is equivariant with respect to the conjugation $\GL_r(\C)$-action on the left, and in turn descends to a complex analytic isomorphism
  \begin{equation*}
\mathrm{RH}: \mathcal{M}^B_{r,d} \longrightarrow \mathcal{M}^{\dR}_{r,d}.
  \end{equation*}

\end{thm}

\begin{rmk}
One way to check that the Riemann--Hilbert isomorphism $\mathrm{RH}: \mathcal{M}^B_{r,d} \rightarrow \mathcal{M}^{\dR}_{r,d}$ is not algebraic is by comparing the mixed Hodge structures of both spaces. Indeed, the Betti moduli space is an affine variety and as such it has a balanced Hodge structure, while the Hodge structure of the de Rham moduli space is pure. See \cite{HauselMSNAHT} for more details.
\end{rmk}

\subsection*{As a symplectic quotient}
Let us consider a smooth vector bundle $E$ of rank $r$ and degree $d$ and, as usual, write $\mu=\mu(E)=d/r$. Recall that we considered the space $\cA_E$ of connections on $E$. A Hermitian metric $H$ on $E$ induces a ``Cartan decomposition''
\begin{equation*}
\End E = \mathfrak{u}_H E \oplus i \mathfrak{u}_H E,
\end{equation*}
since every endomorphism can be decomposed in its Hermitian and skew-Hermitian parts. Therefore, if $D$ is a connection on $E$, then we can write
\begin{equation*}
D= \nabla + i\Phi,
\end{equation*}
where $\nabla$ is a  $H$-unitary connection on $E$ and $\Phi \in \Omega^{1}(X,i \mathfrak{u}_H E)$. Therefore, we obtain an splitting
\begin{equation*}
\cA_E = \cA_{E,H} \oplus \Omega^1(X,i \fu_H E).
\end{equation*}
This splitting determines a Kähler structure on the vector space $\cA_E$. More precisely, this complex structure $I_{\dR}$ acts on a given point $(A,\Psi) \in \Omega^1(X,\fu_H E) \oplus \Omega^1(X,i \fu_H E)$, as
\begin{equation*}
I_{\dR}(A,\Psi) = (-\Psi, A),
\end{equation*}
and, for a pair of points $(A_1,\Psi_1)$ and $(A_2,\Psi_2)$, the symplectic structure $\omega_{\dR}$ is given by
\begin{equation*}
\omega_{\dR}((A_1,\Psi_1),(A_2,\Psi_2)) = \int_X \mathrm{tr} (\Psi_1 \wedge * A_2 - A_1 \wedge * \Psi_2),
\end{equation*}
where $*$ denotes the Hodge star operator on the Riemann surface $X$ (with respect to the prescribed Kähler form $\omega_X$).

\begin{ej}
The unitary gauge group $\cG_{E,H}$ acts by gauge transformations on the space of all connections $\cA_E$.
\textit{Show that this action admits the moment map}
\begin{equation*}
\mu_{\dR}: \cA_E = \cA_{E,H} \oplus \Omega^1(X,i \fu_H E) \longrightarrow \Omega^2(X,\fu_H E), \ D=\nabla + i\Phi \mapsto \nabla * \Phi,
\end{equation*}
\textit{with respect to the symplectic structure $\omega_{\dR}$.}
\end{ej}

The action of the unitary gauge group preserves the curvature, and thus it can be restricted to an action on the subspace $\cF_E\subset \cA_E$ of connections of constant central curvature. We can then consider the symplectic reduction \[(\cM_{r,d},I_{\dR},\omega_{\dR}):= \mu_{\dR}|_{\cF_E}^{-1}(0)/\cG_{E,H},\] which inherits a Kähler structure from $(\cA_E,I_{\dR},\omega_{\dR})$. In general this is an analytic space which is not smooth, but the dense open subset $\cM_{r,d}^s\subset \cM_{r,d}$ of irreducible connections is indeed a Kähler manifold. As we explain in the next section, we can identify this space with (the analytification of) the de Rham moduli space.

\subsection*{Harmonic metrics and the Corlette--Donaldson theorem}
Let $(\sV,\bm{D})$ be a holomorphic vector bundle with meromorphic connection. Recall that, if $H$ is a Hermitian metric on $\sV$, then there is a unique unitary connection $\nabla_H$, the Chern connection, such that $\delbar_{\sV}=\nabla_H^{0,1}$. We define the operator $\partial^{H}_{\sV}:= \nabla_H^{1,0}$.

\begin{ej}
\textit{Following a similar argument, prove that there is a unique unitary connection $\nabla_H'$ such that $\bm{D}=(\nabla'_H)^{1,0}$}.
\end{ej}

We denote $\bm{D}^{\dagger_H}= (\nabla_H')^{0,1}$. We can then consider the operator
\begin{equation*}
\delbar_{\sE}:= \tfrac{1}{2}(\delbar_{\sV} + \bm{D}^{\dagger_H})
\end{equation*}
and the $(1,0)$-form
\begin{equation*}
\varphi : = \tfrac{1}{2} (\bm{D}- \partial_{\sV}^H ).
\end{equation*}
By definition, the \emph{pseudocurvature} of the metric $H$ (with respect to the pair $(\sV,\bm{D})$) is the $(1,1)$-form
\begin{equation*}
G_H = \delbar_{\sE} \varphi.
\end{equation*}

\begin{defn}
A Hermitian metric $H$ on a holomorphic vector bundle with meromorphic connection $(\sV,\bm{D})$ is \emph{harmonic} if $G_H=0$.

A  \emph{Higgs bundle}	is a pair $(\cE,\varphi)$ consisting of a holomorphic vector bundle $\cE$ over $X$ a and holomorphic $\End \cE$-valued $(1,0)$-form $\varphi \in H^{1,0}(X,\End \cE)$. (Equivalently, $\varphi$ is a holomorphic ``twisted endomorphism'' $\varphi:\cE \rightarrow \cE \otimes \bm{\Omega}^1_X$).
\end{defn}

It is clear now that a harmonic metric on a pair $(\sV,\bm{D})$ determines a Higgs bundle. The question then is when does a pair $(\sV,\bm{D})$ admit a harmonic metric. The answer is the Corlette--Donaldson theorem.

\begin{thm}[Corlette--Donaldson]
A holomorphic vector bundle with meromorphic connection $(\sV,\bm{D})$ admits a harmonic metric if and only if it is polystable.
\end{thm}

\begin{ej}
  It is convenient to rewrite the Corlette--Donaldson theorem in terms of flat bundles, through the correspondence $(\sV,\bm{D})\mapsto (E,\bm{D}+\delbar_{\sV})$ from Proposition \ref{prop:holconnflatconn}. Let $(E,D)$ be a flat bundle determined by a pair $(\sV,\bm{D})$, and let us fix a Hermitian metric $H$ on $E$. Consider now the decomposition $D=\nabla + i\Phi$ induced by the splitting $\End E = \fu_H E \oplus i \fu_H H$. \textit{Show that}
\begin{equation*}
\nabla = \nabla_{(\sE,H)} = \delbar_{\sE} + \partial_{\sE}^{H} \text{ and } \Phi = \varphi - \varphi^{\dagger_H},
\end{equation*}
\textit{where $\partial_{\sE}^{H} = \tfrac{1}{2}(\bm{D} + \partial_{\sV}^{H})$ and $\varphi^{\dagger_H}=\tfrac{1}{2}(\delbar_{\sV}-\bm{D}^{\dagger_H})$. Show then that}
\begin{equation*}
G_H = \delbar_{\sE} \varphi = -\nabla * \Phi + i \nabla\Phi
\end{equation*}
\textit{and that}
\begin{equation*}
F_D = F_\nabla + i\nabla \Phi - \Phi \wedge \Phi =  \alpha_E.
\end{equation*}
Using the decomposition $\End E = \fu_H E \oplus i \fu_H E$, \textit{conclude that}
\begin{equation*}
\nabla\Phi = 0, \text{ so then } G_H = -\nabla* \Phi
\end{equation*}
\textit{and}
\begin{equation*}
F_D = F_\nabla + [\varphi,\varphi^{\dagger_H}] =  \alpha_E.
\end{equation*}
We can then rewrite the Corlette--Donaldson theorem as follows.
\end{ej}

\begin{thm}[Corlette--Donaldson]
  A bundle with connection $(E,D)$ with constant central curvature determines a semisimple representation if and only if it admits a Hermitian metric $H$ such that
  \begin{equation*}
\nabla * \Phi = 0,
  \end{equation*}
where $D=\nabla + \Phi$ is the natural decomposition induced by the Cartan splitting $\End E = \fu_H E \oplus i \fu_H E$.
\end{thm}

In a similar way as we did for vector bundles, we can also give a ``dynamical'' interpretation of the Corlette--Donaldson theorem, where instead of fixing the holomorphic structure and finding a canonical metric, we fix the metric and act through the complex gauge group to reach a solution of the moment map equation.

\begin{thm}[Corlette--Donaldson, Simpson]
Let $E$ be a smooth complex vector bundle of rank $r$ and degree $d$ on $X$, and fix a Hermitian metric $H$ on $E$.
For any holomorphic vector bundle with meromorphic connection $(\sV,\bm{D})$ with underlying smooth vector bundle $E$, consider the flat connection $D=\bm{D}+\delbar_{\sV}$, which splits as a sum $D=\nabla+\Phi$ of a unitary connection $\nabla$ and a $1$-form $\Phi\in \Omega^1(X,i\fu_H E)$. The complex gauge group $\cG^{\C}_E$ acts on the pair $(\nabla,\Phi)$.
We have that
    \begin{enumerate}
\item $(\sV,\bm{D})$ is semistable if and only if $\overline{\cG^\C_E \cdot (\nabla,\Phi)} \cap \mu_{\dR}^{-1}(0)\neq \varnothing$;
\item $(\sV,\bm{D})$ is polystable if and only if $\cG^\C_E\cdot (\nabla,\Phi) \cap \mu_{\dR}^{-1}(0)\neq \varnothing$ and, in that case, $\cG^\C_E\cdot (\nabla,\Phi) \cap \mu_{\dR}^{-1}(0)=\cG_{E,H} \cdot (\nabla,\Phi)$. Therefore,
            \begin{equation*}
\mu_{\dR}^{-1}(0) \subset \cG^\C_E \cdot \mu_{\dR}^{-1}(0) = \cD_E^{ps} \subset \cD_E^{ss},
            \end{equation*}
    \end{enumerate}
    where $\cD_E$ denotes the set of holomorphic vector bundles with meromorphic connection with underlying vector bundle $E$, and the superscripts $^{ps}$ and $^{ss}$ stand for taking subsets of polystable and semistable pairs, respectively.

Moreover, the inclusion $\mu_{\dR}^{-1}(0)\subset \cD_E^{ps}$ induces a complex analytic isomorphism
           \begin{equation*}
(\cM_{r,d},I_{\dR},\omega_{dR})=\mu_{\dR}^{-1}(0)/\cG_{E,H} \longrightarrow \cD_E^{ps}/\cG_E^\C \cong \cM^{\dR}_{r,d},
           \end{equation*}
           identifying the Kähler manifold $(\cM^s_{r,d},I_{\dR},\omega_{dR})$ with the (stable part of) de Rham moduli space of rank $r$ and degree $d$.
\end{thm}

\section{The moduli space of Higgs bundles}
\subsection*{Algebraic construction}
We have shown how a holomorphic bundle with meromorphic connection endowed with a harmonic metric determines a Higgs bundle $(\sE,\varphi)$. It is then interesting to study the moduli space of Higgs bundles, also called the \emph{Dolbeault moduli space}. The stability theory and the construction of the moduli space completely parallels the cases we have already studied.

\begin{defn}
A Higgs bundle $(\sE,\varphi)$ on $X$ is \emph{semistable} (resp. \emph{stable}) if and only if for every $\varphi$-invariant holomorphic subbundle $\sE'\subset \sE$ (that is, $\varphi(\sE')\subset \sE' \otimes \bm{\Omega}^1_X$), we have
 \begin{equation*}
\mu(\sE') \leq \mu(\sE) \text{ (resp. } <).
\end{equation*}
We say that $(\sE,\bm{D})$ is \emph{polystable} if it is either stable or a direct sum of stable pairs with of slope equal to $\mu(\sE)$.
\end{defn}

We let $\mathrm{Higgs}^{s}_{r,d}$ denote the set of isomorphism classes of stable Higgs bundles of rank $r$ and degree $d$ on $X$, and let $\mathsf{Higgs}^s_{r,d}$ denote the moduli problem for this set.

\begin{thm}[Nitsure]
There exists a quasiprojective variety $\mathcal{M}^{\Dol}_{r,d}$, the \emph{Dolbeault moduli space of rank $r$ and degree $d$ on $X$}, such that:
\begin{enumerate}
\item The set of closed points $\mathcal{M}^{\Dol}_{r,d}(\C)$ is in natural bijection with the set of isomorphism classes of polystable Higgs bundles of rank $r$ and degree $d$ on $X$.
\item There is a Zariski open subvariety $\mathcal{M}^{\Dol,s}_{r,d}\subset \mathcal{M}^{\Dol}_{r,d}$ which is a coarse moduli space for the moduli problem $\mathsf{Higgs}^s_{r,d}$.
\item If $r$ and $d$ are coprime, then $\mathcal{M}^{\Dol}_{r,d}=\mathcal{M}^{\Dol,s}_{r,d}$ is a fine moduli space for $\mathsf{Higgs}^s_{r,d}$. In particular, there is a \emph{universal Higgs bundle} $\mathscr{U}_{r,d}\rightarrow \mathcal{M}^{\Dol}_{r,d}$ from which any flat family of Higgs bundles of rank $r$ and degree $d$ on $X$ arises as pullback.
\end{enumerate}
\end{thm}

\subsection*{As a symplectic quotient}
Let us consider a smooth vector bundle $E$ of rank $r$ and degree $d$ and, as usual, write $\mu=\mu(E)=d/r$. We fix a Hermitian metric $H$ on $E$. Recall that any connection $D$ on $E$, can be split as
\begin{equation*}
D= \nabla + i\Phi,
\end{equation*}
where $\nabla$ is a  $H$-unitary connection on $E$ and $\Phi \in \Omega^{1}(X,i \mathfrak{u}_H E)$, and thus we obtain an splitting
\begin{equation*}
\cA_E = \cA_{E,H} \oplus \Omega^1(X,i \fu_H E).
\end{equation*}
We now consider a Kähler structure on $\cA_E$ which is different from the one we considered in Section \ref{ss:deRham}. Given a point $(A,\Psi) \in \Omega^1(X,\fu_H E) \oplus \Omega^1(X,i \fu_H E)$, we define
\begin{equation*}
I_{\Dol}(A,\Psi) = (*A, -*\Phi),
\end{equation*}
and, for a pair of points $(A_1,\Psi_1)$ and $(A_2,\Psi_2)$, the symplectic structure $\omega_{\dR}$ is given by
\begin{equation*}
\omega_{\Dol}((A_1,\Psi_1),(A_2,\Psi_2)) = \int_X \mathrm{tr} (-A_1 \wedge A_2 + \Psi_1 \wedge \Psi_2).
\end{equation*}

\begin{rmk}
In some sense, we could understand that, while the complex structure $I_{\dR}$ was induced by the natural complex structure on the space $\End E$ (which is in turn induced by the complex structure on $\GL_r(\C)$), the complex structure $I_{\Dol}$ is induced by the complex structure on $X$.
\end{rmk}

\begin{ej}
  Let us consider the subspace $\cH_E\subset \cA_E$ of connections $D=\nabla + i\Phi$ such that $\nabla\Phi=\nabla* \Phi = 0$.
The unitary gauge group $\cG_{E,H}$ acts by gauge transformations on the space of $\cH_E$.
\textit{Show that this action admits the moment map}
\begin{align*}
  \mu_{\Dol}: \cA_E = \cA_{E,H} \oplus \Omega^1(X,i \fu_H E) & \longrightarrow \Omega^2(X,\fu_H E) \\
 D=\nabla + i\Phi   & \longmapsto F_\nabla + \Phi \wedge \Phi,
\end{align*}
\textit{with respect to the symplectic structure $\omega_{\Dol}$.}
\end{ej}

We can then consider the symplectic reduction \[(\cM_{r,d},I_{\Dol},\omega_{\Dol}):= \mu_{\Dol}|_{\cH_E}^{-1}(0)/\cG_{E,H},\] which inherits a Kähler structure from $(\cA_E,I_{\Dol},\omega_{\Dol})$. As we explain in the next section, we can identify this space with (the analytification of) the Dolbeault moduli space.

\subsection*{Hermitian--Einstein--Higgs metrics and the Hitchin--Simpson theorem}
Let $(\sE,\varphi)$ be a Higgs bundle with underlying smooth bundle $E$. If $H$ is a Hermitian metric on $\sE$, we can construct a connection on $E$ from $(\sE,\varphi)$ by putting
\begin{equation*}
D = \partial_{\sE}^{H} + \delbar_{\sE} + \varphi - \varphi^{\dagger_H}.
\end{equation*}
Since $\varphi$ is holomorphic, we have $\delbar_{\sE} \varphi = 0$, and thus the curvature of $D$ is equal to
\begin{equation*}
F_D = F_H + [\varphi,\varphi^{\dagger_H}],
\end{equation*}
where $F_H$ is the curvature of the Chern connection $\nabla_H = \partial_{\sE}^{H} + \delbar_{\sE}$. Note that $D$ is not necessarily a connection with constant central curvature. We are interested in finding Hermitian metrics for which that is in fact the case.

\begin{defn}
A \emph{Hermitian-Einstein-Higgs} metric (HEH metric) on a Higgs bundle $(\sE,\varphi)$ is a Hermitian metric $H$ on $\sE$ such that
\begin{equation*}
F_H + [\varphi,\varphi^{\dagger_H}] = \alpha_E.
\end{equation*}
\end{defn}

\begin{ej}
\textit{Verify that the equation $F_H + [\varphi,\varphi^{\dagger_H}] = \alpha_E$ is equivalent to the equation $F_H - \Phi \wedge \Phi = \alpha_E$, for $\Phi=\varphi - \varphi^{\dagger_H}$.}
\end{ej}

\begin{thm}[Hitchin--Simpson]
A Higgs bundle admits an HEH metric if and only if it is polystable.
\end{thm}

\begin{ej}
Emulate the proof of Proposition \ref{prop:HEimpliesPS} to show that if a Higgs bundle admits an HEH metric then it is polystable.
\end{ej}

As in the previous situations, we can give a ``dynamical'' interpretation of this result if we fix a Hermitian metric and try to find solutions of the moment map equation.

\begin{thm}[Hitchin--Simpson]
Let $E$ be a smooth complex vector bundle of rank $r$ and degree $d$ on $X$, and fix a Hermitian metric $H$ on $E$.
For any Higgs bundle $(\sE,\varphi)$ with underlying smooth vector bundle $E$, consider the pair $(\nabla,\Phi)$, where $\nabla=\partial_{\sE}^{H} + \delbar_{\sE}$ is the Chern connection and $\Phi=\varphi-\varphi^{\dagger_H}$. The complex gauge group $\cG^{\C}_E$ acts on the pair $(\nabla,\Phi)$.
We have that
    \begin{enumerate}
\item $(\sE,\varphi)$ is semistable if and only if $\overline{\cG^\C_E \cdot (\nabla,\Phi)} \cap \mu_{\Dol}^{-1}(\alpha_E)\neq \varnothing$;
\item $(\sE,\varphi)$ is polystable if and only if $\cG^\C_E\cdot (\nabla,\Phi) \cap \mu_{\Dol}^{-1}(\alpha_E)\neq \varnothing$ and, in that case, $\cG^\C_E\cdot (\nabla,\Phi) \cap \mu_{\Dol}^{-1}(\alpha_E)=\cG_{E,H} \cdot (\nabla,\Phi)$. Therefore,
            \begin{equation*}
\mu_{\Dol}^{-1}(\alpha_E) \subset \cG^\C_E \cdot \mu_{\Dol}^{-1}(\alpha_E) = \cS_E^{ps} \subset \cS_E^{ss},
            \end{equation*}
    \end{enumerate}
    where $\cS_E$ denotes the set of Higgs bundles with underlying vector bundle $E$, and the superscripts $^{ps}$ and $^{ss}$ stand for taking subsets of polystable and semistable pairs, respectively.

Moreover, the inclusion $\mu_{\Dol}^{-1}(\alpha_E)\subset \cS_E^{ps}$ induces a complex analytic isomorphism
           \begin{equation*}
(\cM_{r,d},I_{\Dol},\omega_{\Dol})=\mu_{\Dol}^{-1}(\alpha_E)/\cG_{E,H} \longrightarrow \cS_E^{ps}/\cG_E^\C \cong \cM^{\Dol}_{r,d},
           \end{equation*}
           identifying the Kähler manifold $(\cM^s_{r,d},I_{\Dol},\omega_{\Dol})$ with the Dolbeault moduli space of rank $r$ and degree $d$.
\end{thm}

\subsection*{A bit more about Higgs bundles}
\begin{ej}[Some examples of Higgs bundles] \label{ej:Examples_Higgs}
\textit{Can you think of any ``trivial'' or easy examples of Higgs bundles. Are they stable?}

A less trivial example is obtained if we consider any holomorphic line bundle $\sL$ over $X$ and take $\sE=\sL \otimes \bm{\Omega}^1_X \oplus \sL$. For any pair of sections $(a_1,a_2)\in H^0(X,\bm{\Omega}^1_X) \oplus H^0(X,(\bm{\Omega}^1_X)^{\otimes 2})$, we can equip $\sE$ with the Higgs field
 \begin{equation*}
\varphi =
\begin{pmatrix}
a_1 & a_2 \\
1 & a_1
\end{pmatrix}.
\end{equation*}

\textit{Show that, despite the fact that $\sE$ is not stable nor polystable, the Higgs bundle $(\sE,\varphi)$ is indeed stable.}
\end{ej}

\begin{ej}[Spin structures, and some more examples]
A \emph{spin structure} or \emph{theta-characteristic} on $X$ is a holomorphic line bundle $\sL$ on $X$ such that $\sL^{\otimes 2}\cong \bm{\Omega}^1_X$.

\textit{Show that the set of spin structures on $X$ up to equivalence is a torsor under the cohomology group $H^1(X,\Z_2)$. Therefore, there are exactly $2^{2g}$ equivalent spin structures on a genus $g$ surface. Why do you think these are called spin structures?}

Fix a spin structure $\sL$ on $X$ and consider the holomorphic vector bundle
\begin{equation*}
\sE=\sL \oplus \sL^{-1}.
\end{equation*}
For any $a_2 \in H^0(X,(\bm{\Omega}^1_X)^{\otimes 2})$, we can equip $\sE$ with the Higgs field
\begin{equation*}
\varphi =
\begin{pmatrix}
0 & a_2 \\
1 & 0
\end{pmatrix}.
\end{equation*}
In particular, note that $\det \sE = \sO_X$ and that  $\mathrm{tr}(\varphi)=0$. This is what is called an $\SL_2$-Higgs bundle.
\end{ej}

\begin{ej}[Uniformization à la Hitchin] \label{ej:hitchin_unif}
The Hitchin--Simpson theorem is so strong that it implies the uniformization theorem. Let us explore this in detail. We start by fixing a Riemannian metric $g=u(z,\bar{z}) dz d\bar{z}$ compatible with the complex structure of $X$  (that is, compatible with the conformal structure). The Levi-Civita connection associated to this metric can be regarded as a $\U(1)$-connection on the line bundle $\bm{\Omega}^1_X$. The curvature $F_0$ of the metric $g$ is the curvature of the induced $\U(1)$-connection on the tangent bundle $(\bm{\Omega}^1_X)^{\otimes -1}$.

Let us now fix a spin structure $\sL$ on $X$ with the induced $\U(1)$-connection. In turn we obtain a connection (reducible to $\U(1)$) on the vector bundle $\sE=\cL \oplus \cL^{-1}$, with curvature
\begin{equation*}
F =
\begin{pmatrix}
-\tfrac{1}{2} F_0 & 0 \\
0 &\tfrac{1}{2} F_0
\end{pmatrix}.
\end{equation*}

Consider the Higgs field
 \begin{equation*}
\varphi =
\begin{pmatrix}
0 & 0 \\
1 & 0
\end{pmatrix}.
\end{equation*}
We already now that this Higgs bundle $(\sE,\varphi)$ is stable (it is a particular case of Exercise \ref{ej:Examples_Higgs}). The moment map equation then becomes
\begin{equation*}
\begin{pmatrix}
-\tfrac{1}{2} F_0 & 0 \\
0 &\tfrac{1}{2} F_0
\end{pmatrix}=
\begin{pmatrix}
1 & 0 \\
0 & -1
\end{pmatrix}u dz d\bar{z}
.
\end{equation*}
Therefore, we obtain the equation
\begin{equation*}
F_0 = -2 g.
\end{equation*}
\textit{Verify that this means precisely that $g$ has constant curvature equal to  $-4$. Conclude from here the uniformization theorem.}
\end{ej}

\section{The Hitchin moduli space}
\subsection*{Recap: The non-abelian Hodge correspondence}
Let us pause for a second to summarize the main statements of what we have done so far. Let $E$ be a smooth vector bundle on $X$ of rank $r$ and degree $d$.

\begin{enumerate}
\item If $(\sE,\varphi)$ is a polystable Higgs bundle on $X$, with underlying smooth vector bundle $E$, by the Hitchin--Simpson theorem we can find a HEH metric $H$ on it, and obtain a flat connection
\begin{equation*}
D= \partial_{\sE}^H +\delbar_{\sE}+ \varphi - \varphi^\dagger.
\end{equation*}
In turn, the operators $\delbar_{\sV}=\delbar_{\sE} + \varphi^{\dagger_H}$ and $\bm{D}=\partial_{\sE}^H + \varphi$ determine a holomorphic vector bundle with meromorphic connection $(\sV,\bm{D})$. This defines a map
        \begin{align*}
          \cM_{r,d}^{\Dol} & \longrightarrow  \cM_{r,d}^{\dR} \\
           (\sE,\varphi) & \longmapsto (\sV,\bm{D}).
        \end{align*}

  \item Conversely, if $(\sV,\bm{D})$ is a polystable holomorphic vector bundle with meromorphic connection, with underlying smooth vector bundle $E$, by the Corlette--Donaldson theorem we can find a harmonic metric $H$ on it, and obtain a Higgs bundle $(\sE,\varphi)$ by putting $\delbar_{\sE}= \tfrac{1}{2}(\delbar_{\sV}+ \bm{D}^{\dagger_H})$ and $\varphi=\tfrac{1}{2}(\bm{D}-\delbar_{\sV}^H)$. This determines a map
        \begin{align*}
          \cM_{r,d}^{\dR} & \longrightarrow  \cM_{r,d}^{\Dol} \\
           (\sV,\bm{D}) & \longmapsto (\sE,\varphi).
        \end{align*}
\end{enumerate}
The above determines a bijection, and in fact a real analytic isomorphism between (the analytifications of) the Dolbeault moduli space $\cM^{\Dol}_{r,d}$ and the de Rham moduli space $\cM^{\dR}_{r,d}$. We are using $\cM_{r,d}$ to denote the underlying real space to any of these two complex spaces. What we have found is two different complex structures $I_{\Dol}$ and $I_{\dR}$ on $\cM_{r,d}$ \textit{which are not isomorphic}.

\begin{ej}[Abelian Hodge theory]
  We emphasized the fact that the de Rham and Dolbeault moduli spaces are not isomorphic as complex spaces. Let us see that this is the case even in the simplest situation: $r=1$ and $d=0$. Since $\GL_1(\C)=\C^*$ is an abelian group, the Betti moduli space $\cM_{1,0}^B$ is just the representation variety
  \begin{equation*}
\cM_{1,0}^B = \Hom(\pi_1(X,x_0),\C^*) = (\C^*)^{2g}
  \end{equation*}
On the other hand, since every line bundle is stable, we just have
\begin{equation*}
\cM_{1,0}^{\Dol} = \bm{T}^* \Jac(X) = \Jac(X) \times H^{1,0}(X) \cong \Jac(X) \times \C^{g}.
\end{equation*}
\textit{Write an explicit diffeomorphism between $\cM_{1,0}^B$ and $\cM_{1,0}^{\Dol}$. Are these two manifolds complex-analytically isomorphic?} \textbf{Hint}: For the diffeomorphism, recall that $\Jac(X)$ is an abelian variety of dimension $g$, and thus isomorphic to $\C^g/\Lambda$, for some lattice $\Lambda$.
\end{ej}

\begin{rmk}
  We can also consider the Abelian de Rham moduli space, $\cM_{1,0}^{\dR}$, which is a non-trivial affine bundle over $\Jac(X)$, and it is complex-analytically isomorphic to $(\C^*)^{2g}$. We can show that this isomorphism is not algebraic by comparing the \emph{mixed Hodge polynomials} of both spaces (see \cite[Section 2.2]{HauselMSNAHT} for the definition of the mixed Hodge polynomial). Indeed, one can easily check that
  \begin{equation*}
H(x,y,t;\cM_{1,0}^{\dR}) = H(x,y,t;\cM_{1,0}^{\Dol}) = (1+xt)^g(1+yt)^g,
  \end{equation*}
  and
  \begin{equation*}
H(x,y,t;\cM_{1,0}^{B}) = (1+xyt)^{2g}.
  \end{equation*}

\end{rmk}

\subsection*{Hyperkähler manifolds and hyperkähler quotients}
The fact that we have found two different complex structures on $\cM_{r,d}$ is hinting towards a more general setting in which this space should be studied, and from which the different complex structures will arise. This is the setting of hyperkähler geometry.

\begin{defn}
  A \emph{hyperkähler manifold} is a tuple $(M,g,I_1,I_2,I_3)$, formed by a smooth manifold $M$, a Riemannian metric $g$, and three complex structures $I_1$, $I_2$ and $I_3$ such that
  \begin{enumerate}
\item the metric $g$ is Kähler with respect to the three structures $I_1$, $I_2$ and $I_3$ (that is, the $2$-forms $\omega_i$ defined as $g=\omega_i(-,I_i-)$ are closed),
\item the complex structures satisfy the quaternionic relation $I_1I_2=I_3$.
  \end{enumerate}
\end{defn}

\begin{ej}
  The $4$-dimensional Euclidean space can be naturally endowed with a hyperkähler structure, by regarding it as the space of quaternions
  \begin{equation*}
\mathbb{H} = \left\{\bm{q}= x_0+x_1\bm{i} + x_2\bm{j} + x_3\bm{k} : (x_0,x_1,x_2,x_3)\in \bbR^4 \right\}.
  \end{equation*}
  The metric $g$ is the Euclidean metric, which in this terms is the quaternion inner product
  \begin{equation*}
(\bm{q},\bm{q}') = \bm{q} \bar{\bm{q}}' = x_0 x_0' + x_1 x_1' + x_2 x_2' + x_3 x_3'.
  \end{equation*}
The complex structures $I_1$, $I_2$ and $I_3$ are given by multiplication by $\bm{i}$, $\bm{j}$ and $\bm{k}$, respectively.
\textit{Compute the corresponding symplectic forms $\omega_1$, $\omega_2$ and $\omega_3$}.
\end{ej}

We can also construct \emph{hyperkähler quotients}. Let $(M,g,I_1,I_2,I_3)$ be a hyperkähler manifold endowed with the action of a Lie group $K$, which acts by isometries and preserving the hyperkähler structure. A \emph{hyperkähler moment map} for the $K$-action is a map
\begin{equation*}
\bm{\mu}=(\mu_1,\mu_2,\mu_3): M \longrightarrow \fk^* \otimes \bbR^3
\end{equation*}
such that each $\mu_i$ is a moment map for the $K$-action with respect to $\omega_i$.

\begin{thm}
Let $(M,g,I_1,I_2,I_3)$ be a hyperkähler manifold endowed with the action of a Lie group $K$, acting by isometries and preserving the hyperkähler structure, and with hyperkähler moment map $\bm{\mu}:M\rightarrow \fk^* \otimes \bbR^3$. Suppose as well that the action of $K$ on $\bm{\mu}^{-1}(0)$ is free and proper. Then, the \emph{hyperkähler quotient}
\begin{equation*}
M \hkgit K := \bm{\mu}^{-1}(0)/K
\end{equation*}
is a smooth manifold of dimension $\dim M - 4 \dim K$, and the quotient metric determines a hyperkähler structure on it.
\end{thm}

The interested reader can consult Neitzke's notes \cite[Theorem 3.49]{Neitzke} for a proof.

\subsection*{The Hitchin equations}
Let $E$ be a smooth complex vector bundle on $X$, with rank $r$ and degree $d$, and let $H$ be a Hermitian metric on $E$. Let $(\nabla,\Phi)$ be a pair formed by a $H$-unitary connection $\nabla$ on $E$ and a $\fu_H E$-valued $1$-form $\Phi\in \Omega^1(X,\fu_H E)$.

\begin{defn}
  The \emph{Hitchin equations} for such a pair $(\nabla,\Phi)$ are the equations
  \begin{equation*}
\begin{cases}
  F_\nabla + \Phi \wedge \Phi = \alpha_E \\
  \nabla \Phi = \nabla * \Phi = 0.
\end{cases}
  \end{equation*}

\end{defn}

The metric $H$ identifies the space of pairs $(\nabla,\Phi)$ with the space $\cA_E$ of connections on $E$. Moreover, it determines a hyperkähler structure on $\cA_E$. For a pair $(A_1,\Psi_1)$, $(A_2,\Psi_2)$ of elements of $\Omega^1(X,\fu_H E) \oplus \Omega^1(X,i\fu_H E)$, the Riemannian metric $g$ is defined as
\begin{equation*}
g((A_1,\Psi_1), (A_2,\Psi_2)) = -\int_X \mathrm{tr} (A_1 \wedge * A_2 + \Psi_1 \wedge * \Psi_2).
\end{equation*}
The complex structures are given by
\begin{equation*}
  \begin{cases}
I_1(A,\Psi)=I_{\Dol}(A,\Psi) = (*A,-*\Psi),\\ I_2(A,\Psi)=I_{\dR}(A,\Psi)=(-\Psi,A),\\ I_3(A,\Psi)=(-*\Psi,-*A).
  \end{cases}
\end{equation*}
One can easily check that the corresponding symplectic structures are then given by
\begin{equation*}
\begin{cases}
\omega_1((A_1,\Psi_1), (A_2,\Psi_2))= \int_X \mathrm{tr}(-A_1\wedge A_2 + \Psi_1 \wedge \Psi_2), \\
\omega_2((A_1,\Psi_1), (A_2,\Psi_2))= \int_X \mathrm{tr}(\Psi_1\wedge * A_2 - A_1 \wedge * \Psi_2), \\
\omega_3((A_1,\Psi_1), (A_2,\Psi_2))= \int_X \mathrm{tr}(\Psi_1\wedge A_2 + A_1 \wedge \Psi_2).
\end{cases}
\end{equation*}

\begin{ej}
Recall that the unitary gauge group $\cG_{E,H}$ acts on $\cA_E$ by gauge transformations. \textit{Show that the map}
\begin{align*}
  \bm{\mu}: \cA_E = \cA_{E,H} \oplus i\Omega^1(X,\fu_H E) & \longrightarrow  \Omega^2(X,\fu_H E) \otimes \bbR^3 \\
D=\nabla + i \Phi    & \longmapsto (F_\nabla + \Phi \wedge \Phi, -\nabla * \Phi, \nabla \Phi)
\end{align*}
\textit{is a hyperkähler moment map for the $\cG_{E,H}$-action, with respect to the hyperkähler structure $(g,I_1,I_2,I_3)$}.
\end{ej}

\begin{defn}
  The \emph{Hitchin moduli space} is the hyperkähler space
  \begin{equation*}
(\cM_{r,d},g,I_1,I_2,I_3):= \mu^{-1}(0) \hkgit \cG_{E,H},
  \end{equation*}
which classifies solutions to the Hitchin equations up to unitary gauge.
\end{defn}

We can now restate the non-abelian Hodge correspondence in the following way.

\begin{thm}[Corlette, Donaldson, Hitchin, Simpson]
  We have the following equivalences.
  \begin{itemize}
\item The Hitchin moduli space $\cM_{r,d}$ equipped with the complex structure $I_1$ is complex-analytically isomorphic to (the analytification of) the Dolbeault moduli space $\cM_{r,d}^{\Dol}$, classifying polystable Higgs bundles of rank $r$ and degree $d$ on $X$.
\item The Hitchin moduli space $\cM_{r,d}$ equipped with the complex structure $I_2$ is complex-analytically isomorphic to (the analytification of) the de Rham moduli space $\cM_{r,d}^{\dR}$, classifying polystable holomorphic bundles with meromorphic connection of rank $r$ and degree $d$ on $X$. In turn, by the Riemann--Hilbert correspondence, it is complex-analytically isomorphic to (the analytification of) the Betti moduli space $\cM_{r,d}^B$, which is a twisted character variety classifying certain semisimple representations of the fundamental group $\pi_1(X\setminus \left\{ x_1 \right\},x_0)$.
  \end{itemize}
  In particular, all these different complex spaces have the same underlying real space $\cM_{r,d}$.
\end{thm}

\subsection*{The twistor family}
\begin{ej}
A hyperkähler manifold comes with a whole sphere of complex structures. Indeed, let $(M,g,I_1,I_2,I_3)$ be a hyperkähler manifold and let $\bm{u}=(u_1,u_2,u_3) \in S^2$ be a unit vector in $\bbR^3$. Set
  \begin{equation*}
I_{\bm{u}} = u_1 I_1 + u_2 I_2 + u_3 I_3 \text{ and } \omega_{\bm{u}} = u_1 \omega_1 + u_2 \omega_2 + u_3 \omega_3.
  \end{equation*}
\textit{Show that $(M,g,I_{\bm{u}})$ is a Kähler manifold, with Kähler form $\omega_{\bm{u}}$.} \textit{Moreover, show that, if $\bm{\mu}$ is a hyperkähler moment map then, for each $\bm{u}\in S^2$, the map $\mu_{\bm{u}}=\bm{\mu}\cdot \bm{u}:M\rightarrow \fk^*$ is a moment map for $\omega_{\bm{u}}$.}
\end{ej}

We observe that, if $(M,g,I_1,I_2,I_3)$ is a hyperkähler manifold, then we can consider the holomorphic $2$-form $\Omega_1=\omega_2 + i \omega_3 \in \Omega^{2,0}(M,I_1)$, which endows $(M,I_1)$ with the structure of a \emph{holomorphic symplectic manifold}. More generally, we can do the same for any $\bm{u}\in S^2$. Indeed, we consider the vector space $V=\mathrm{span}(\omega_1,\omega_2,\omega_3) \subset \Omega^2(M,\C)$ and, for every $\bm{u}\in S^2$, consider the line
\begin{equation*}
L_{\bm{u}} = V \cap \Omega^{2,0}(M,I_{\bm{u}}),
\end{equation*}
which is generated by some $2$-form $\Omega_{\bm{u}}$ which is holomorphic with respect to $I_{\bm{u}}$. This determines a line bundle $L\rightarrow S^2$, called the \emph{twistor family} of holomorphic symplectic forms on $(M,g,I_1,I_2,I_3)$. Moreover, it can be shown \cite[Lemma 3.18]{Neitzke} that the dependence on $\bm{u}$ is holomorphic. More precisely, this means that, if we equip $S^2$ with its standard complex structure, then $L$ determines a holomorphic line bundle over $\bbP^1$, which is canonically isomorphic to $\sO_{\bbP^1}(-2)$.

We can also construct the \emph{twistor space}. This is the manifold
\begin{equation*}
Z= M \times S^2,
\end{equation*}
equipped with the almost complex structure
\begin{equation*}
I(-,\bm{u}) = I_{\bm{u}} \oplus I_{\bbP^1}.
\end{equation*}
This complex structure is in fact integrable, so $Z$ is a complex manifold. It has the following properties:
\begin{enumerate}
\item The projection $\mathrm{pr}_2: Z \rightarrow \bbP^1$ is holomorphic.
\item It carries a twisted fibrewise holomorphic symplectic form $\Omega \in \Omega^{2,0}_{\text{fibre}}(Z)\otimes \mathrm{pr}_2^* \sO_{\bbP^1}(2)$, where $\Omega^{2,0}_{\text{fibre}}:=\wedge^2(T^{1,0}_{\text{vert}}Z)^*$.
\item It carries a real structure $\sigma:Z\rightarrow Z$ covering the real structure on $\bbP^1$ given by $z\mapsto z^{-1}$ and such that $\sigma^*\Omega = \bar{\Omega}$.
\end{enumerate}

\begin{defn}
The \emph{Hodge moduli space} $\cM^{\Hod}_{r,d}$ of rank $r$ and degree $d$ associated with $X$ is the twistor space of the Hitchin moduli space $(\cM_{r,d},g,I_1,I_2,I_3)$.
\end{defn}

This twistor space $\cM^{\Hod}_{r,d}$ is a complex space with a holomorphic fibration $\cM^{\Hod}_{r,d}\rightarrow \bbP^1$. The fibre over $z=0$ is the Dolbeault moduli space $(\cM_{r,d},g,I_1)$, while the fibre over $z=1$ is the de Rham moduli space $(\cM_{r,d},g,I_2)$. The existence of the real structure on twistor space implies that the fibre over $z=\infty$ is the Dolbeault moduli space with its ``opposite'' complex structure $(\cM_{r,d},g,\bar{I}_1)$.

A natural question then is what the complex structure $I_3$ and all the other complex structures correspond to. That is, given a general $\lambda \in \bbP^1$, we want to know what kind of objects does the fibre of $\lambda$ in $\cM^{\Hod}_{r,d}$ parametrize. The answer is in $\lambda$-\emph{connections}.

\begin{defn}
Let $\sV$ be a holomorphic vector bundle. A \emph{holomorphic $\lambda$-connection} $\bm{D}^{(\lambda)}$ on $\sV$ is a $\C$-linear morphism of sheaves
\begin{equation*}
\bm{D}^{(\lambda)}: \sV \longrightarrow \sV \otimes \bm{\Omega}^1_X
\end{equation*}
satisfying the \emph{$\lambda$-twisted holomorphic Leibniz rule}
\begin{equation*}
\bm{D}^{(\lambda)}(fs) = f \bm{D}^{(\lambda)} s +  \lambda  \partial_X f \otimes s,
\end{equation*}
for any local sections $f\in \sO_X(U)$ and $s\in \sV(U)$.
\end{defn}

Note that, for $\lambda=1$, a $\lambda$-connection is just a connection, and for $\lambda=0$, it determines a Higgs bundle. Moreover, for any $\lambda \in \C^*$, if $\bm{D}^{(\lambda)}$ is a $\lambda$-connection, then the operator
\begin{equation*}
\bm{D} = \lambda^{-1} \bm{D}^{(\lambda)}
\end{equation*}
is a connection.
The definition generalizes automatically to meromorphic $\lambda$-connections, which allows us to consider underlying bundles $E$ with $\deg E\neq 0$.  The algebraic and symplectic constructions of the moduli space of holomorphic vector bundles with $\lambda$-connection is completely parallel to the cases we already studied.

\begin{ej}
\textit{Show that, for $\lambda\in \bbP^1$, the fibre over $\lambda$ of the Hodge moduli space is complex-analytically isomorphic to (the analytification of) the moduli space of holomorphic vector bundles with $\lambda$-connection.}
\end{ej}

We remark that, in fact, the map $\bm{D}^{(\lambda)}\mapsto \bm{D}$, for $\lambda \in \C^*$, determines an isomorphism between the moduli space of holomorphic vector bundles with $\lambda$-connection and the de Rham moduli space. However, this map is not defined for $\lambda=0$. Indeed, over $\lambda=0$ we get the moduli space of Higgs bundles, which has a different, non-isomorphic, complex structure.

\subsection*{Moduli of (twisted) $\SL_r$-Higgs bundles and $\PGL_r$-Higgs bundles}
Recall that we defined the Betti moduli space $\cM^{\Bt}_{r,d}$ as the twisted character variety $\cX_{\Pi,\GL_r}^{c_d,z}$ parametrizing classes of linear representations of the fundamental group $\Pi=\pi_1(X\setminus \left\{ x_1 \right\},x_0)$ sending the class $z$ of a loop around $x_1$ contractible in $X$ to the conjugacy class $c_d$ of the matrix $e^{-2\pi i d/r}I_r$. In a similar way as we did for vector bundles in Section \ref{ss:Narasimhan-Seshadri}, we can also consider $\SL_r$ and $\PGL_r$ versions of $\cM^{\Bt}_{r,d}$. The \emph{$d$-twisted $\SL_r$-Betti moduli space} is the twisted character variety
\begin{equation*}
\check{\cM}^{\Bt}_{r,d} = \cM^{\Bt}_d(\SL_r) := \cX_{\Pi,\SL_r}^{c_d,z}
\end{equation*}
parametrizing classes of homomorphisms $\Pi\rightarrow \SL_r(\C)$ mapping $z$ to $c_d$ (regarded as a class in $\SL_r(\C)$). In particular, the complex dimension of $\check{\cM}^{\Bt}_{r,d}$ is
\begin{equation*}
\dim_{\C}\check{\cM}^{\Bt}_{r,d} = 2(r^2-1)(g-1).
\end{equation*}

The $\PGL_r$-character variety
\begin{equation*}
\cX_{\Pi,\PGL_r} = \Hom(\Pi,\PGL_r(\C)) \git \PGL_r(\C)
\end{equation*}
has $r$ connected components (by the same argument as for the $\PU(r)$-character variety). We label these components as $\cX^0_{\Pi,\PGL_r},\dots,\cX^{r-1}_{\Pi,\PGL_r}$, and, for each $d=0,\dots,r-1$, define the \emph{$d$-twisted $\PGL_r$-Betti moduli space} as
\begin{equation*}
\hat{\cM}^{\Bt}_{r,d}= \cM^{\Bt}_{d}(\PGL_r) := \cX^d_{\Pi,\PGL_r}.
\end{equation*}
Note that
\begin{equation*}
\hat{\cM}^{\Bt}_{r,d} = \check{\cM}^{\Bt}_{r,d} / (\Z_r)^{2g}.
\end{equation*}

From the Dolbeault side, the situation is very similar to the case of vector bundles. For any integer $d$, we fix a degree $d$ holomorphic line bundle $\xi$ and consider the map
\begin{align*}
  \cM^{\Dol}_{r,d} & \longrightarrow \bm{T}^* \Pic^d(X) \cong \Pic^d(X) \times H^0(X,\bm{\Omega}^1_X)  \\
 (\sE,\varphi)   & \longmapsto (\det \sE, \tr \varphi).
\end{align*}
We define the \emph{$d$-twisted $\SL_r(\C)$-Dolbeault moduli space} $\check{\cM}^{\Dol}_{r,d}=\cM^{\Dol}_{d}(\SL_r)$ or \emph{moduli space of semistable holomorphic $d$-twisted $\SL_r(\C)$-Higgs bundles} as the preimage of $(\xi,0)$ by the above map. Note that the finite group $\Gamma_r=\Jac(X)[r]$ acts by tensorization on $\cM^{\Dol}_{r,d}$ and on $\Pic^d(X)$, and in turn on $\bm{T}^* \Pic^d(X)$. The moduli space $\cM^{\Dol}_{r,d}$ can then be recovered from $\check{\cM}^{\Dol}_{r,d}$ as
\begin{equation*}
\cM^{\Dol}_{r,d} = (\check{\cM}^{\Dol}_{r,d} \times \bm{T}^* \Pic^d(X))/\Gamma_r.
\end{equation*}
The \emph{$d$-twisted $\PGL_r(\C)$-Dolbeault moduli space} is by definition the quotient
\begin{equation*}
\hat{\cM}^{\Dol}_{r,d}=\cM^{\Dol}_d(\PGL_r):= \check{\cM}^{\Dol}_{r,d}/\Gamma_r.
\end{equation*}
Non-abelian Hodge theory provides real-analytic isomorphisms
\begin{equation*}
\check{\cM}^{\Bt}_{r,d} \cong_{\bbR} \check{\cM}^{\Dol}_{r,d} \ \text{ and } \ \hat{\cM}^{\Bt}_{r,d} \cong_{\bbR} \hat{\cM}^{\Dol}_{r,d}.
\end{equation*}

\begin{rmk}
As in the case of vector bundles, we remark that, when $r$ and $d$ are coprime, the subspace $\check{\cM}^s_{r,d}\subset \check{\cM}_{r,d}$ of stable Higgs bundles (or, equivalently, of irreducible connections) is smooth, but $\hat{\cM}_{r,d}^s$ is not, it is an orbifold, as it arises as a quotient of $\check{\cM}^s_{r,d}$ by the finite group $\Gamma_r$.
\end{rmk}

\begin{rmk}
If $G$ is a complex reductive group, with Lie algebra $\fg$, by a \emph{$G$-Higgs bundle} we mean a pair $(\mathscr{P},\varphi)$ formed by a holomorphic principal $G$-bundle $\mathscr{P}$ and a holomorphic section $\varphi \in H^0(X,\mathrm{Ad}(\mathscr{P}) \otimes \bm{\Omega}^1_X)$. Here, $\mathrm{Ad}(\mathscr{P})=\mathscr{P} \times^{G,\mathrm{Ad}} \mathfrak{g}$ denotes the \emph{adjoint vector bundle}, that is, the bundle of Lie algebras associated to $\mathscr{P}$ and the adjoint action of $G$ on $\mathfrak{g}$. These $G$-Higgs bundles already appear in one of Hitchin's the fundational papers of the theory \cite{Hitchin_Systems}. One can consider (semi)stability conditions and construct a moduli space of $G$-Higgs bundles $\cM(G)$. Non-abelian Hodge theory can be extended to this setting, and it  real-analytically identifies the  moduli space $\cM(G)$ with the character variety $\cX_{\pi_1(X,x_0),G}$ of representations of the fundamental group in $G$. This was originally observed by Simpson \cite{Simpson_HiggsLocal}, who gave an argument using Tannakian categories. Explicit stability conditions and an explicit proof of the analogue of the Hitchin--Simpson theorem for $G$-Higgs bundles can be found in \cite{OscarGothenMundet}.
\end{rmk}

\chapter{The Hitchin system} \label{sec:hitchinsystem}
\section{Integrable systems}
Let $(M,\omega)$ be a finite-dimensional symplectic manifold and consider some functions $f_1,\dots,f_n:M\rightarrow \bbR$. Associated with these functions we have the corresponding \emph{Hamiltonian vector fields} $\bm{v}^{f_1},\dots,\bm{v}^{f_n} \in \Omega^0(M,TM)$, defined by the property
\begin{equation*}
df_i = i_{\bm{v}^{f_i}} \omega.
\end{equation*}
The Lie bracket of vector fields induces the structure of a Lie algebra in the space $\Omega^0(M,TM)$, and the vectors $\bm{v}^{f_1},\dots,\bm{v}^{f_n}$ generate a Lie subalgebra $\mathfrak{h}$ of it. The functions $f_i$ define a map
\begin{align*}
  \mu:M & \longrightarrow \mathfrak{h}^* \\
m    & \longmapsto \sum_{i=1}^n f_i(n) \nu_i,
\end{align*}
where $\left\{ \nu_1,\dots,\nu_n \right\}\subset \mathfrak{h}^*$ is the dual basis of $\left\{ \bm{v}^{f_1},\dots,\bm{v}^{f_n} \right\}$. If these vectors integrate to determine the action of a Lie group $H$ on $M$, with $\Lie H = \mathfrak{h}$, then the map $\mu$ defined above is a moment map for this action.

A special case of this situation is when $n=\tfrac{1}{2} \dim M$ and the vector fields $\bm{v}^{f_i}$ pairwise commute, so the Lie algebra $\mathfrak{h}$ they generate is actually abelian. If the map $\mu$ is proper, then its regular fibres are $n$-dimensional tori. Moreover, these tori are Lagrangian (i.e. the symplectic form $\omega$ vanishes on them), and in tubular neighbourhoods of these tori one can find the so-called \emph{action-angle coordinates}, on which the flows of the vector fields $\bm{v}^{f_i}$ are linear. In that case, we say that the functions $f_1,\dots,f_n:M\rightarrow \bbR$ determine a \emph{completely integrable system}.

The same notion can be generalized to the realm of complex algebraic geometry. In this case we fix a smooth quasi-projective variety $M$ over $\C$, equipped with a holomorphic symplectic form $\Omega\in \Omega^{2,0}(M)$. For such a space, an integral subvariety $N \subset M$ is said to be \emph{Lagrangian} if for a generic point $n\in N$, the symplectic form $\Omega$ vanishes at $T_nN \subset T_nM$. An \emph{algebraically completely integrable system} is then a proper flat morphism $f:M\rightarrow B$, where $B$ is an affine space over $\C$, such that, over the complement $B\setminus \Delta$ of some proper closed subvariety $\Delta\subset B$, it is a Lagrangian fibration whose fibres are isomorphic to abelian varieties. The algebraic counterpart of action-angle coordinates is given by a polarization on the fibres, which allows to explictly solve the equations of movement in terms of theta-functions.

\section{The Hitchin map}
Let $\cM=\cM^\Dol_{r,d}=(\cM_{r,d},I_1,\omega_1)$ be the Dolbeault moduli space parametrizing Higgs bundles on $X$ of rank $r$ and degree $d$. This moduli space is a quasi-projective variety, and it admits the holomorphic symplectic form $\Omega=\Omega_1=\omega_2+i\omega_3$. Hitchin \cite{Hitchin_Systems} found an algebraically completely integrable system on the algebraic symplectic manifold $(\cM,\Omega)$.

The Hitchin map is defined by considering the ``characteristic polynomial'' of a Higgs bundle. What we mean by this is that, if $(\sE,\varphi)$ is a Higgs bundle of rank $r$, then we can, at least formally, consider the polynomial
\begin{equation*}
\det(t - \varphi) = t^r + b_1(\varphi) t^{r-1} + \dots + b_{r-1}(\varphi) t + b_r(\varphi).
\end{equation*}
The coefficient $b_i(\varphi)$ is naturally a section of the line bundle $(\bm{\Omega}^1_X)^{\otimes i}$. Thus, we define the \emph{Hitchin base} as the vector space
\begin{equation*}
\cB = \cB_{r,d} = \bigoplus_{i=1}^r H^0(X,(\bm{\Omega}^1_X)^{\otimes i}).
\end{equation*}
The \emph{Hitchin map} is then defined as
\begin{align*}
  f: \cM & \longrightarrow \cB \\
(\sE,\varphi)    & \longmapsto (b_1(\varphi),\dots,b_r(\varphi)).
\end{align*}

\begin{thm}[Hitchin]
The map $f:\cM \rightarrow \cB$ determines an algebraically integrable system on the space $(\cM,\Omega)$.
\end{thm}

\begin{ej}
\textit{Compute the dimension of $\cB$}. Note that it is indeed half the dimension of $\cM$.  \textbf{Hint}: Use Riemann--Roch.
\end{ej}

The Hitchin map can also be defined in terms of the moduli stack $\mathbf{Higgs}_{r,d}$ sending any test $\C$-scheme $S$ to the groupoid of families of Higgs bundles of rank $r$ and degree $d$ parametrized by $S$, with equivalence. In particular it restricts of the substack $\mathbf{Higgs}^s_{r,d}$ of stable Higgs bundles. The image is still the space $\cB$, and we write
\begin{equation*}
\bm{f}: \mathbf{Higgs}_{r,d} \longrightarrow \cB.
\end{equation*}

\section{The spectral correspondence}
An explicit description of the fibres of $f:\cM \rightarrow \cB$ can be provided in terms of ``spectral data''. This should be thought of as a ``global'' analogue of diagonalizing a matrix.

\begin{ej}
Let $E$ be a complex vector space and $\varphi \in \End E$ an endomorphism of it. \textit{Show that $\varphi$ endows $E$ with the structure of a $\C[t]$-module}. If we regard this $\C[t]$-module as a sheaf over $\bbA^1$, \textit{what is the support of this sheaf?}
\end{ej}

Let $(\sE,\varphi)$ be a Higgs bundle. We can rewrite the map $\varphi:\sE\rightarrow \sE\otimes \bm{\Omega}^1_X$ as a map $(\bm{\Omega}^1_X)^{-1}\rightarrow \End \sE$, which naturally endows $\sE$ with the structure of a module over the sheaf $(\bm{\Omega}^1_X)^{-1}$, and naturally over its symmetric algebra $\mathrm{Sym}^*(\bm{\Omega}^1_X)^{-1}$. The relative spectrum of this algebra is the total space $p:\bm{T}^* X\rightarrow X$ of the holomorphic cotangent bundle of $X$. Note that this space comes equipped with a tautological section $\tau \in \Omega^1(\bm{T}^* X,p^* \bm{\Omega}_X^1)$.

This way, Higgs bundle determines a $\sO_{\bm{T}^*X}$-module, $\sF$. This sheaf is supported on a dimension $1$ subspace $Y_\varphi \subset \bm{T}^* X$, called the \emph{spectral curve} of $(\sE,\varphi)$. The natural projection $p:\bm{T}^* X\rightarrow X$ restricts to a finite flat morphism $\pi:Y_\varphi \rightarrow X$. Generically, one should think about $Y_\varphi$ as parametrizing the eigenvalues of $\varphi$ over $X$. When restricted to $Y_{\varphi}$, the sheaf $\sF$ determines a coherent sheaf $\sL \rightarrow Y_{\varphi}$ of generic rank $1$, and the Higgs bundle $(\sE,\varphi)$ is recovered as
\begin{equation*}
(\sE,\varphi) = (\pi_* \sL, \pi_* (\tau|_{Y_\varphi})).
\end{equation*}
Since $\sE = \pi_* \sL$ is by assumption locally free, and the map $\pi$ is flat, the sheaf $\sL$ must be torsion-free. In particular, if $Y_{\varphi}$ is smooth, then $\sL$ is a line bundle.

\begin{figure}[h]
\centering
\includegraphics[width=0.7\textwidth]{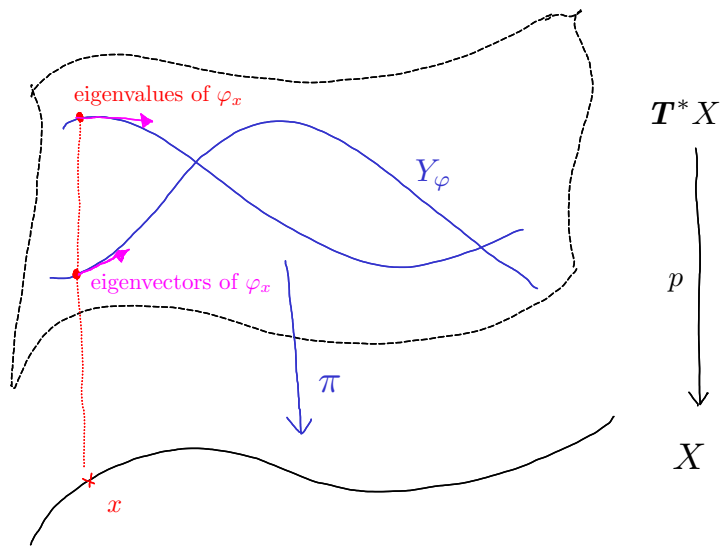}
\caption{The spectral curve}
\end{figure}

The Cayley--Hamilton theorem implies that the curve $Y_\varphi$ in fact only depends on the ``characteristic polynomial'' of $\varphi$, that is, it only depends on its image $b=(b_1,\dots,b_r)=f(\sE,\varphi)$ by the Hitchin map. More precisely, the spectral curve $Y_\varphi=Y_b$ is the zero-locus of the section
\begin{equation*}
\sigma_b = \tau ^r + p^* b_1 \tau ^{r-1} + \dots + p^* b_{r-1} \tau + p^* b_r \in H^0(\bm{T}^* X,p^* \bm{\Omega}_X^1).
\end{equation*}
We have found an equivalence of categories between
\begin{itemize}
\item Higgs bundles $(\sE,\varphi)$ with $f(\sE,\varphi)=b$,
\item torsion-free sheaves of generic rank $1$ on $Y_b$.
\end{itemize}

\begin{ej}
\textit{Show that, if $Y_b$ is irreducible, then, for any torsion free sheaf $\sL$ on $Y_b$ of generic rank $1$, the corresponding Higgs bundle $(\sE,\varphi)$ is stable}. \textbf{Hint}: What would happen if there was a $\varphi$-invariant holomorphic subbundle $\sE'\subset \sE$?
\end{ej}

Moreover, if $Y_b$ is smooth, then isomorphism classes of Higgs bundles $(\sE,\varphi)$ with $f(\sE,\varphi)=b$ are in bijection with the Picard group $\Pic(Y_b)$.
Note that $\Pic(Y_b)$ has many components, labelled by the degree of the line bundle. However, by the Grothendieck--Riemann--Roch formula
\begin{equation*}
\deg \pi_* \sL = \deg \sL + (1-g(Y_b)) - r(1-g),
\end{equation*}
where $g(Y_b)$ is the genus of $Y_b$, prescribing the degree of $\sE=\pi_* \sL$ also fixes the degree of $\sL$. We denote this degree by $\delta=\deg \sL$. We have found an isomorphism
\begin{equation*}
f^{-1}(b) \cong \Pic^{\delta}(Y_b).
\end{equation*}

\begin{ej}
  We can explicitly compute the genus of $Y_b$ in several steps.
  \begin{enumerate}
    \item \textit{Show that the ramification divisor $R\subset Y_b$ is the zero locus of a section of $p^* (\bm{\Omega}^1_X)^{\otimes r(r-1)}$ and thus}
          \begin{equation*}
\deg R = \deg((\bm{\Omega}^1_X)^{\otimes r(r-1)}) = 2r(r-1)(g-1).
          \end{equation*}
    \item \textit{Use the Riemann--Hurwitz formula to compute the genus of $Y_b$}.
  \end{enumerate}
You should get $g(Y_b)=r^2(g-1)+1$. \textit{Did you expect this number?}
\end{ej}

\begin{ej}
\textit{Prove that, when $g\geq 2$, for a generic element $b\in \cB$, the spectral curve $Y_b$ is smooth}. \textbf{Hint}: Use Bertini's theorem and the fact that $(\bm{\Omega}^1_X)^{\otimes r}$ is base-point free.
\end{ej}

Thus, we conclude that the subset of elements $b \in \cB$ such that $Y_b$ is smooth is Zariski open. We denote this subset by $\cB^\diamond \subset \cB$.
We have shown that the Hitchin fibres $f^{-1}(b)$ with $b\in \cB^\diamond$ are canonically isomorphic to the abelian varieties $\Pic^\delta(Y_b)$. Another way to interpret this result is that the Hitchin fibres $f^{-1}(b)$ are torsors under the action of the Jacobian $P_b:=\Jac(Y_b)$ by tensorization. We can put all these together to define a family of abelian varieties $P\rightarrow \cB^\diamond$, and conclude that the restriction $f:\cM|_{\cB^\diamond} \rightarrow \cB^\diamond$ has the structure of a $P$-torsor.

The spectral correspondence, being an equivalence of categories, not only provides a description of the fibres of $f:\cM\rightarrow \cB$ but also of the fibres of the stacky map $\bm{f}:\mathbf{Higgs}_{r,d}\rightarrow \cB$. First, we observe that since smooth spectral curves are in particular irreducible, over $\cB^\diamond$, the map $\bm{f}$ coincides with its restriction to the stable locus $\mathbf{Higgs}_{r,d}^s$. The fibre over a point $b \in \cB^\diamond$ is the Picard stack
\begin{equation*}
\bm{f}^{-1}(b) \cong \mathbf{Pic}^\delta_{Y_b}.
\end{equation*}
In particular, this implies that the stacky fibre $\bm{f}^{-1}(b)$ is a torsor under the stack
\begin{equation*}
\cP_b := \mathbf{Pic}^0_{Y_b} \cong \Jac(Y_b) \times \bbB \C^*.
\end{equation*}
Again, we can see this as a family $\cP \rightarrow \cB^{\diamond}$, and conclude that the restriction $\bm{f}:\mathbf{Higgs}_{r,d}|_{\cB^\diamond}\rightarrow \cB^\diamond$ has the structure of a $\cP$-torsor.

\begin{rmk}
The stack $\cP \rightarrow \cB^{\diamond}$ can be also understood as a moduli stack of $J$-torsors, where $J\rightarrow \cB^{\diamond}$ is some group scheme (namely, it is the multiplicative group on each spectral curve). This endows the Hitchin fibration with the structure of a $J$-gerbe and thus in the literature it is not uncommon to find statements of the sort ``the Hitchin fibration is a gerbe'' or "$\mathbf{Higgs}$ is a gerbe''. It is important to distiguish this gerby structure of the Hitchin map from the fact that the map $\mathbf{Higgs}^s_{r,d}\rightarrow \cM^s$ is naturally a $\C^*$-gerbe. Thus we warn the readers to beware of the appearence of the word ``gerbe'' in different places in the literature, as these could refer to different gerbes.
\end{rmk}

\begin{rmk}
It is also important to remark that this $\Pic$-action in fact extends beyond $\cB^{\diamond}$. Indeed, since for every $b\in \cB$ the spectral cover $\pi:Y_b \rightarrow X$ is flat, if $(\sE,\varphi)$ is a Higgs bundle induced from a coherent sheaf $\sL$ on $Y_b$ then, for every locally free sheaf $\sM$ of degree $0$ on $Y_b$, the tensor product $\sL \otimes \sM$ induces another Higgs bundle on $X$ of the same rank and degree.
\end{rmk}

\section{The Hitchin map for $G$-Higgs bundles. $\SL_n$ vs $\PGL_n$}
In his paper \cite{Hitchin_Systems}, Hitchin already defined his celebrated fibration for $G$-Higgs bundles, where $G$ is an arbitrary complex reductive group with Lie algebra $\mathfrak{g}$. The Hitchin map is ultimately modelled in Chevalley's restriction map. This is the map of affine schemes
\begin{equation*}
\mathfrak{g} \rightarrow \mathfrak{g}\git G,
\end{equation*}
induced by the inclusion $\C[\mathfrak{g}]^G \hookrightarrow \C[\mathfrak{g}]$, where $G$ acts on $\mathfrak{g}$ through the adjoint action. \emph{Chevalley's restriction theorem} asserts that the invariant ring $\C[\mathfrak{g}]^G$ is in fact a polynomial ring, isomorphic to the invariant ring $\C[\mathfrak{t}]^W$, where $\mathfrak{t}\subset \mathfrak{g}$ is the Lie algebra of a maximal torus $T\subset G$ and $W=N_G(T)/T$ is the corresponding \emph{Weyl group}. Let us pick generators $b_1,\dots,b_r$, where $r$ is the rank of $G$, of the invariant ring $\C[\mathfrak{t}]^W$, and let us write $d_i=\deg b_i$. The \emph{Hitchin base for $G$-Higgs bundles} is then defined as
\begin{equation*}
\cB(G) = \bigoplus_{i=1}^r H^0(X,(\bm{\Omega}^1_X)^{\otimes d_i}).
\end{equation*}
Note that the vector bundle $\oplus_{i=1}^r (\bm{\Omega}^1_X)^{\otimes d_i}\cong \mathrm{Fr}(\bm{\Omega}^1_X)\times^{\C^*} (\mathfrak{g}\git G)$ is the associated vector bundle to the frame bundle of $\bm{\Omega}^1_X$ and the $\C^*$-action on $\mathfrak{g}\git G$ induced from the homothecy $\C^*$-action on $\mathfrak{g}$. The \emph{Hitchin map for $G$-Higgs bundles} is the map
\begin{align*}
   \cM(G) & \longrightarrow \cB(G) \\
 (\mathscr{P},\varphi)   & \longmapsto (b_1(\varphi),\dots,b_r(\varphi)).
\end{align*}

In the case of $\mathfrak{g}=\mathfrak{gl}_r(\C)$, the Lie algebra $\mathfrak{t}$ is the affine space $\C^r$ and the Weyl group $W=\mathfrak{S}_r$ is the symmetric group. The invariant polynomials are then the elementary symmetric polynomials, which precisely describe the coefficients of the characteristic polynomial of a square $r\times r$ matrix. For \[\mathfrak{g}=\mathfrak{sl}_r(\C)=\left\{ A\in \mathfrak{gl}_r(\C): \tr A = 0 \right\},\] which is the Lie algebra of both $\SL_r(\C)$ and $\PGL_r(\C)$, the Lie algebra \[\mathfrak{t}=\left\{ (t_1,\dots,t_r) \in \C^r: t_1+\cdots+t_r=0 \right\}\] is isomorphic to $\C^{r-1}$ and the Weyl group $W=\mathfrak{S}_{r-1}$ is the symmetric group acting by permutation on the coordinates of $\C^{r-1}$. The invariant polynomials are the elementary symmetric polynomials in $r$ variables except for \[b_1(t_1,\dots,t_r)=t_1+\cdots+ t_r\] which vanishes automatically on $\mathfrak{t}$. Therefore, we identify the Hitchin base for $\SL_r(\C)$-Higgs bundles (and for $\PGL_r(\C)$-Higgs bundles) as
\begin{equation*}
\cB_0:=\cB(\SL_r(\C))=\cB(\PGL_r(\C)) = \bigoplus_{i=2}^r H^0(X,(\bm{\Omega}^1_X)^{\otimes i}) \subset \cB.
\end{equation*}
Using this framework, we can define the \emph{Hitchin map for the $d$-twisted $\SL_r(\C)$-Dolbeault moduli space} as
\begin{align*}
   \check{f}:\check{\cM}^{\Dol}_{r,d} & \longrightarrow \cB_0 \\
 (\mathscr{E},\varphi)   & \longmapsto (b_2(\varphi),\dots,b_r(\varphi)).
\end{align*}
This map is clearly $\Gamma_r$-equivariant, for $\Gamma_r=\Jac(X)[r]$ acting on $\check{\cM}^{\Dol}$ by tensorization, so it descends to a map
\begin{equation*}
\hat{f}: \hat{\cM}^{\Dol}_{r,d} \longrightarrow \cB_0.
\end{equation*}

Consider an element $b\in \cB_0$ (and in turn of $\cB$). Associated to it we have a spectral cover $\pi:Y_b\rightarrow X$. We assume that $Y_b$ is smooth, and thus that the Hitchin fibre $f^{-1}(b)$ is identified with the abelian variety $\Pic^{\delta}(Y_b)$. The fibre $\check{f}^{-1}(b)$ is then the subspace of $\Pic^{\delta}(Y_b)$ formed by isomorphism classes of holomorphic line bundles $\sL$ on $\Pic^{\delta}(Y_b)$ that admit a trivialization
\begin{equation*}
\sO_X \overset{\sim}{\longrightarrow} \det \pi_* \sL.
\end{equation*}
The fibre $\hat{f}^{-1}(b)$ is simply the quotient $\check{f}^{-1}(b)/\Gamma_r$. In terms of spectral data, an element $\sM$ of the group $\Gamma$ acts on $\check{f}^{-1}(b)$ as $\sL \mapsto \sL \otimes \pi^* \sM$, since we have a natural isomorphism
\begin{equation*}
\pi_* (\sL \otimes \pi^*\sM) \overset{\sim}{\rightarrow} \pi_* \sL \otimes \sM.
\end{equation*}

We have the following duality result, that we prove in the next section.

\begin{thm}[Hausel--Thaddeus \cite{HauselThaddeus}] \label{thm:duality}
If $b\in \cB_0$ is an element such that the spectral curve $Y_b$ is smooth, then the Hitchin fibres $\check{f}^{-1}(b)$ and $\hat{f}^{-1}(b)$ are dual abelian varieties.
\end{thm}

\section{Norm maps and Prym varieties}
\begin{ej}
  Let $\pi:Y\rightarrow X$ be a ramified cover of compact Riemann surfaces. Associated with $\pi$ we have the \emph{Norm map}
  \begin{align*}
    \Nm_\pi: \Pic(Y) & \longrightarrow \Pic(X) \\
      \sO_Y(\sum_i n_i y_i)& \longmapsto \sO_X(\sum_i n_i \pi(y_i)).
  \end{align*}
  Note that it preserves the degree.
\textit{Show that, for any holomorphic line bundle $\sL\rightarrow Y$, we have the formula}
\begin{equation*}
\det \pi_* \sL \cong \det \pi_* \sO_Y \otimes \Nm_\pi \sL.
\end{equation*}
In particular, note that
\begin{equation*}
\Nm_\pi \pi^* \sL \cong \sL^{r},
\end{equation*}
where $r=\deg \pi$, since $\pi_*(\pi^* \sL) \cong \sL \otimes \pi_* \sO_{Y_b}$.
\end{ej}

\begin{defn}
  The Prym variety associated with a ramified cover $\pi:Y\rightarrow X$ of compact Riemann surface is the neutral connected component of the kernel of the norm map
  \begin{equation*}
P(\pi:Y\rightarrow X) = (\ker \Nm_\pi)^0.
  \end{equation*}
\end{defn}

\begin{ej}
Recall that the Jacobian of a compact Riemann surface is naturally self-dual. \textit{Show that, under this self-duality, the dual of the norm map $\Nm_\pi:\Jac(Y)\rightarrow \Jac(X)$ is the pull-back $\pi^*:\Jac(X) \rightarrow \Jac(Y)$.} \textbf{Hint}: Consider the \emph{Abel-Jacobi map}
\begin{align*}
  \mathrm{AJ}_X : X & \longrightarrow \Pic^1(X) \\
  x  & \longmapsto \sO_X(x),
\end{align*}
and similarly $\mathrm{AJ}_Y:Y\rightarrow \Pic^1(Y)$, show that the following diagram commutes
\begin{center}
  \begin{tikzcd}
   Y \ar{r}{\mathrm{AJ}_Y} \ar{d}{\pi} & \Pic^1(Y) \ar{d}{\Nm_\pi} \\
   X \ar{r}{\mathrm{AJ}_X} & \Pic^1(X),
  \end{tikzcd}
\end{center}
and apply the functor $\Jac(-)$ to the diagram.
\end{ej}

\begin{ej}
Let $a:A_1 \rightarrow A_2$ be a homomorphism of abelian varieties and let $\hat{a}:\hat{A}_2 \rightarrow \hat{A}_1$ be the dual map. \textit{Show that if $a$ is injective, then $\hat{a}$ has connected fibres.} More generally, suppose that $K=\ker a$ is a finite group. \textit{Show that the group of connected components of $\ker \hat{a}$ is equal to the character group $\hat{K}=\Hom(K,\C^*)$.} \textbf{Hint}: Consider the isogeny $A_1\rightarrow A_1/K$.
\end{ej}

\begin{ej}
\textit{Show that if $g\geq 2$ and $\pi:Y_b\rightarrow X$ is a spectral cover associated with $b\in \cB_0$ and such that $Y_b$ is smooth, then the map $\pi^*:\Jac(X)\rightarrow \Jac(Y_b)$ is injective. Conclude that}
\begin{equation*}
\check{P}_b := P(\pi:Y_b \rightarrow X) = \ker \Nm_\pi.
\end{equation*}
\textit{Moreover, conclude that the dual of the Prym variety $\check{P}_b$ is the variety}
\begin{equation*}
\hat{P}_b= \check{P}_b/\pi^* \Gamma_r,
\end{equation*}
where $\Gamma_r = \Jac(X)[r]$ is the group of $r$-torsion points of $\Jac(X)$, for $r=\deg \pi$.

\textbf{Hint}: For the first statement, let $\sL$ be a nontrivial bundle with $\pi^*\sL$ trivial. Then $\pi_*(\pi^* \sL)\cong \pi_*\sO_{Y_b}$. But we also have $\pi_*(\pi^* \sL) \cong \sL \otimes \pi_* \sO_{Y_b}$. Use that
\begin{equation*}
\pi_*(\sO_{Y_b}) \cong \sO_X \oplus (\bm{\Omega}^1_X)^{\otimes -1} \oplus (\bm{\Omega}^1_X)^{\otimes -2} \oplus \dots \oplus (\bm{\Omega}^1_X)^{\otimes -r+1},
\end{equation*}
to conclude that either $\sL$ is trivial or $\deg \bm{\Omega}^1_X = 0$, reaching a contradiction. For the second statement, dualize the map $\Nm_\pi:\Jac(Y_b)\rightarrow \Jac(X)$ to find that the dual of $\check{P}_b$ is the abelian variety $\hat{P}_b=\Jac(Y_b)/\pi^* \Jac(X)$. Construct a polarization $\rho:\check{P}_b \rightarrow \hat{P}_b$ by composing the inclusion $\check{P}_b \hookrightarrow \Jac(Y_b)$ with the projection $\Jac(Y_b)\rightarrow \hat{P}_b$, and show that $\rho(\check{P}_b)\subset \check{P}_b$, so we can write
\begin{equation*}
\hat{P}_b \cong \check{P}_b / \check{P}_b \cap \pi^* \Jac(X) \cong \check{P}_b / \pi^* \Gamma_r.
\end{equation*}
\end{ej}

The duality theorem of Hausel--Thaddeus, Theorem \ref{thm:duality}, now follows from the following.

\begin{prop}
  Let $\pi:Y_b\rightarrow X$ be a spectral cover associated with $b\in \cB_0$ and such that $Y_b$ is smooth. There is a natural isomorphism
  \begin{equation*}
\check{P}_b \overset{\sim}{\longrightarrow} \check{f}^{-1}(b).
  \end{equation*}
  In turn, there is a natural isomorphism
  \begin{equation*}
\hat{P}_b \overset{\sim}{\longrightarrow} \hat{f}^{-1}(b).
  \end{equation*}
\end{prop}

\begin{proof}
Consider a holomorphic line bundle $\sL\rightarrow Y_b$. We observe that the line bundle $\det \pi_* \sL$ is trivial if and only if $\Nm_\pi \sL$ is isomorphic to \[\det \pi_* \sO_Y \cong (\bm{\Omega}^1_X)^{\otimes n(n-1)/2}\cong \Nm_\pi(\pi^*(\bm{\Omega}^1_X)^{\otimes (n-1)/2}).\]
Therefore, if we denote $\sM =\pi^*(\bm{\Omega}^1_X)^{\otimes (n-1)/2}$, we conclude that $\sL \otimes \sM^{-1}$ must lie in the Prym variety $\check{P}_b$. The map
\begin{align*}
  \check{P}_b & \longrightarrow \check{f}^{-1}(b)  \\
  \sL  & \longmapsto \sL \otimes \sM,
\end{align*}
gives the desired isomorphism.
\end{proof}

In terms of torsors, what we have proven is that the fibres $\check{f}^{-1}(b)$ and $\hat{f}^{-1}(b)$ are torsors under the dual abelian varieties $\check{P}_b$ and $\hat{P}_b$. Globally, we can consider the open subset $\cB^{\diamond}_0\subset \cB$ of elements $b$ with $Y_b$ smooth, and consider dual families of abelian varieties $\check{P}\rightarrow \cB^{\diamond}_0$ and $\hat{P}\rightarrow \cB^\diamond_0$. The Hitchin fibrations $\check{f}:\check{\cM}\rightarrow \cB^{\diamond}_0$ and $\hat{f}:\hat{\cM}\rightarrow \cB^{\diamond}_0$ are torsors under $\check{P}$ and $\hat{P}$, respectively.

As in the $\GL_n$ case, we can also consider a version $\check{\bm{f}}: \mathbf{Higgs}_d(\SL_r) \rightarrow \cB^\diamond_0$ whose fibres are torsors under some stacky versions of the Prym variety. More precisely, note that the norm map can be naturally extended to the Picard stacks, to yield a map
\begin{equation*}
\mathbf{Nm}_\pi: \mathbf{Pic}_{Y_b} \rightarrow \mathbf{Pic}_X.
\end{equation*}
At the level of automorphism groups, this map determines the map $\C^* \rightarrow \C^*$, $z \mapsto z^r$.
The corresponding \emph{Prym stack} is then the neutral connected component of the kernel of this map
\begin{equation*}
 \check{\cP}_b:= \bm{P}(\pi:Y_b \rightarrow X) = (\ker \mathbf{Nm}_\pi)^0.
\end{equation*}
The points of $\check{\cP}_b$ coincide with the points of $\check{P}_b$, but the automorphism groups are naturally isomorphic to the group of roots of unity $\mu_r$. That is,
\begin{equation*}
\check{\cP}_b \cong \check{P}_b \times \bbB \mu_r.
\end{equation*}
This determines a family $\check{\cP}\rightarrow \cB^\diamond_0$, and the Hitchin fibration $\check{\bm{f}}: \mathbf{Higgs}_d(\SL_r) \rightarrow \cB^\diamond_0$ has the structure of a $\check{\cP}$-torsor.
On the other hand, since $\PGL_r(\C)$ has no center, the stack $\mathbf{Higgs}_d^s(\PGL_r)$ of stable $\PGL_r(\C)$-Higgs bundles is just the moduli space $\hat{\cM}$ and thus the restriction $\hat{\bm{f}}:\mathbf{Higgs}_d^s(\PGL_r)\rightarrow \cB^{\diamond}_0$ is just the map $\hat{f}:\hat{\cM}\rightarrow \cB^{\diamond}_0$.

The duality of abelian varieties can be extended to determine a duality for stacks of this form, by defining the \emph{dual stack} of $\check{\cP}_b$ as
\begin{equation*}
\hat{\cP}_b := \Hom(\check{P}_b \times \bbB \mu_r, \bbB \C^*) \cong \hat{P}_b \times \Z_r.
\end{equation*}
Note that under this notion of duality, the group of automorphisms is exchanged with the group of connected components and vice-versa. Therefore, the dual of a stacky fibre $\hat{\bm{f}}^{-1}(b)$ is the whole fibre of $b$ through the map $\hat{\bm{f}}: \mathbf{Higgs}^s(\PGL_r)\rightarrow \cB^{\diamond}_0$, without fixing $d$, not just one connected component.

Hausel and Thaddeus \cite{HauselThaddeus} give an interpretation of this. They consider the trivial $\mu_r$-gerbe $\check{\beta}=\check{\cP}_b\rightarrow \check{P}_b$ and show that, for each $e$ coprime with $r$, the duality explained above identifies the component $\hat{P}_b^e$ with the set of trivializations of the gerbe $\check{\beta}^{\otimes e}$ over $\check{P}_b$. More generally, we have the following.

\begin{thm}[Gerby duality of Hausel--Thaddeus]
  Let $d$ and $e$ be two numbers coprime with $r$ and consider the moduli spaces $\cM_d(\SL_r)$ and $\cM_e(\PGL_r)$ which are endowed with $\mu_r$-gerbes \[\check{\beta}:\mathbf{Higgs}_d(\SL_r) \rightarrow \cM_d(\SL_r)\] and \[\hat{\beta}: \mathbf{Higgs}_e(\SL_r)/\Gamma_r \rightarrow \cM_e(\PGL_r).\] Consider the Hitchin fibrations $\check{f}_d:\cM_d(\SL_r) \rightarrow \cB_0$ and $\hat{f}_e: \cM_e(\PGL_r)\rightarrow \cB_0$.
  For any $b\in \cB_0^\diamond$ we have the following equivalences:
\begin{itemize}
\item The fibre $\hat{f}_e^{-1}(b)$ is identified with the set of trivializations of the gerbe $\check{\beta}^{\otimes e}$ over the fibre $\check{f}_d^{-1}(b)$.
\item The fibre $\check{f}_d^{-1}(b)$ is identified with the set of trivializations of the gerbe $\hat{\beta}^{\otimes d}$ over the fibre $\hat{f}_e^{-1}(b)$.
\end{itemize}
\end{thm}

\section{Mirror symmetry}
\subsection*{Calabi--Yau manifolds}
A \emph{Calabi--Yau manifold} $(M,\omega,\nu)$ is a Kähler manifold $(M,\omega)$ equipped with a trivialization of its canonical line bundle $\bm{\Omega}^n_M$ (i.e. a holomorphic $n$-form $\nu \in \Omega^{n,0}(M)$), where $n$ is the complex dimension of $M$. \emph{Mirror symmetry} is a general theoretical framework, motivated from Physics, which predicts the existence of a ``mirror partner'' $\check{M}$ to a Calabi--Yau manifold $M$. This is another Calabi--Yau manifold with exchanged deformation spaces of the complex and Kähler structures. An important part of this framework are two auxiliary ``$B$-fields'' $B$ on $M$ and $\check{B}$ on $\check{M}$, which Hitchin \cite{Hitchin_SpecialLagrangian} interprets as $\U(1)$ gerbes on $M$ and $\check{M}$, respectively. More precisely, the mirror partner $\check{M}$ of a Calabi--Yau manifold is not necessarily a manifold, but rather could be a Calabi--Yau \emph{orbifold}. For our purposes, this shall be simply a quotient stack $[\tilde{M}/\Gamma]$, where $\tilde{M}$ is a Calabi--Yau manifold and $\Gamma$ is a finite group acting on $\tilde{M}$ by biholomorphisms.

\subsection*{SYZ mirror symmetry}
Strominger, Yau and Zaslow \cite{SYZ} proposed a setting for the construction of mirror partners. A submanifold $L$ of a Calabi--Yau manifold $(M,\omega,\nu)$ is \emph{special Lagrangian} if $\omega|_L=0$ and $\mathrm{Im}\ \nu|_L = 0$. Two Calabi--Yau orbifolds of complex dimension $n$ equipped with $B$-fields $(M,B)$ and $(\check{M},\check{B})$ are \emph{SYZ mirror partners} if there exists an orbifold $A$ of real dimension $n$ and smooth surjections $f:M\rightarrow A$ and $\check{f}:\check{M}\rightarrow A$ such that, for every $a\in A$ which is a regular value of $f$ and $\check{f}$, the fibres $L_a=f^{-1}(a)$ and $\check{L}_a=\check{f}^{-1}(a)$ special Lagrangian tori which are dual in the sense that there are the following equivalences, depending smoothly on $a$:
\begin{itemize}
\item The fibre $L_a$ is identified with the set of trivializations of the gerbe $\check{B}$ over the fibre $\check{L}_a$.
\item The fibre $\check{L}_a$ is identified with the set of trivializations of the gerbe $B$ over the fibre $L_a$.
\end{itemize}

Hyperkähler manifolds are a particularly special case of Calabi--Yau manifolds. Indeed, recall that if $(M,g,I_1,I_2,I_3)$ is a hyperkähler manifold of real dimension $4n$ with Kähler forms $\omega_1$, $\omega_2$ and $\omega_3$, then it is holomorphic symplectic with respect to each of the $I_i$ by taking $\Omega_i=\omega_j + i \omega_k$, and in particular it is Calabi--Yau by putting $\nu_i=\Omega_i^n \in \Omega^{2n,0}(M,I_i)$.

\begin{ej}
Let $(M,g,I_1,I_2,I_3)$ be a hyperkähler manifold. \textit{Show that if a submanifold $L\subset M$ is a complex Lagrangian submanifold of $(M,I_1,\Omega_1)$ then it is an special Lagrangian submanifold of $(M,I_2,\nu_2)$}.
\end{ej}

The result of Hausel--Thaddeus \cite{HauselThaddeus} can then be reinterpreted in these terms as follows.

\begin{thm}[SYZ mirror symmetry for Hitchin fibrations]
  Let $d$ and $e$ be two numbers coprime with $r$.
The de Rham moduli space $\cM^{\dR}_d(\SL_r)$, equipped with the Calabi--Yau structure $\check{\nu}_{\dR}=\Omega_{\dR}^{\dim_\R \cM^{\dR}_d(\SL_r)/4}$ and with the $B$-field $\check{\beta}^{\otimes e}$ and the de Rham moduli space $\cM^{\dR}_e(\PGL_r)$, equipped with the Calabi--Yau structure $\hat{\nu}_{\dR}=\Omega_{\dR}^{\dim_\R \cM^{\dR}_e(\PGL_r)/4}$ and with the $B$-field $\hat{\beta}^{\otimes d}$ are SYZ mirror partners, with respect to the Hitchin fibrations $\check{f}_d:\cM_d(\SL_r) \rightarrow \cB_0$ and $\hat{f}_e: \cM_e(\PGL_r)\rightarrow \cB_0$.
\end{thm}

\subsection*{$E$-polynomials}
Calabi--Yau mirror partners which are compact are expected to satisfy the identity $h^{p,q}(\check{M})=h^{\dim M - p,q}(M)$ between their Hodge numbers. On the other hand, compact hyperkähler manifolds satisfy the identity $h^{p,q}(M)=h^{\dim M - p,q}(M)$. Therefore, compact mirror partners which are hyperkähler satisfy $h^{p,q}(\check{M})=h^{p,q}(M)$. Hausel and Thaddeus were motivated by the conjecture that this equality of Hodge numbers could in fact hold for the moduli spaces $\cM^{\dR}_d(\SL_r)$ and $\cM^{\dR}_e(\PGL_r)$. To make their conjecture precise, we need to consider a certain invariant encoding the information about the Hodge numbers, the $E$-polynomial.

If $M$ is a complex algebraic variety with pure Hodge structure, we define its \emph{$E$-polynomial} as
\begin{equation*}
E(M;u,v) = \sum_{i,j} (-1)^{i+j} h^{i,j}(M) u^i v^j.
\end{equation*}
More generally, if $M$ is any complex algebraic variety, we can consider the weight filtration on its compactly-supported cohomology $H^*_c(M)$ induced by its mixed Hodge structure, and define
\begin{equation*}
E(M;u,v) = \sum_{i,j} (-1)^{i+j} h^{i,j}(\mathrm{Gr}^W_{p+q} H^i_c(M))u^pv^q.
\end{equation*}

For an orbifold equipped with a gerbe, we can also defined an invariant, which is a slight modification of the usual $E$-polynomial. Let $[M/\Gamma]$ be an orbifold obtained as the stacky quotient of a complex algebraic manifold $M$ by a finite group $\Gamma$. For any element $\gamma \in \Gamma$, we denote by $M^\gamma\subset M$ the subspace of fixed points of $\gamma$. Suppose also that $B \rightarrow [M/\Gamma]$ is a $\C^*$-gerbe. Equivalently, $B\rightarrow M$ is a $\Gamma$-equivariant $\C^*$-gerbe. Explicitly, this means that there is a character $\kappa:\Gamma \rightarrow \C^*$ such that the associated $2$-cocycle $b_{ijk}$ satisfies the condition
\begin{equation*}
b_{ijk}(\gamma \cdot m) = \kappa(\gamma) b_{ijk}(m),
\end{equation*}
for every $\gamma \in \Gamma$ and every $m\in M$.
In particular, if we restrict $B$ to the fixed points $M^\gamma$ of some nontrivial element $\gamma \in \Gamma$, then $$b_{ijk}(m)=b_{ijk}(\gamma \cdot m)=\kappa(\gamma) b_{ijk}(m),$$ so $b_{ijk}$ is trivial, and thus it comes from some line bundle $L_{B,\gamma}\rightarrow M^\gamma$. This line bundle is equivariant under the action of the centralizer $C_\gamma \subset \Gamma$. Moreover, if the gerbe $B$ was induced from a finite subgroup $A\subset \C^+$, then the line bundle $L_{B,\gamma}$ is in fact a local system. In that case, we can consider cohomology with local coefficients on $L_{B,\gamma}$, and the corresponding $E$-polynomial $E(M^\gamma/C_\gamma,L_{B,\gamma};u,v)$. Finally, we define the \emph{fermionic shift} of $\gamma \in \Gamma$ as $F(\gamma)=\sum w_i$, where $\gamma$ acts on $TM|_{M^\gamma}$ with eigenvalues $e^{2\pi i w_i}$, $w_i\in [0,1)$. The \emph{stringy $E$-polynomial of $[M/\Gamma]$ with respect to the gerbe $B$} is then defined as
\begin{equation*}
E_{\mathrm{st}}(M,B;u,v) = \sum_{[\gamma] \in [\Gamma]} E(M^\gamma/C_\gamma,L_{B,\gamma};u,v)(uv)^{F(\gamma)},
\end{equation*}
where $[\gamma]$ runs through the set $[\Gamma]$ of conjugacy classes of $\Gamma$.

\subsection*{Topological mirror symmetry}
For moduli spaces of Higgs bundles, Hausel and Thaddeus conjectured the following.

\begin{thm}[Topological mirror symmetry]
  For $d$ and $e$ coprime with $r$, we have
  \begin{equation*}
E(\cM_d^{\Dol}(\SL_r); u,v) = E_{\mathrm{st}}(\cM^{\Dol}_e(\PGL_r), \hat{\beta}^{\otimes d}; u,v).
  \end{equation*}
\end{thm}

Hausel--Thaddeus proved this statement for $r=2$ and $3$ in their original paper, and conjectured it for general $r$. The general proof is due to Groechenig, Wyss and Ziegler \cite{GWZ_Mirror}, using a method of $p$-adic integration. A different proof using an Ngô-style support theorem was provided by Maulik and Shen \cite{MaulikShen_Endoscopic}. We reproduce the argument for rank $2$ in Section \ref{ss:topologicalmirrorproof}.

The above formula of $E$-polynomials can be slightly unraveled if we consider the action of $\Gamma=\Gamma_r$ on $\cM_d^\Dol(\SL_r)$ and thus on its compactly-supported cohomology $H_c^*(\cM_d^\Dol(\SL_r))$. We obtain a decomposition in terms of the group of characters $\hat{\Gamma}=\Hom(\Gamma,\C^*)$,
\begin{equation*}
H^*_c(\cM_d^\Dol(\SL_r)) = \bigoplus_{\kappa \in \hat{\Gamma}} H^*_c(\cM_d^\Dol(\SL_r))_\kappa.
\end{equation*}
We can thus decompose
\begin{equation*}
E(\cM_d^\Dol(\SL_r);u,v) = \sum_{\kappa \in \hat{\Gamma}} E_\kappa(\cM_d^\Dol(\SL_r);u,v),
\end{equation*}
where
\begin{equation*}
E_\kappa(\cM_d^\Dol(\SL_r);u,v) = \sum_{i,j} (-1)^{i+j} h^{i,j}(\mathrm{Gr}^W_{p+q} H^i_c(\cM_d^\Dol(\SL_r)))_\kappa u^pv^q.
\end{equation*}

We obtain the equality
\begin{equation*}
\sum_{\kappa \in \hat{\Gamma}} E_\kappa(\cM_d^\Dol(\SL_r);u,v) = \sum_{\gamma \in \Gamma} E(\cM_e^\Dol(\SL_r)^\gamma/\Gamma,L_{\hat{\beta}^{\otimes d},\gamma};u,v)(uv)^{F(\gamma)}.
\end{equation*}
Note that, since $\Gamma$ is commutative, we have $[\gamma]=\gamma$ and $C_\gamma=\Gamma$ for every $\gamma \in \Gamma$. We remark that there are the same number of terms on each side of the equality. In fact, there is a canonical way to identify $\Gamma=\Jac(X)[r]\cong H^1(X,\Z_r)$ with $\hat{\Gamma}$ through the \emph{Weil pairing}, which is the canonical pairing
\begin{equation*}
w: H^1(X,\Z_r) \times H^1(X,\Z_r) \longrightarrow H^2(X,\Z_r)=\Z_r
\end{equation*}
naturally induced by Poincaré duality. This pairing induces a canonical isomorphism $w:\hat{\Gamma}\rightarrow \Gamma$. The topological mirror symmetry formula is then unraveled as the formula
\begin{equation}
E_\kappa(\cM_d^\Dol(\SL_r);u,v) = E(\cM_e^\Dol(\SL_r)^\gamma/\Gamma,L_{\hat{\beta}^{\otimes d},\gamma};u,v)(uv)^{F(\gamma)}
\end{equation}
for each $\kappa\in \hat{\Gamma}$, with $\gamma=w(\kappa)$. Note that the equality is trivially satisfied for $\kappa=0$.

\section{Langlands duality}
The groups $\SL_r(\C)$ and $\PGL_r(\C)$ are examples of Langlands dual groups. Recall that a complex semisimple group $G$ is determined by a semisimple Lie algebra $\fg$ and by its centre $Z_G$ and its fundamental group $\pi_1(G)$, which are finite abelian groups. The data $(\fg,Z_G,\pi_1(G))$ can be dualized in the following way. The dual semisimple Lie algebra $\fg^*$ is the semisimple Lie algebra determined by the dual root system, for example, the dual of $\mathfrak{sl}_r(\C)$ is $\mathfrak{sl}_r(\C)$, but the dual of $\mathfrak{so}_{2r+1}(\C)$ is $\mathfrak{sp}_{2r}(\C)$. The \emph{Cartier duals} of the groups $Z_G$ and $\pi_1(G)$ are their sets of characters $\hat{Z}_G:=\Hom(Z_G,\C^*)$ and $\hat{\pi}_1(G):=\Hom(\pi_1(G),\C^*)$. Note that these are also finite groups. For example, the Cartier dual of a cyclic group $\Z_r=\Z/r\Z$ is the group of roots of unity $\mu_r \subset \C^*$ (which is of course isomorphic to $\Z_r$, since $\C$ is algebraically closed). Therefore, the data $(\fg^*,\hat{\pi}_1(G),\hat{Z}_G)$ determines another semisimple group $G^\vee$ with $Z_{G^\vee}=\hat{\pi}_1(G)$ and $\pi_1(G^\vee)=\hat{Z}_G$. This group $G^\vee$ is called the \emph{Langlands dual} of $G$.

The semisimple Lie algebra $\fg$ is equipped the \emph{Killing form}, which induces an isomorphism $\fg \git G \rightarrow \fg^* \git G^\vee$. This induces an isomorphism of the Hitchin bases $\cB(G)\rightarrow \cB(G^\vee)$ and thus we can consider the corresponding Hitchin fibrations $\mathbf{Higgs}_G\rightarrow \cB(G)$ and $\mathbf{Higgs}_{G^\vee}\rightarrow \cB(G^\vee)$ as mapping over the same space.

Donagi and Gaitsgory \cite{DonagiGaitsgory} proved a generalization of the spectral correspondence for arbitrary reductive groups. More precisely, they showed that there is a Zariski open subspace $\cB(G)^\diamond \subset \cB(G)$ over which $\mathbf{Higgs}_G$ is a torsor under a stack $\cP(G)\rightarrow \cB(G)$, which is a family of stacks of the form
\begin{equation*}
\cP_b(G) = \bbB Z(G) \times P_b(G) \times \pi_1(G),
\end{equation*}
where $P_b(G)$ is a certain abelian variety associated with the \emph{cameral cover} $\pi_b:\tilde{X}_b \rightarrow X$ obtained as a pullback of the natural quotient map $\mathfrak{t}\rightarrow \mathfrak{t}/W$. Later, Donagi and Pantev \cite{DonagiPantev} showed that the varieties $P_b(G)$ and $P_b(G^\vee)$ are dual abelian varieties, and thus
\begin{align*}
\Hom(\cP(G),\bbB\C^*) &= \bbB \hat{\pi}_1(G) \times \hat{P}(G) \times \hat{Z}(G) \\ &= \bbB Z(G^\vee) \times P(G^\vee) \times \pi_1(G^\vee) = \cP(G^\vee).
\end{align*}

\begin{rmk}
  We can say a few more words in the case where $G$ is simply connected. In that case, what Donagi and Gaitsgory proved is that $\cP_b(G)$ is the stack $\mathbf{Bun}^W_T(\tilde{X}_b)$ of (strongly) $W$-equivariant $T$-bundles on $\tilde{X}_b$, where $T\subset G$ is a maximal torus. In this case, it is clear that
  \begin{equation*}
\mathbf{Bun}^W_T(\tilde{X}_b) = (\bbB T \times (\Jac(\tilde{X}_b)\otimes \Lambda) \times \Lambda )^W = \bbB Z(G) \times (\Jac(\tilde{X}_b)\otimes \Lambda)^{W,0},
  \end{equation*}
  where $\Lambda=\Hom(\C^*,T)$ is the cocharacter lattice of $T$. The dual of this stack is
  \begin{equation*}
\hat{\cP}_b(G) = (\Jac(\tilde{X}_b)\otimes \Lambda)_{W} \times \pi_1(G^\vee).
  \end{equation*}
The remaining step is proving that $P_b(G^\vee)$ coincides with $(\Jac(\tilde{X}_b)\otimes \Lambda)_{W}$. It suffices to provide an isomorphism of the singular homology groups \[H_1(P_b(G^\vee)) \cong H_1((\Jac(\tilde{X}_b)\otimes \Lambda)_{W}) \cong H^1(\tilde{X}_b,\Lambda)_{W,\text{torsion free}}.\]
Donagi and Pantev manage to prove the above by using Poincaré duality and describing the homology $P_b(G^\vee)$ in terms of the local system $(\pi_{b,*}^0 \Lambda^\vee)^W$, where $\pi_b^0$ is the restriction of the cameral cover to its unramified locus.
\end{rmk}

The \emph{Fourier--Mukai transform}\footnote{We refer the reader to Huybrechts book \cite{Huybrechts} for an introduction to Fourier--Mukai transforms.} relates the derived categories of coherent sheaves on an abelian variety $P$ and on its dual $\check{P}$. All these notions can be generalized to stacks $\cP$ as above, and one can obtain an equivalence
\begin{equation*}
\mathrm{FM}: D^b(\mathrm{Coh}(\cP(G))/\cB^\diamond(G)) \longrightarrow D^b(\mathrm{Coh}(\cP(G^\vee))/\cB^\diamond(G^\vee)).
\end{equation*}
A (conjectural) extension of this isomorphism beyond the $\diamond$-locus would yield an equivalence
\begin{equation}\label{eq:Dolbgeolanglands}
D^b(\mathrm{Coh}(\mathbf{Higgs}_G/\cB(G))) \longrightarrow D^b(\mathrm{Coh}(\mathbf{Higgs}_{G^\vee}/\cB(G^\vee))).
\end{equation}
Donagi and Pantev \cite{DonagiPantev} interpreted this as a deformation or ``classical limit'' of the \emph{geometric Langlands program}, which roughly predicts an equivalence
\begin{equation} \label{eq:geolanglands}
D^b(\mathrm{DMod}(\mathbf{Bun}_G)) \longrightarrow D^b(\mathrm{Coh}(\mathbf{Conn}_{G^\vee})),
\end{equation}
between the derived category of $D$-\emph{modules} on $\mathbf{Bun}_G$ and the derived category of coherent sheaves on the stack of holomorphic $G^\vee$-connections.

\begin{rmk}
It is appropriate to mention that the geometric Langlands conjecture might no longer be a conjecture. A proof of (a refined and corrected version of) the equivalence \eqref{eq:geolanglands} above has been recently made public \cite{ProofGLC}. The statement \eqref{eq:Dolbgeolanglands} in ``Dolbeault terms'' remains conjectural to this day.
\end{rmk}

To close this circle of ideas, we also mention that a ``physical interpretation'' of the geometric Langlands program was provided by Kapustin and Witten \cite{KapustinWitten}. In particular, they relate the geometric Langlands equivalence with homological mirror symmetry for the de Rham moduli space $\cM^{\dR}_G$. More precisely, if $(M,\omega,\nu)$ and $(\check{M},\check{\omega},\check{\nu})$ are mirror Calabi--Yau partners Kontsevich's \emph{homological mirror symmetry} \cite{KontsevichHMS} predicts a derived equivalence
\begin{equation*}
D^b(\mathrm{Coh}(M)) \rightarrow \mathrm{Fuk}(\check{M},\check{\omega})
\end{equation*}
between the derived category of coherent sheaves on $M$ (which is determined by the holomorphic structure of $M$) and the \emph{Fukaya category} of $\check{M}$, a certain category related with the Lagrangian submanifolds on $(\check{M},\check{\omega})$, and thus determined by the symplectic structure of $(\check{M},\check{\omega})$. One could roughly interpret the results of Donagi--Pantev as the fact that $\cM^\dR(G)$ and $\cM^{\dR}(G^\vee)$ are SYZ mirror partners (again, very roughly, since these might not be orbifolds). Homological mirror symmetry would then predict an equivalence
\begin{equation*}
D^b(\mathrm{Coh}(\cM(G^\vee),I_{\dR})) \rightarrow \mathrm{Fuk}(\cM(G),\omega_{\dR}).
\end{equation*}
Kapustin and Witten gave a physical interpretation of the Fukaya category $\mathrm{Fuk}(\cM(G),\omega_{\dR})$ as the category of $D$-modules on $\mathbf{Bun}_{G}$, recovering in this way the geometric Langlands equivalence.
More generally, Kapustin and Witten proposed some form of ``hyperkähler enhanced'' mirror symmetry, which takes into account the whole hyperkähler structure of $\cM(G)$. Finding and relating submanifolds of $\cM(G)$ which enter this hyperkähler enhanced mirror framework --for example ``BAA branes'', which are supported on holomorphic Lagrangian submanifolds of $\cM^\Dol(G)$, or ``BBB branes'', supported on hyperkähler submanifolds-- remains a very active topic of research to this day. For more details, we refer the reader to \cite{HauselEnhanced}.

\chapter{Global topology in low rank} \label{sec:global}
\section{Poincaré polynomials}
The main upshot of non-abelian Hodge theory is that these three: Betti, de Rham and Dolbeault moduli spaces have the same underlying topology. Therefore, if we want to understand some topological properties of the character variety, it might just be convenient to study Higgs bundles, and vice-versa. We illustrate this by reviewing some celebrated computations of Betti numbers in the cases of rank $r=2$ and $d=1$. Thus, in this chapter we denote $\cN=\cN_{2,1}$, $\check{\cN}=\cN_{1}(\SL_2)$ and $\hat{\cN}=\cN_1(\PGL_2)$ for the moduli spaces of vector bundles, and $\cM=\cM^\Dol_{2,1}$, $\check{\cM}=\cM^\Dol_{1}(\SL_2)$ and $\hat{\cM}=\cM^\Dol_1(\PGL_2)$ for the corresponding moduli spaces of Higgs bundles. We also denote $\Gamma=\Gamma_2=\Jac(X)[2]$.

For any graded $\C$-algebra $A=\bigoplus_{i} A_i$, we define its \emph{Poincaré series}, as the formal sum
\begin{equation*}
P_t(A) = \sum_i \dim_{\C}(A_i) t^i.
\end{equation*}
In particular, if $M$ is a smooth manifold, we define its \emph{Poincaré polynomial} as the Poincaré series of its $\C$-valued singular cohomology, that is
\begin{equation*}
P_t(M):= P_t(H^*(M,\C)) = \sum_i b_i t^i,
\end{equation*}
where the coefficients $b_i=b_i(M)=\dim_{\C} H^i(M,\C)$ are the \emph{Betti numbers} of $M$.

\begin{ej}
\textit{Prove that the Poincaré polynomial of a circle is} \[P_t(S^1)=(t+1).\] \textit{In turn, show that the Poincaré polynomial of a $k$-dimensional torus $T^k$ is \[P_t(T^k)=(t+1)^k.\]}
\end{ej}

\begin{thm}[Harder--Narasimhan]
  The Poincaré polynomial of $\check{\cN}$ is
  \begin{equation} \label{eq:poincN}
P_t(\check{\cN}) = \frac{(1+t^3)^{2g}-t^{2g}(1+t)^{2g}}{(1-t^2)(1-t^4)}.
  \end{equation}
\end{thm}

\begin{thm}[Hitchin]
The Poincaré polynomial of $\check{\cM}$ is
\begin{align} \label{eq:poincMthm}
  P_t&(\check{\cM})= \sum_{i=1}^{6g-6} b_i t^i= \\
  &\frac{(1+t^3)^{2g}}{(1-t^2)(1-t^4)} - \frac{t^{4g-4}}{4(1-t^2)(1-t^4)} \left[(1+t^2)^2(1+t)^{2g} - (1+t)^4(1-t)^{2g} \right] \\
  &-(g-1) t^{4g-3} \frac{(1+t)^{2g-2}}{(1-t)} + 2^{2g-1} t^{4g-4}[(1+t)^{2g-2}-(1-t)^{2g-2}].
\end{align}

\end{thm}

The Poincaré polynomial of $\check{\cN}$ was computed by Harder and Narasimhan originally using a purely algebraic method, via the Weil conjectures. Several years later, Atiyah and Bott gave a new computation in terms of the $\cG$-equivariant cohomology of the Harder--Narasimhan strata of the space $\cC_E$ of holomorphic structures on a smooth complex vector bundle $E$ of rank $2$ and degree $1$. The computation of the Poincaré polynomial of $\check{\cM}$ is from Hitchin's paper \cite{Hitchin_SelfDuality}. There, he uses a stratification of the moduli space $\check{\cM}$ induced by a $\U(1)$-action, of which the fixed point subspaces can be explicitly described. We dedicate the rest of this chapter to give a short review of the main arguments behind the computations of $P_t(\check{\cN})$ and $P_t(\check{\cM})$ made by Atiyah--Bott and Hitchin, respectively. As a consequence, we also obtain the proof of the topological mirror symmetry conjecture of Hausel--Thaddeus for this particular case.

\section{Equivariant cohomology and stratified spaces}
In order to explain the computation of the Poincaré polynomial of $\check{\cN}$, we first need to give a short review of equivariant cohomology and the theory of stratifications.

Recall that, associated with any topological group $G$, there exists a universal principal $G$-bundle $EG\rightarrow BG$. If $M$ is a topological space with a $G$-action, then we can consider the space $M_G=(M\times EG)/G$. The \emph{$G$-equivariant cohomology of $M$} is by definition
\begin{equation*}
H_G^*(M) = H^*(M_G).
\end{equation*}
The free action of $G$ on $EG$ induces a fibration $M\rightarrow M_G \rightarrow BG$ so, if $M$ is contractible, we have $H^*_G(M)=H^*(BG)$, which is not trivial in general. Moreover, if $G$ acts freely on $M$, then we also have a fibration $EG\rightarrow M_G \rightarrow M/G$ and, since $EG$ is contractible, $H^*_G(M)=H^*(M/G)$. The \emph{$G$-equivariant Poincaré polynomial} $GP_t(M)$ is defined as the Poincaré series of $H^*_G(M)$.

Consider now a (possibly, infinite-dimensional) manifold $M$. By a \emph{$G$-invariant stratification} of $M$ we mean a set $\left\{ M_i: i\in I \right\}$, indexed by a partially ordered set $I$ with minimal element $0\in I$, of locally closed submanifolds $M_i \subset M$ such that
\begin{equation*}
M = \bigcup_{i\in I} M_i
\end{equation*}
and
\begin{equation*}
\overline{M}_i = \bigcup_{j\geq i} M_j.
\end{equation*}
We also assume that $M_0 \neq \varnothing$, and thus it is the unique open stratum. We also make two extra assumptions
\begin{enumerate}
  \item For every finite subset $I' \subset I$, the set of minimal elements in $I\setminus I'$ is finite and nonempty.
  \item For every $k>0$, the set $\left\{ i \in I: \mathrm{codim}(M_i)<k \right\}$ is finite.
\end{enumerate}

For any subset $J\subset I$, the union of strata $M_J = \bigcup_{j \in J} M_j$ is open if and only if for every $j\in J$, we have that $j'\in J$ for all $j' \leq j$. If $J$ satisfies that property, we say that it is open. If $J$ is open and $i$ is a minimal element of $I\setminus J$, then $J \cup \left\{ i \right\}$ is also open, so the stratum $M_i=M_{J \cup \left\{ i \right\}} \setminus M_J$ is closed. The Thom isomorphism then yields
\begin{equation*}
H^{*-k_i}(M_i) = H^*(M_{J \cup \left\{ i \right\}},M_J),
\end{equation*}
for $k_i=\mathrm{codim}(M_i)$. Hence, we get a long exact sequence
\begin{center}
  \begin{tikzcd}
\cdots  \rar & H_G^{q-k_i}(M_i) \rar & H^q_G(M_{J\cup \left\{i\right\}}) \rar & H^q_G(M_J) \rar & \cdots.
  \end{tikzcd}
\end{center}
The stratification is said to be \emph{perfect} if, for every $q$ and every $i\in I$, the map $H^q_G(M_{J\cup \left\{i\right\}}) \rightarrow H^q_G(M_J)$ is surjective. In that case, the above long exact sequence splits into short exact sequences and we can calculate the $G$-equivariant Poincaré polynomial of $M$ as
\begin{equation*}
GP_t(M) = \sum_{i \in I} t^{k_i} GP_t(M_i).
\end{equation*}

\section{The Harder--Narasimhan stratification} \label{ss:HarderNarasimhan}
Recall from Exercise \ref{ej:stratificationP1} that the space of vector bundles on $\bbP^1$ has a ``weird'' topology. Indeed, even though as a set it is isomorphic to the natural numbers, its topology is far from being discrete. In particular, the point $0$ is dense and, more generally, in this topology the closure of any $n\in \mathbb{N}$ is the set of all numbers $n'\in \mathbb{N}$ with $n' \geq n$. This is precisely an example of the kind of stratifications that we are studying in this chapter. In general, this is a particular case of the Harder--Narasimhan stratification.

Let $\sE$ be a holomorphic vector bundle on $X$, of rank $r$ and degree $d$.
\begin{prop}
 There exists a unique subbundle $\sE_1 \subset \sE$ such that, for every subbundle $\sE'\subset \sE$, we have
  \begin{equation*}
\mu(\sE') \leq \mu(\sE_1),
  \end{equation*}
with equality only if $\rk \sE' \leq \rk \sE_1$. Moreover, this subbundle is semistable.
\end{prop}

\begin{ej}
\textit{Show that there exists an integer $\mu_0$ such that, for every subbundle $\sE'\subset \sE$, we have $\mu(\sE')\leq \mu_0$. Deduce from here the existence of $\sE_1$.} \textbf{Hint}: Start by showing that there cannot exist non trivial maps $\sL\rightarrow \sE$ with $\sL$ is a line bundle with arbitrarily large degree. Indeed, in that case, if we fix a very ample line bundle $\sO_X(1)$ we would have sections of $E(-n)$, for $n$ arbitrarily large.

\textit{Show that $\sE_1$ is unique.} \textbf{Hint}: If there is another $\sE_1'$, we can consider the quotient $\sE/\sE_1'$, and its subbundle $\sF$ generated by the projection of $\sE_1$. Since $\sE_1$ is semistable, we have $\mu(\sF)\geq \mu(\sE_1)$. Reach a contradiction by proving that $\mu(\sF)< \mu(\sE_1)$.
\end{ej}

We call the subbundle $\sE_1 \subset \sE$ the \emph{maximal semistable subbundle} of $\sE$. Iterating this construction, we deduce the \emph{Harder--Narasimhan filtration}
\begin{equation*}
\left\{ 0 \right\} = \sE_0 \subset \sE_1 \subset \sE_2 \subset \dots \subset \sE_{s-1} \subset \sE_s=\sE,
\end{equation*}
defined by letting $\sE_i/\sE_{i-1}$ be the maximal semistable subbundle of $\sE/\sE_{i-1}$. We denote $r_i=\rk(\sE_i/\sE_{i-1})$, $d_i=\deg(\sE_i/\sE_{i-1})$ and $\mu_i=\mu(\sE_i/\sE_{i-1})=d_i/r_i$. Note that $r=\sum_{i=1}^s r_i$ and that $d=\sum_{i=1}^s d_i$.

\begin{ej}
\textit{Show that}
\begin{equation*}
\mu_1 > \mu_2 > \dots > \mu_s.
\end{equation*}
\end{ej}

We can now consider the vector
\begin{equation*}
\bm{\mu}= \bm{\mu}(\sE) = (\mu_1,\overset{(r_1)}{\dots},\mu_1,\mu_2,\overset{(r_2)}{\dots},\mu_2,\dots,\mu_{s-1},\overset{(r_{s-1})}{\dots},\mu_{s-1},\mu_s,\overset{(r_s)}{\dots},\mu_{s}) \in \mathbb{Q}^r.
\end{equation*}
This vector is called the \emph{Harder--Narasimhan type} of the bundle $\sE$. Note that it is a holomorphic invariant of $\sE$, since the Harder--Narasimhan filtration is canonical. We also note that, if $\sE$ is semistable, then $\bm{\mu}=(d/r,\dots,d/r)$.

Consider now the smooth complex vector bundle $E$ underlying $\sE$ and the space $\cC=\cC_E$ of holomorphic structures on $E$. For each $\bm{\mu}\in \mathbb{Q}^r$ (with $\mu_1\geq \mu_2\geq \dots \geq \mu_r$), we can consider the subspace $\cC_{\bm{\mu}}\subset \cC$ of holomorphic structures with Harder--Narasimhan type $\bm{\mu}$. This determines a decomposition
\begin{equation*}
\cC = \bigcup_{\bm{\mu}} \cC_{\bm{\mu}},
\end{equation*}
called the \emph{Harder--Narasimhan stratification}.

To show that the Harder--Narasimhan stratification is indeed an stratification we need to understand how the Harder--Narasimhan filtration degenerates on families. This was understood by Shatz \cite{Shatz}, from an algebraic point of view, and years later Atiyah and Bott \cite{AtiyahBott} gave a differential-geometric description. We consider the following partial order on $\mathbb{Q}^r$
\begin{equation*}
\bm{\mu} \leq \bm{\mu}' \text{ if } \sum_{j=1}^i \mu_j \leq \sum_{j=1}^i \mu'_j, \ i=1,\dots,r-1.
\end{equation*}
Note that $\sum_i \mu_i = d$ is fixed. Shatz's theorem says that
\begin{equation*}
\overline{\cC}_{\bm{\mu}} = \bigcup_{\bm{\mu}' \geq \bm{\mu}} \cC_{\bm{\mu}'}.
\end{equation*}
The minimal $\bm{\mu}$ appearing in the decomposition is $\bm{\mu}_0=(d/r,\dots,d/r)$, so $\cC_{\bm{\mu}_0}=\cC^{ss}$ is the subspace of semistable holomorphic structures. In particular, this subspace is open and dense in $\cC$.

\section{The Poincaré polynomial of $\check{\cN}$}
Going back to the study of $\check{\cN}$, we fix a complex vector bundle $E$ of rank $2$ and degree $1$, and let $\cC=\cC_E$ denote the space of holomorphic structures on $E$. In this case, the possible Harder--Narasimhan types of an element of $\cC$ are of the form $(k+1,-k)$, for $k\in \mathbb{N}$. This gives a stratification $\cC=\bigcup_{k=0}^\infty \cC_k$, for $\cC_k:=\cC_{(k+1,-k)}$. Note that $\cC_0=\cC_{(1,0)}$ is the space of stable holomorphic structures.

The complex gauge group $\cG^\C=\cG^\C_E$ acts on $\cC$ preserving the strata $\cC_k$. Thus we obtain a $\cG^\C$-equivariant stratification. Atiyah and Bott \cite{AtiyahBott} showed that this stratification is, in fact, perfect, and that the codimensions of the strata $\cC_k$ are finite and equal to $2g+4k-4$, so we can compute
\begin{equation*}
\cG^\C P_t(\cC) = \cG^\C P_t(\cC_0) + \sum^\infty_{k=1} t^{2g+4k-4} \cG^\C P_t(\cC_k).
\end{equation*}
In particular, this is telling us that we can obtain the $\cG^\C$-equivariant Poincaré polynomial for $\cC_0$ in terms of the Poincaré polynomials for $\cC$ and for the strata $\cC_k$. Atiyah and Bott proved that the cohomology of each of the strata $\cC_k$, for $k>0$ admits an explicit description, and its $\cG^\C$-equivariant Poincaré polynomial is given by
\begin{equation*}
\cG^\C P_t(\cC_k) = \left( \frac{(1+t)^{2g}}{1-t^2} \right)^2.
\end{equation*}
On the other hand, since $\cC$ is an affine space, it is contractible and $H^*_{\cG^\C}(\cC)=H^*(B\cG^\C)$. Atiyah and Bott also compute explicitly $H^*(B\cG^\C)$ and find that its Poincaré polynomial is equal to
\begin{equation*}
P_t(B\cG^\C) = \frac{[(1+t)(1+t^3)]^{2g}}{(1-t^2)^2(1-t^4)}.
\end{equation*}

The group $\cG^\C$ has the subgroup $\C^*$ of constant central gauge transformations, which acts trivially on $\cC$. The quotient $\overline{\cG}^\C=\cG^\C/\C^*$ acts freely on the stable locus $\cC_0$, and thus we have $\cN=\cC_0/\overline{\cG}^\C$. The classifying space $B\cG^\C$ decomposes as a product $B\cG^\C=B\C^* \times B\overline{\cG}^\C$, so we can write
\begin{equation*}
H_{\cG^\C}^*(\cC_0) = H^*(B\C^*) \otimes H^*_{\overline{\cG}^{\C}}(\cC_0) = H^*(B\C^*) \otimes H^*(\cC_0/\overline{\cG}^{\C}) = H^*(B\C^*) \otimes H^*(\cN).
\end{equation*}
For the Poincaré polynomial, this implies
\begin{equation*}
P_t(\cN) = (1-t^2)\cG^\C P_t(\cC_0).
\end{equation*}
Putting everything together, we get
\begin{equation*}
P_t(\cN) =(1+t)^{2g} \left[ \frac{(1+t^3)^{2g}-t^{2g}(1+t)^{2g}}{(1-t^2)(1-t^4)}\right].
\end{equation*}
Finally, recall that $\cN= (\check{\cN}\times \Pic^1(X))/\Gamma$ and that $P_t(\Pic^1(X))=(1+t)^{2g}$. Atiyah and Bott also proved that the finite group $\Gamma=\Jac(X)[2]$ acts trivially on $H^*(\Pic^1(X))$ and on $H^*(\check{\cN})$. This implies the formula of Harder--Narasimhan
\begin{equation*}
P_t(\check{\cN}) =\left[ \frac{(1+t^3)^{2g}-t^{2g}(1+t)^{2g}}{(1-t^2)(1-t^4)}\right].
\end{equation*}

\section{The Bialynicki--Birula stratification}
Let $M$ be a smooth complex quasi-projective variety equipped with a $\C^*$-action. We say that $M$ is \emph{semi-projective} if the fixed point locus $M^{\C^*}$ is projective and, for every $m\in M$, there exists some $m_0\in M^{\C^*}$ such that $\lim_{\lambda \rightarrow 0} \lambda \cdot m = m_0$. If $M$ is a semi-projective variety, we can associate with it a Bialynicki--Birula stratification. We start by defining the upward flow: for each point $m_0 \in M^{\C^*}$, we define its \emph{upward flow} as the subspace
\begin{equation*}
U_{m_0} = \left\{ m\in M: \lim_{\lambda \rightarrow 0} \lambda \cdot m =m_0 \right\}.
\end{equation*}
The \emph{Bialynicki--Birula partition} is the decomposition $M=\bigcup_{m_0 \in M^{\C^*}} U_{m_0}$, which indeed covers the whole $M$ since it is assumed to be quasi-projective.
For each connected component $F\subset M^{\C^*}$ of the fixed points, we define its \emph{attractor} as the subset
\begin{equation*}
U_F = \left\{ m \in M: \lim_{\lambda \rightarrow 0} \lambda \cdot m \in F \right\}= \bigcup_{m_0 \in F} U_{m_0}.
\end{equation*}
The \emph{Bialynicki--Birula stratification} is the decomposition
\begin{equation*}
M = \bigcup_{F \in \pi_0(M^{\C})} U_F.
\end{equation*}

The dimension of an upward flow $U_{m_0}$ can be computed as follows. The $\C^*$-action on $M$ determines a linear $\C^*$-action on the tangent space $T_{m_0} M$, and in turn a weight decomposition $T_{m_0}M = \bigoplus_{k\in \Z} (T_{m_0}M)_k$. We let $T_{m_0}^+M = \bigoplus_{k> 0} (T_{m_0}M)_k$. Note that a vector $v\in T_{m_0}M$ is tangent to $U_{m_0}$ if and only if it has positive weight, so we can identify $T_{m_0}U_{m_0}\cong T_{m_0}^+M$. The \emph{index} of the point $m_0$ is defined as the real dimension $i(m_0)=\dim_{\bbR}U_{m_0}=\dim_{\bbR} T^+_{m_0}M$. Note that the index is continuous as a map from $M$ to $\bbN$, so for a connected component $F\subset M^{\C^*}$, the index $i(F)=i(m_0)$, for any $m_0\in F$, is well defined, and we have
\begin{equation*}
\dim_{\bbR} U_F = \dim_{\bbR} F + i(m_0).
\end{equation*}

Suppose moreover that $M$ is endowed with a holomorphic symplectic form $\Omega$ such that $\lambda^* \Omega = \lambda \Omega$. Such a symplectic form identifies the component of $T_{m_0}M$ with weight $k$ with the one with weight $-k$, and thus we have
\begin{equation*}
i(m_0) = \tfrac{1}{2} \dim_{\bbR} M = \dim_{\C} M,
\end{equation*}
so
\begin{equation*}
\dim_{\bbR} U_F = \dim_{\bbR} F + \dim_{\C} M.
\end{equation*}
Moreover, it is not hard to show that in fact the symplectic form $\Omega$ vanishes at the subspace $T_{m_0}^+M$, so in fact the upward flows $U_{m_0}$ are (holomorphic) \emph{lagrangian submanifolds} of $M$.

\section{The $\C^*$-action on $\check{\cM}$}
The space $\check{\cM}$ comes equipped with a $\C^*$-action, defined as
\begin{align*}
  \C^* \times \check{\cM} & \longrightarrow \check{\cM} \\
(\lambda,(\sE,\varphi))    & \longmapsto (\sE,\lambda \varphi).
\end{align*}
The limit $\lim_{\lambda \rightarrow 0} (\sE,\lambda \varphi)$ always exists, so $\check{\cM}$ is in fact semi-projective.

\begin{ej}
Consider the holomorphic symplectic form $\Omega_1=\omega_2 + i\omega_3\in \Omega^{2,0}(\check{\cM})$. \textit{Show that $\lambda^* \Omega_1 = \lambda \Omega_1$.}
\end{ej}

\begin{thm}[Hitchin]
The space of fixed points $\check{\cM}^{\C^*}$ decomposes as a finite union of connected components
\begin{equation*}
\check{\cM}^{\C^*} = \bigcup_{k=0}^{g-1} F_k.
\end{equation*}
The component $F_0$ is the subspace of equivalence classes of pairs $(\sE,\varphi)$ with $\varphi=0$, and thus we can identify $F_0 \cong \check{\cN}$. Denoting $\bar{k}=2g-2k-1$, the component $F_k$ for $k>0$ is a $2^{2g}$-cover of the symmetric product
\begin{equation*}
S^{\bar{k}} X = (X\times \overset{(\bar{k})}{\cdots} \times X)/\mathfrak{S}_{\bar{k}},
\end{equation*}
where $\mathfrak{S}_{\bar{k}}$ is the symmetric group, with Galois group \[\mathrm{Gal}(F_k/S^{\bar{k}}X)\cong \Gamma (\cong (\Z_2)^{2g}).\]
\end{thm}

\begin{proof}
  Clearly $(\sE,0)$ is a fixed point of the $\C^*$-action. Moreover, stability for the pair $(\sE,0)$ amounts to stability for $\sE$, so $F_0\cong \check{\cN}$. Now, if $\varphi \neq 0$, then $(\sE,\varphi)$ is a fixed point of the $\C^*$-action if there is an induced $\C^*$-action $\lambda \mapsto (f_\lambda: \sE \rightarrow \sE)$ with $(\lambda \varphi) \circ f_\lambda = (f_\lambda \otimes \id_{\bm{\Omega}^1_X}) \circ \varphi$. The $\C^*$-action on $\sE$ induces an splitting
  \begin{equation*}
\sE = \sL \oplus \sL^{-1} \xi,
  \end{equation*}
  where $\sL$ is a holomorphic line bundle of degree $k=\deg \sL$ and we recall that $\xi=\det \sE$. Compatibility of the action with $\varphi$ implies that $\varphi$ can be written in lower triangular form as
  \begin{equation*}
    \varphi =
    \begin{pmatrix}
      0 & 0 \\
      \phi & 0
    \end{pmatrix},
  \end{equation*}
  for $\phi \in H^0(X,L^{-2}\bm{\Omega}^1_X \xi)$. For $\phi$ to be non-trivial we need that
  \begin{equation*}
0 \leq \deg (L^{-2}\bm{\Omega}^1_X \xi) =-2k + 2g-2 + 1= \bar{k},
  \end{equation*}
  so we conclude that $k \leq g -1$. We can thus identify $F_k$ as the moduli space of pairs
  \begin{equation*}
(\sL,\phi: \sL \rightarrow \sL^{-1} \bm{\Omega}^1_X \xi)
  \end{equation*}
  with $\deg \sL = k$. Consider now the map $\Pic^{k}(X)\rightarrow \Pic^{\bar{k}}(X)$, $\sL \mapsto \sL^{-2} \bm{\Omega}^1_X \xi$ and the Abel--Jacobi map
  \begin{align*}
    S^{\bar{k}} X & \longrightarrow  \Pic^{\bar{k}}(X) \\
    D=\sum_{i=1}^{\bar{k}} x_i  & \longmapsto \sO_X(D).
  \end{align*}
The space $F_k$ can then be identified as the fibered product
\begin{equation*}
F_k = S^{\bar{k}} X \times_{\Pic^{\bar{k}}(X)} \Pic^k(X).
\end{equation*}
The natural projection $F_k \rightarrow S^{\bar{k}} X$ is indeed a $2^{2g}$-cover, induced by the $\Gamma$-action on $F_k$.
\end{proof}

Recall that the $\C^*$-action on $\check{\cM}$ induces the Bialynicki--Birula stratification
\begin{equation*}
\check{\cM} = \bigcup_{k=0}^{g-1} U_k,
\end{equation*}
where the strata $U_k$ are the \emph{attractors}
\begin{equation*}
U_k := U_{F_k}= \left\{ (\sE,\varphi) \in \check{\cM}: \lim_{\lambda \rightarrow 0} (\sE,\lambda \varphi) \in F_k \right\}.
\end{equation*}

\begin{rmk}
Equivalently, Hitchin obtained the stratification above by a differential-geometric method, using the Morse function
\begin{equation*}
f(\sE,\varphi) = 2i \int_X \tr(\varphi \varphi^{\dagger_H}),
\end{equation*}
for $H$ a HEH metric on $(\sE,\varphi)$.
\end{rmk}

\section{The Poincaré polynomial of $\check{\cM}$}
We can use the Bialynicki--Birula stratification to compute the Poincaré polynomial of $\check{\cM}$. Indeed, Hitchin \cite{Hitchin_SelfDuality} showed that this stratification is perfect, so we can compute
\begin{equation*}
P_t(\check{\cM}) = \sum_{i=0}^{g-1} t^{\mathrm{codim}_{\bbR} U_k} P_t(U_k) = P_t(U_0) + \sum_{i=1}^{g-1} t^{\mathrm{codim}_{\bbR} U_k} P_t(U_k).
\end{equation*}
Note that $U_0\subset \check{\cM}$ is open, so it has codimension $0$.

First, we compute the real codimensions of the $U_k$, for $k>0$,
\begin{equation*}
\mathrm{codim}_{\bbR} U_k = \dim_{\C} \check{\cM} - \dim_{\bbR} F_k = 6g - 6 - \dim_{\bbR} S^{\bar{k}} X = 6g - 6 - 2\bar{k} = 2(g + 2k - 2) .
\end{equation*}
Therefore, the Poincaré polynomial of $\check{\cM}$ is
\begin{equation*}
P_t(\check{\cM}) = P_t(U_0) + \sum_{i=1}^{g-1} t^{2(g+2k-2)} P_t(U_k).
\end{equation*}
Now, since each $F_k$ is a deformation retract of $U_k$, we have $P_t(U_k)=P_t(F_k)$, so we can compute
\begin{equation} \label{eq:poincM}
P_t(\check{\cM}) = P_t(\check{\cN}) + \sum_{i=1}^{g-1} t^{2(g+2k-2)} P_t(F_k).
\end{equation}

The missing pieces are the Poincaré polynomials of the $F_k$, for $k>0$. Recall that $F_k$ is a Galois $\Gamma$-covering of $S^{\bar{k}}X$. Therefore, there is a $\Gamma$-action on the cohomology ring $H^*(F_k,\C)$, inducing a decomposition
\begin{equation*}
H^*(F_k, \C) = H^*(F_k,\C)^\Gamma \oplus \bigoplus_{\kappa \in \hat{\Gamma}\setminus \left\{ 0 \right\}} H^*(F_k,\C)_{\kappa}.
\end{equation*}
Where $H^*(F_k,\C)^\Gamma$ is the invariant part, $\hat{\Gamma}=\Hom(\Gamma,\C^*)$ is the group of characters of $\Gamma$ and $H^*(F_k,\C)_{\kappa}$ is the corresponding isotypic component. The invariant part is
\begin{equation*}
H^*(F_k,\C)^\Gamma = H^*(S^{\bar{k}}X,\C).
\end{equation*}
Now, the cohomology of the symmetric product $S^{\bar{k}}X$ was studied by Macdonald \cite{Macdonald}, who in particular showed that $P_t(S^{\bar{k}}X)$ is the coefficient in $x^{\bar{k}}$ of the expression
\begin{equation} \label{eq:macdonald}
\frac{(1+xt)^{2g}}{(1-x)(1-xt^2)}.
\end{equation}
On the other hand, the isotypic component $H^*(F_k,\C)_{\kappa}$ coincides with the cohomology of $S^{\bar{k}}X$ with coefficients in the local system
\begin{equation*}
\C_{F_k,\kappa} = F_k \times^{\kappa:\Gamma \rightarrow \C^*} \C := F_k \times \C /\left\{ (p,z) \sim (\gamma\cdot p, \kappa(\gamma)^{-1} z) \right\} .
\end{equation*}
We write
\begin{equation*}
H^*(F_k,\C)_\kappa = H^*(S^{\bar{k}}X,\C_{F_k,\kappa}).
\end{equation*}

Recall now the isomorphism $w:\hat{\Gamma}\rightarrow \Gamma$ induced by the Weil pairing, and put $\gamma=w(\kappa)$. The element $\gamma$ determines a line bundle $\sL_\gamma\rightarrow X$ of order $2$. We denote by $\C_\gamma$ the sheaf of locally constant sections of $\sL_\gamma$, which is a local system of rank $1$. From this local system, we can define a local system $\C_{\gamma,\bar{k}}$ of rank $1$ on $X^{\bar{k}}$ by putting \[\C_{\gamma,\bar{k}}=\pr_1^* \C_\gamma \otimes \pr_2^* \C_\gamma \otimes\cdots \otimes \pr_{\bar{k}}^* \C_\gamma.\]
We have a natural action of the symmetric group $\mathfrak{S}_n$ on the cohomology ring $H^*(X^{\bar{k}},\C_{\gamma,\bar{k}})$, and the cohomology $H^*(F_k,\C)_\kappa$ can be identified with the $\mathfrak{S}_n$-invariant part.

If $\C_\gamma$ is trivial, then $H^*(X^{\bar{k}},\C_{\gamma,\bar{k}})^{\mathfrak{S}_n}=H^*(X^{\bar{k}},\C)^{\mathfrak{S}_n}=H^*(S^{\bar{k}}X,\C)$, that was already considered. Now, if $\C_\gamma$ is not trivial, then it cannot have global sections, so $H^0(X,\C_\gamma)=0$ and, by Poincaré duality $H^2(X,\C_\gamma)=0$. For $H^1$, the Hodge decomposition gives an isomorphism
\begin{equation*}
H^1(X,\C_\gamma) \cong H^1(X,\sL_\gamma) \oplus H^0(X,\sL_\gamma \bm{\Omega}^1_X).
\end{equation*}
Hence, since $h^0(X,\sL_\gamma)=0$, using Serre duality and the Riemann--Roch theorem we get
\begin{equation*}
\dim_\C H^1(X,\C_\gamma) = h^1(X,\sL_\gamma) + h^0(X,\sL_\gamma\bm{\Omega}^1_X) = 2 h^1(X,\sL_\gamma) = 2g-2.
\end{equation*}

Consider now the fibration $X\rightarrow X^{\bar{k}} \rightarrow X^{\bar{k}-1}$. Using Mayer--Vietoris, since only $H^1(X,\C_\gamma)$ is non-zero, we get an isomorphism
\begin{equation*}
H^*(X^{\bar{k}},\C_{\gamma,\bar{k}}) \cong H^1(X,\C_\gamma) \otimes H^{*-1}(X^{\bar{k}-1},\C_\gamma) \cong \dots \cong H^1(X,\C_\gamma)^{\otimes \bar{k}}.
\end{equation*}
By construction, the symmetric part $H^*(X^{\bar{k}},\C_{\gamma,\bar{k}})$ is the alternating part of $H^1(X,\C_\gamma)^{\otimes \bar{k}}$, so
\begin{equation*}
H^*(F_k,\C)_\kappa= H^*(S^{\bar{k}}X,\C_{F_k,\kappa}) =  H^*(X^{\bar{k}},\C_{\gamma,\bar{k}})^{\mathfrak{S}} = \wedge^{\bar{k}} H^1(X,\C_\gamma).
\end{equation*}
We conclude that the Poincaré polynomial of $F_k$ is
\begin{equation} \label{eq:poincFk}
P_t(F_k) = P_t(S^{\bar{k}} X) + (2^{2g}-1) {2g-2 \choose \bar{k}}.
\end{equation}

\begin{ej}
Derive Hitchin's formula \eqref{eq:poincMthm} from the formulas \eqref{eq:poincM} and \eqref{eq:poincFk}, using Macdonald's formula \eqref{eq:macdonald} and the formula \eqref{eq:poincN} of Harder--Narasimhan.
\end{ej}

\section{Topological mirror symmetry in rank $2$} \label{ss:topologicalmirrorproof}
It turns out that the unraveled topological mirror symmetry formula for $r=2$ and $d=1$, for $\kappa\in \hat{\Gamma}$ and $\gamma=w(\kappa)$,
\begin{equation}
E_\kappa(\check{\cM};u,v) = E(\check{\cM}^\gamma/\Gamma,L_{\hat{\beta},\gamma};u,v)(uv)^{F(\gamma)},
\end{equation}
follows almost immediately from the calculations of previous section.

For the left hand side term, we observe from the decomposition
\begin{equation*}
H^*_c(\check{\cM}) = H^*(\check{\cN}) \oplus \bigoplus_{k=1}^{g-1} H_c^{*+ 2(g+2k-2)}(F_k)
\end{equation*}
and from the fact that $\Gamma$ acts trivially on $H^*(\check{\cN})$, that
\begin{equation*}
H^*_c(\check{\cM})_\kappa = \bigoplus_{k=1}^{g-1} H_c^{*+ 2(g+2k-2)}(F_k)_\kappa.
\end{equation*}
Recall that
\begin{equation*}
H_c^*(F_k)= H_c^{\bar{k}}(F_k) = \wedge^{\bar{k}} H^1(X,\C_\gamma)
\end{equation*}
for $\gamma=w(\kappa)$ and $\bar{k}=2g-2k-1$. Note that
\begin{equation*}
\bar{k} + g+2k-2 = 3g-3.
\end{equation*}
Therefore,
\begin{equation*}
E_\kappa(\check{\cM};u,v) = \sum_{k=1}^{g-1} (uv)^{3g-3} \sum_{p+q=\bar{k}} h^{p,q}(\wedge^{\bar{k}} H^1(X,\C_\gamma))u^p v^q.
\end{equation*}
It follows from our computations in last section that the cohomology group $H^1(X,\C_\gamma)$ has Hodge type $(g-1,g-1)$ and thus $h^{p,q}(\wedge^{\bar{k}} H^1(X,\C_\gamma))={g-1 \choose p}{g-1 \choose q}$. Therefore,
\begin{align*}
  E_\kappa(\check{\cM};u,v) &= (uv)^{3g-3} \sum_{\bar{k}=1,\text{odd}}^{2g-2} \ \sum_{p+q=\bar{k}} {g-1 \choose p}{g-1 \choose q} u^p v^q \\
  &= \tfrac{1}{2}(uv)^{3g-3}[(1-u)^{g-1}(1-v)^{g-1} - (1+u)^{g-1}(1+u)^{g-1}].
\end{align*}

We consider now the right hand side. A non-trivial element $\gamma \in \Gamma$ determines an unramified Galois $\Z_2$-cover $\pi_\gamma:X_\gamma \rightarrow X$, and we have a commutative diagram
\begin{center}
  \begin{tikzcd}
   \bm{T}^* \Pic^1(X_\gamma) \ar{r}{(\pi_\gamma)_*} \ar{d} & \cM \ar{d}{\det} \supset \check{\cM} \\
   \bm{T}^* \Pic^1(X_\gamma) \ar{r}{\Nm_{\pi_\gamma}} & \bm{T}^* \Pic^1(X) \ni (\xi,0) .
  \end{tikzcd}
\end{center}
From this diagram, it is not hard to see\footnote{This is a well known result of Narasimhan and Ramanan \cite{NarasimhanRamanan}, the reader can also consult \cite{OscarRamanan}.} that $\check{\cM}^\gamma$ is a torsor under $\bm{T}^* P_\gamma$, where $P_\gamma:=P(\pi_\gamma:X_\gamma\rightarrow X)$ is the Prym variety associated with the cover $\pi_\gamma$. Therefore,
\begin{equation*}
E(\check{\cM}^\gamma/\Gamma,L_{\hat{\beta},\gamma};u,v) = (uv)^{g-1} E(P_\gamma/\Gamma,L_{\hat{\beta},\gamma};u,v).
\end{equation*}
The Hodge numbers of the $(g-1)$-dimensional abelian variety $P_\gamma$ are easy to compute: we have
\begin{equation*}
h^{p,q}(P_\gamma) = {g-1 \choose p}{g-1 \choose q}.
\end{equation*}
Hence,
\begin{equation*}
E(P_\gamma;u,v) = \sum_{k=0}^{g-1} \sum_{p+q = k} {g-1 \choose p}{g-1 \choose q} u^p v^q = (1+u)^{g-1}(1+v)^{g-1}.
\end{equation*}
If we take coefficients on the local system $L_{\hat{\beta},\gamma}$ over $P_\gamma/\Gamma$ we just get the averaging
\begin{align*}
  E(P_\gamma/\Gamma,L_{\hat{\beta},\gamma};u,v) &= \frac{1}{2^{2g}} \sum_{\gamma' \in \Gamma}w(\gamma,\gamma') (1+ w(\gamma,\gamma')u)^{g-1}(1+ w(\gamma,\gamma')v)^{g-1} \\
  & = \tfrac{1}{2}[(1-u)^{g-1}(1-v)^{g-1} - (1+u)^{g-1}(1+u)^{g-1}].
\end{align*}
To prove the equality, we just need to show that the fermionic shift $F(\gamma)$ is equal to $2g-2$. But, indeed, since the $\gamma$-action respects the holomorphic symplectic structure, the fermionic shift is just half the complex codimension of $\bm{T}^* P_\gamma$ inside of $\check{\cM}$, which is equal to
\begin{equation*}
  \tfrac{1}{2}[6g-6 - 2(g-1)] = 2g-2.
\end{equation*}

\begin{rmk}
A surprising feature of topological mirror symmetry is the fact that the left hand side accounts for the contributions of the non-trivial components of the fixed points of the $\C^*$-action (since $\Gamma$ acts trivially on the cohomology of $\check{\cN}$). On the other hand, $\Gamma$ acts freely on these non-trivial components, so the contributions of the right hand side come essentially from the fixed points $\check{\cN}^\gamma \subset \check{\cN}$. In terms of representation theory, the fixed-points spaces $\check{\cM}^\gamma$ were interpreted by Ngô \cite{Ngo} as related to the \emph{endoscopic subgroups} of $\SL_2$.  In his monumental work, Ngô developed a very profound understanding of the cohomology of the moduli spaces of $G$-Higgs bundles, for general reductive $G$, and also over fields of positive characteristic. In particular, Ngô found some formulas from which, by taking trace of the Frobenius automorphism (when considered in positive characteristic) he managed to deduce the Fundamental Lemma of Langlands and Shelstad, which is one of the crucial pieces of the (classical) Langlands program.
\end{rmk}

\nocite{*}
\printbibliography
\end{document}